%% file: MoM_draft_arxiv.tex
\documentclass[letterpaper,11pt]{article}
\usepackage[margin=1in]{geometry}

\input{./preamble.tex}

\usepackage{bbold}
\usepackage[mathcal]{euscript}
\usepackage{bibunits}

\def\bV{\mathbf{V}}

\def\bM{\mathbf{M}}
\def\bX{\mathbf{X}}
\def\bx{\mathbf{x}}

\def\be{\mathbf{e}}
\def\bbeta{\bm{\beta}}
\def\trace{\mathsf{trace}}

\def\bv{\mathbf{v}}
\def\bu{\mathbf{u}}

\def\bI{\mathbf{I}}

\def\bSigma{\mathbf{\Sigma}}
\def\bOmega{\mathbf{\Omega}}

\def\balpha{\bm{\alpha}}

\def\div{\mathsf{div}}

\def\expit{\mathsf{expit}}

\def\MLE{\mathsf{MLE}}
\def\MAR{\mathsf{MAR}}
\def\GLM{\mathsf{GLM}}

\def\GCM{\mathsf{GCM}}
\def\CE{\mathsf{CE}}
\def\Wishart{\mathsf{Wishart}}

\newcommand{\unknowna}[1]{\textcolor{black}{#1}}
\newcommand{\unknownb}[1]{\textcolor{black}{#1}}
\newcommand{\unknownc}[1]{\textcolor{black}{#1}}
\newcommand{\unknownd}[1]{\textcolor{black}{#1}}

\newtheorem{innercustomthm}{Assumption}
\newenvironment{customas}[1]
{\renewcommand\theinnercustomthm{#1}\innercustomthm}
{\endinnercustomthm}

\begin{document}

\title{Method-of-Moments Inference for GLMs and Doubly Robust Functionals under Proportional Asymptotics}

\author[1]{Xingyu Chen \thanks{E-mail: \texttt{xingyuchen0714@sjtu.edu.cn}}}
\author[1, 2]{Lin Liu \thanks{E-mail: \texttt{linliu@sjtu.edu.cn}}}
\author[3]{Rajarshi Mukherjee \thanks{E-mail: \texttt{ram521@mail.harvard.edu} R. Mukherjee is supported by NSF CAREER 8529216-01.}}

\affil[1]{School of Mathematical Sciences, CMA-Shanghai, Shanghai Jiao Tong University}
\affil[2]{Institute of Natural Sciences, MOE-LSC, SJTU-Yale Joint Center for Biostatistics and Data Science, Shanghai Jiao Tong University}
\affil[3]{Department of Biostatistics, Harvard T. H. Chan School of Public Health}
    
\date{\today}
    
\maketitle
\begin{abstract}
In this paper, we consider the estimation of regression coefficients and signal-to-noise (SNR) ratio in high-dimensional Generalized Linear Models (GLMs), and explore their implications in inferring popular estimands such as average treatment effects in high dimensional observational studies. Under the ``proportional asymptotic'' regime and Gaussian covariates with known (population) covariance $\bSigma$, we derive $\sqrt{n}$-Consistent and Asymptotically Normal (CAN) estimators of our targets of inference through a Method-of-Moments type of estimators that bypasses estimation of high dimensional nuisance functions and hyperparameter tuning altogether. Additionally, under non-Gaussian covariates, we demonstrate universality of our results under certain additional assumptions on the regression coefficients and $\bSigma$. We also demonstrate that knowing $\bSigma$ can be relaxed in our proposed methodology. Finally, we complement our theoretical results with extensive numerical experiments, in comparisons with competing methods.
\end{abstract}
    
    
\allowdisplaybreaks

\begin{bibunit}[plainnat]

\section{Introduction}
\label{sec:intro}

Statistical inference in Generalized Linear Models (GLMs) \citep{nelder1972generalized, mccullagh1989generalized}, although a classical topic in statistics, has witnessed renewed enthusiasm in the modern big data era spurred by both theoretical and computational challenges that arise therein \citep{jankova2018semiparametric, cai2023statistical, sur2019modernb, sur2019likelihood, candes2020phase, zhao2022asymptotic}. This line of research has in turn found resonance in the challenges encountered in the context of inference in observational studies \citep{chernozhukov2018double, athey2018approximate, jiang2025new, yadlowsky2022explaining, celentano2023challenges}. Specifically, estimation of quantities like the causal effect of an exposure on an outcome or estimation of population quantities under missing data typically relies on understanding nuisance functions such as outcome regression and propensity scores \citep{robins1994estimation, scharfstein1999adjusting}. These regression functions are often modeled as suitable GLMs, when one needs to adjust for confounders possibly larger in dimension than the available sample size. There now exists a dedicated and comprehensive methodology to deal with inference in both GLMs or observational studies with high dimensional covariates/confounders focused on ideas based on semiparametric theory \citep{zhang2014confidence, javanmard2014confidence, jankova2018semiparametric, athey2018approximate, smucler2019unifying, bradic2019minimax, bradic2019sparsity, tan2020model, dukes2021inference, wang2024debiased, liu2023root, su2023estimated}. Indeed, this class of methods, in turn, relies on rates of convergence for consistent estimators of high dimensional GLM parameters \citep{negahban2012unified}. However, even the mere existence of such a consistent estimator relies on \textit{a priori} unknown low-dimensional (such as sparsity) assumptions in respective GLMs \citep{verzelen2012minimax, collier2017minimax, cai2017confidence, bellec2022biasing}. 

To complement the above framework, recent times have witnessed a parallel focus to deal with cases when the entire GLM parameter vector cannot be estimated consistently, and yet there are potential low dimensional summaries of them that can yield themselves to desirable inferential strategies.  As a byproduct, one can possibly provide reliable estimation in observational studies. One specific instance that has become popular is when the GLM parameters grow proportionally to the sample size in dimension and do not satisfy additional low-dimensional assumptions \citep{bean2013optimal, el2013asymptotic, donoho2016high, lei2018asymptotics}\footnote{In the econometrics literature, similar problems have also been studied in (partially) linear models \citep{cattaneo2018inference, cattaneo2019two} under the name ``many-regressor asymptotics'' or in settings with many weak instrumental variables \citep{newey2009generalized, mikusheva2022inference} under the name ``many-instrument/many weak IV asymptotics''.}. To reflect the inherent difficulty of this setup in terms of the information-theoretic impossibility of estimating the GLM parameters consistently \citep{verzelen2012minimax, barbier2019optimal}, recent research has coined it as the ``inconsistency regime'' \citep{celentano2023challenges}, and fundamental ideas have already started to carve the contours of this paradigm. A major theme of research in this regime often pertains to initial progress made under Gaussian covariates with known covariance \citep{bellec2023debiasing} and subsequent demonstration of a universality principle \citep{zhao2022asymptotic, dudeja2023universality, han2023universality}. Indeed, the Gaussian assumption is not necessarily a simplifying assumption in the development and analysis of the methods in this literature -- one requires highly involved probabilistic machinery to produce sharp analyses of the derived estimators. Nevertheless, the assumption of Gaussianity, coupled with the knowledge or a sufficiently accurate estimator of the population (variance-)covariance matrix of baseline covariates inject enough structure and \textit{a priori} information to simplify the process. Such structure and information enable one to bypass complicated estimators and analyses, while still achieving remarkably parallel guarantees for inference in GLMs and observational studies in the proportional asymptotic high-dimensional regime. The primary aim of this article is to take steps in that direction. Specifically, we demonstrate that for GLMs with link function meeting certain conditions (see Assumption~\ref{as:link} later), it is possible to construct a diffeomorphism between functionals of the GLM parameters and carefully crafted low-degree moments of the data, for which $\sqrt{n}$-Consistent and Asymptotic Normal (CAN) estimators exist. This crucial observation forms the core of our proposed methodology.

\subsection{Results Highlight}

We summarize the main results of the paper below:
\begin{enumerate}[label = (\arabic*)]
\item We propose moments-based identification strategies (for the precise meaning of identification, we refer readers to Lemma~\ref{lem:glm mean zero} later in the paper) for statistical functionals with nuisance models parameterized as high-dimensional GLMs with the dimension $p$ proportional to $n$ when the covariance matrix of the covariates are known. This allows the construction of estimators of relevant low-dimensional summaries of high-dimensional GLM parameters such as contrasts and Signal-to-Noise Ratio (SNR). Moreover, our methods being reliant on only a few low-dimensional moments of the data are computationally efficient. 

\item Our moment-based identification and estimation strategies generalize to parallel inferential techniques for popular objects of interest in observational studies such as average treatment effects and mean estimands under missing data -- where the analyses depend on two nuisance functions modeled by high dimensional GLMs. Compared to the literature in this class of problems, we do not require sample-splitting and cross-fitting-based ideas owing to our ability to avoid estimating nuisance functions.

\item Our estimators completely bypass the estimation of high dimensional nuisance parameters and are CAN when the baseline covariates are Gaussian under some additional regularity conditions. We further demonstrate the universality of the proposal beyond Gaussian designs in terms of rates of convergence.

\item We also demonstrate that the assumption of knowing the population covariance matrix $\bSigma$ of the design can be dropped for our proposal when the sample covariance matrix estimator of $\bSigma$ is invertible and $p < c \cdot n$ for some constant $c$, under Gaussian designs. 

\item We conduct extensive numerical experiments to verify the validity of our proposals in finite sample, as well as comparing them with methods from the  emerging recent literature. Readers can access the codes for replicating our numerical results from \href{https://github.com/cxy0714/Method-of-Moments-Inference-for-GLMs}{the accompanied GitHub repository}.
\end{enumerate}

\subsection{Related Works}
\label{sec:review}
Our research draws inspiration from several past and ongoing research that aims to address inference in high-dimensional problems. To present a compact survey and comparison with the most related members of this literature we divide our discussions across three broad themes: inference in GLMs, inference for popular observational studies, and the knowledge of variance-covariance matrix of baseline covariates. In each of these sub-parts, we shall further briefly touch upon both ultra-high dimensional regimes under sparsity and proportional high dimensional regimes without sparsity aspects of the results in literature.

\subsubsection{Inference in GLMs} 
In the last two decades, statistical inference for linear and quadratic forms of high dimensional GLM parameters has attracted significant attention from the statistical research community \citep{verzelen2012minimax, zhang2014confidence, javanmard2014confidence, dicker2014variance, verzelen2018adaptive, cai2018accuracy, collier2017minimax, guo2022moderate, battey2023inference, celentano2024correlation}. Two complementary tracks of emphasis have emerged in this regard. In the first line of activities, the strategy of inference often draws inspiration from classical semiparametric theory \citep{jankova2018semiparametric} and requires the consistent estimation of ultra high-dimensional (when the dimension $p$ is \textit{much larger} than the sample size $n$) GLM parameters -- which need to rely on apriori low-dimensional assumptions, such as sparsity, on GLM parameter vectors \citep{collier2017minimax, cai2018accuracy, cai2023statistical}. To explore regimes where the existence of consistent estimators of entire GLM parameter vectors are impossible, a complementary theme of inference in GLMs has sprung in the last decade under proportional asymptotics (when the dimension $p$ is \textit{proportional} to the sample size $n$) \citep{sur2019modernb, sur2019likelihood, candes2020phase, zhao2022asymptotic}. In this regime, the strategy typically involves a careful debiasing surgery on initial suitable yet inconsistent GLM parameter vectors to yield sophisticated CAN estimators of linear and quadratic forms. Indeed, literature in this second direction is more recent and had initially focused on linear models in terms of (i) characterizing the precise risk behavior of convex regularized procedures -- first for Gaussian covariates (see \citet{bayati2011lasso, stojnic2013framework, thrampoulidis2018precise, miolane2021distribution, celentano2023lasso} and references therein) and then beyond Gaussian \citep{gerbelot2020asymptotic, gerbelot2022asymptotic, li2023spectrum, han2023universality}; and (ii) inference of linear and quadratic forms of the parameter vector -- albeit mainly in the regime where the design covariance is known apriori \citep{bellec2022biasing, bellec2025observable, bellec2023debiasing, song2024hede}. Results for GLMs are more complete in terms of estimation of the whole parameter vector using convex regularized methods. Parallel methods in GLMs for inference on linear and quadratic forms are more recent, quite case-specific (e.g. consider binary regression with logistic and/or probit link), do not always cover the whole proportional regime (i.e. all aspects ratio considerations of $p/n$) without further assumptions, and often yield coverage guarantees in an average sense on individual coordinates of the GLM parameter vector instead of individually across coordinates \citep{bellec2025observable}\footnote{It is noteworthy that \citet{bellec2025observable} additionally considered Single-Index Models (SIMs) with an unknown link function. We further discuss possible extensions of our work from GLMs to SIMs in Section~\ref{sec:discussion}.}. Another work related to ours is \citet{sawaya2023moment}, which also concerns statistical inference for GLM parameters. In particular, under assumptions (1) the link function having certain asymmetry (see Section A.8 of \citet{sawaya2023moment} for a precise statement) and (2) the covariates $\bX$ having zero mean, \citet{sawaya2023moment} use only moments of $Y$ to estimate certain quantities in the State Evolution system, that characterizes the asymptotic behavior of the maximum likelihood and its convex regularized analog, to conduct inference -- thus obviating the requirement of knowing the population covariance matrix $\bSigma$ of $\bX$ or estimating $\bSigma$ with sufficiently fast convergence rate. However, this important advantage is at the expense of precluding important GLMs such as the logistic or probit regression. Finally, the theoretical results of \cite{sawaya2023moment} rely on assuming the existence and suitable boundedness of estimators based on minimizing possibly regularized GLM loss functions as well as the existence of unique positive solutions to relevant state evolution equations -- which needs to be further verified and rested outside the scope of the work \cite{sawaya2023moment}. Since we bypass the estimation of the entire parameter vector while performing CAN estimation of low-dimensional summaries of them, our results do not rely on such further assumptions.

\subsubsection{Inference in Observational Studies} 

Quantities like average treatment effects and mean parameters in missing data problems have now emerged as quintessential examples of functionals in observational studies where the challenges of high dimensional baseline covariates require careful methodological consideration. Similar to the literature in GLM, two complementary themes have emerged here as well -- one regarding ultra-high-dimensional regimes under sparsity and another regarding proportional asymptotic regimes without sparsity but under known Gaussian covariate designs \citep{celentano2023challenges} or for specific functionals with $p < n$ \citep{yadlowsky2022explaining, jiang2025new}. Since the ultra-high-dimensional regime under sparsity has been heavily studied \citep{athey2018approximate, smucler2019unifying, bradic2019minimax, bradic2019sparsity, tan2020model, tan2020regularized, dukes2021inference, wang2024debiased, liu2023root}, the results therein are somewhat complete in terms of necessary and sufficient conditions for CAN estimation. However, without Gaussian covariates or the assumption that $p<n$, neither systematic methods nor CAN guarantees exist for the above examples.  Our methods aim to fill this gap in the literature. Finally, we remark that our proposed estimators involve second-order $U$-statistics, thus also drawing connections to the growing literature on using Higher-Order $U$-statistics in semiparametric problems in observational studies \citep{robins2008higher, van2021higher, kennedy2024minimax, bonvini2024doubly, breunig2019simple, breunig2024adaptive}. Also see Remark~\ref{rem:hoif} of Section~\ref{sec:unknown} for a more in-depth discussion.


\subsubsection{Known (Population) Covariance}

The majority of our results relies on the knowledge of the variance-covariance matrix of baseline covariates in the study. This known (population) covariance assumption has also been consistently imposed in the literature on the inference of high-dimensional GLMs under proportional asymptotics \citep{bellec2023debiasing, bellec2025observable}, in particular when $p > n$. Indeed, \citet{verzelen2018adaptive} demonstrate the impossibility of estimating or conducting statistical inference on certain functionals in high dimensional regression with unknown arbitrary variance-covariance matrix of the covariates when $p \gg n^{1 + c}$ for some $c > 0$. This does not preclude the designing of procedures informed by a priori assumptions on the variance-covariance matrix of the covariates -- a philosophy that has indeed been successfully espoused in the ultra-high-dimensional sparse GLM-based inferences \citep{verzelen2018adaptive}. Its parallels in the proportional asymptotic regime without sparsity assumptions are quite sporadic, and we are only aware of \citet{li2023spectrum}, and to some extent \citet{takahashi2018statistical}, that address this problem for linear forms of the coefficients under the right-rotationally invariant design. Our main results also assume that $\bSigma$ is known. However, under Gaussian designs, in Section~\ref{sec:unknown}, we establish $\sqrt{n}$-consistency of our proposed estimator when $\bSigma$ is unknown as long as the sample covariance matrix estimator of $\bSigma$ is invertible, demonstrating that knowing $\bSigma$ is not essential for our proposal. Furthermore, upon the completion of the first version of our draft, N. Verzelen brought to our attention \citet{kong2018estimating}. In that paper, the authors developed an estimator of the quadratic form of the regression coefficients (also known as ``learnability'' in the theoretical computer science literature) in logistic regression only when $\bX \sim \N_{p} (\bm{0}, \bSigma)$ with $p \asymp n$, with and without knowing $\bSigma$. Their estimator without knowing $\bSigma$ formally resembles the Higher-Order Influence Function estimators \citep{robins2008higher, robins2016technical}; we will discuss their similarity and difference further in Remark~\ref{rem:hoif}. Our results cover general GLMs beyond logistic regression without forcing the covariates to have zero mean. We also additionally consider more complex functionals often encountered in observational studies, such as the average treatment effects.

\subsection*{Organization}

To elaborate on the main thesis of the paper, we divide our discussions into the following subsections. In Sections~\ref{sec:glm} (knowing $\bSigma$) and \ref{sec:unknown} (not knowing $\bSigma$), we present our results on inference in GLMs followed by its applications in observational studies collected in Section~\ref{sec:obs}. Subsequently, Section~\ref{sec:sims} validates the theoretical results via numerical experiments. Our article ends by discussing some open problems in Section~\ref{sec:discussion}. All the proof details are deferred to the Appendix.

\subsection*{Notation}

We denote $(\be_{j}, j = 1, \cdots, p)$ as the standard bases of $\bbR^{p}$. Given any positive integers $\ell \leq k$, we denote $[k] \coloneqq \{1, \cdots, k\}$ and $[\ell:k] \coloneqq \{\ell, \ell + 1, \cdots, k\}$. $\bbU_{n, m} (\cdot)$ is the $m$-th order $U$-statistic operator: given a function $h: \bbR^{m} \rightarrow \bbR$,
\begin{align*}
\bbU_{n, m} [h (O_{1}, \cdots, O_{m})] \coloneqq \frac{(n - m)!}{n!} \sum_{1 \leq i_{1} \neq \cdots \neq i_{m} \leq n} h (O_{i_{1}}, \cdots, O_{i_{m}}),
\end{align*}
where $O_{i} \in \bbR$, for $i \in [n]$. When $m = 1$, $\bbU_{n, 1} [h (O)] \equiv n^{-1} \sum_{i = 1}^{n} h (O_{i})$ then reduces to the empirical mean operator. Given any two vectors $\bv, \bu$ and any matrix $\bbA$ with matching dimensions, we denote $\langle \bv, \bu \rangle_{\bbA} \coloneqq \bv^{\top} \bbA \bu$ the inner product between $\bv$ and $\bu$ with respect to $\bbA$; when $\bbA$ is non-negative semi-definite (n.n.s.d.), given any vector $\bv$, this inner product induces a norm $\Vert \bv \Vert_{\bbA} \equiv \bv^{\top} \bbA \bv$. When $\bbA = \bI$, the identity matrix, $\Vert \cdot \Vert_{\bI} \equiv \Vert \cdot \Vert$ reduces to the standard $\ell_{2}$-norm of a vector. Given a random vector $\bX$, $\Vert \bX \Vert_{\psi_{2}}$ denotes its Orlicz $\psi_{2}$-norm. $\lambda_{\min} (\bbA)$ and $\lambda_{\max} (\bbA)$, respectively, denote the minimum and maximum eigenvalues of $\bbA$ when it is symmetric and n.n.s.d..

To avoid clutter, we also introduce the short-hand notation $\bbE^{m} [\cdot] \equiv \{\bbE [\cdot]\}^{m}$. A general theme throughout this paper is to construct a multi-valued map $\Psi = (\Psi_{1}, \cdots, \Psi_{k}): \calD \rightarrow \calR$ from its domain $\calD$ to its range $\calR$ using MoM. Given a subset $I \subset [k]$, we let $\Psi_{I} \coloneqq \{\Psi_{j}, j \in I\}$ and $\Psi_{I}^{-1}$ be the inverse map of $\Psi_{I}$ if $\Psi_{I}$ is invertible. Given a $k$-th differentiable function $f$, let $f^{(k)}$ denote its $k$-th derivative. Finally, we denote $\Vert f \Vert_{q}$ and $\Vert f \Vert_{\infty}$ as, respectively, the $L_{q} (\bbP)$- and $L_{\infty}$-norms of $f$, for $q \geq 1$.

\section{Inference in GLMs}
\label{sec:glm}

In this section, we illustrate our main idea under the following stylized GLM. Suppose that we observe 
\begin{equation}
\label{GLM} \tag{$\mathsf{GLM}$}
\begin{split}
& (Y_i, \mathbf{X}_i)_{i = 1}^{n} \stackrel{\rm i.i.d.}{\sim}\mathbb{P}_{\bbeta}, \text{ with } Y_i \in \mathbb{R}, \, \bX_{i} \sim (\bm{\mu}, \bSigma)
\end{split}
\end{equation}
where $\bm{\mu} \in \bbR^{p}$ is the \textit{unknown} mean vector and $\bSigma \in \bbR^{p \times p}$ is the \textit{known} n.n.s.d. population covariance matrix, and there exists a (possibly) nonlinear \textit{known} link function $\phi: \bbR \rightarrow \calR \subseteq \bbR$ such that $\bbE [Y | \bX = \bx] = \phi (\bx^\top \bbeta)$ with $\bbeta = (\beta_1, \ldots, \beta_p)^{\top} \in \mathbb{R}^p$. The range $\calR$ of $\phi$ is problem specific -- e.g. when $Y$ is binary and $\phi$ is the expit/logistic function, then $\calR \equiv [0, 1]$. In this part, we address the question of $\sqrt{n}$-consistent estimation of $\beta_j$ for any $j = 1, \ldots, p$, the linear form of $\bbeta$ along the direction of $\bm{\mu}$ and the quadratic form of $\bbeta$ with respect to $\bSigma$:
\begin{equation}
\label{key parameters}
\lambda_{\beta} \coloneqq \bbeta^{\top} \bm{\mu}, \, \, \, \, \gamma_{\beta}^{2} \coloneqq \|\bbeta\|_{\bSigma}^2.
\end{equation}
In particular, the quadratic form has been used often in applications related to heritability estimation in genetics, e.g. \citet{guo2019optimal, song2024hede} and references therein. Moreover, as we will see while studying problems of estimating functionals of interest in observational studies in Section~\ref{sec:obs} with two nuisance functions parameterized by GLMs, our analysis in this section will provide the fundamental building blocks.  Therefore, looking forward to the case of simultaneously dealing with two high-dimensional GLMs, for the regression coefficients $\balpha$ from a separate GLM, we shall also adopt the same convention by denoting $\lambda_{\alpha} \coloneqq \balpha^{\top} \bm{\mu}$ and $\gamma_{\alpha}^{2} \coloneqq \Vert \balpha \Vert_{\bSigma}^{2}$. The bounded conditional fourth moment condition on $Y | \bX$ is required when establishing the CAN property of our proposed estimators and is imposed here to simplify the exposition (see Appendix~\ref{app:clt}).

To discuss the main results of this section and later parts of the paper, we will work with a set of assumptions that we present and discuss before introducing the main ideas of the proposal. It is worth noting that we index assumptions using a single capital letter to highlight their substantive meanings, the majority of which is summarized in Table~\ref{tab:glossary}, together with where the assumptions are imposed throughout this paper.

\begin{table}[htbp]
\centering
\begin{tabular}{c|cc}
\hline
Assumption & Meaning & Whereabout \\
\hline
$\mathsf{D}$ & Bounds on Design Mean \& Covariance & Global \\
$\mathsf{L}$ & Link Function & Global \\
$\mathsf{C}$ & Condition Number $p / n$ & Global \\
$\mathsf{B}$ & Bounds on $\Vert \bbeta \Vert$ & Almost Global \\
$\mathsf{V}$ & Conditional Variance/Moments of Response & Almost Global \\
$\mathsf{G}_{0}$ & Gaussian Design and Knowing $\bmu = \bm{0}$ & Sections~\ref{sec:zero mean} and \ref{sec:unknown} \\
$\mathsf{U}_{0}$ & Universality Conditions and Knowing $\bmu = \bm{0}$ & Sections~\ref{sec:zero mean} and \ref{sec:unknown} \\
$\mathsf{G}$ & Gaussian Design with Unknown $\bmu$ & Except Sections~\ref{sec:zero mean} and \ref{sec:unknown} \\
$\mathsf{U}$ & Universality Conditions with Unknown $\bmu$ & Except Sections~\ref{sec:zero mean} and \ref{sec:unknown} \\
\hline
\end{tabular}
\caption{A glossary for a part of the assumption indices: ``Almost Global'' means that certain parts of the Assumption are not imposed in some of the theorems. Here $\bmu \coloneqq \bbE \bX$ is the mean of the Design $\bX$.}
\label{tab:glossary}
\end{table}

First, we state the following ``global'' assumptions imposed on $\bm{\mu}$,  $\bSigma$, and the link functions $\phi$.

\begin{customas}{$\mathsf{D}$}
\label{as:Sigma}
There exist universal constants $M > 0$ such that
\begin{align*}
M^{-1} \leq \liminf_{p \rightarrow \infty} \lambda_{\min}(\bSigma) \leq \limsup_{p \rightarrow \infty} \lambda_{\max}(\bSigma) \leq M \text{ and } \Vert \bm{\mu} \Vert \leq M.
\end{align*}
\end{customas}

\begin{customas}{$\mathsf{L}$}
\label{as:link}
The link function $\phi: \bbR \rightarrow \calI \subseteq \bbR$, where $\mathcal{I}$ is a closed or open interval in $\bbR$, assumed to satisfy the following conditions:
\begin{enumerate}
\item[(1)] $\phi$ is three-times differentiable; the first, second, and third derivatives of the link function, together with the link function itself, are integrable with respect to the law of $\bX$ and the integrals are all strictly bounded by some universal constant. There also exists a bounded function $f: \bbR \rightarrow \bbR_+$ with $\lim_{|t| \rightarrow \infty} f(t) = 0$ such that $|\phi^{(\ell)} (t)| \leq e^{t^2 f (t)}$ for almost all $t \in \bbR$, for $\ell =1, 2, 3$.
\item[(2)] $\phi$ is strictly monotone and both $\phi (x)$ and $\phi^{(1)} (x)$ converge to the boundaries of their respective ranges (possibly $-\infty$ or $+\infty$) as $|x| \rightarrow \infty$.
\end{enumerate}
\end{customas}

\begin{remark}
\label{rem:link}
Assumption~\ref{as:link} accommodates many GLMs commonly encountered in practice, including the logistic regression, probit regression, Poisson/Negative-Binomial log-linear regression, and etc. The latter part of Assumption~\ref{as:link}~(2) also holds for all the above link functions and it will be needed to show that the map from the moments to the functionals of regression coefficients is a global diffeomorphism; see Appendix~\ref{app:mu}.
\end{remark}

As stated in the Introduction, our focus is on estimating functionals related to GLMs within the framework of the proportional asymptotic regime. Consequently, we also operate under the following assumption between the dimension $p$ and the sample size $n$.

\begin{customas}{$\mathsf{C}$}
\label{as:proportion}
There exists $\delta \in (0, \infty)$ such that $\lim_{n \rightarrow \infty} p / n \rightarrow \delta$.
\end{customas}

\begin{remark}
\label{rem:assumptions}
To shorten the exposition, we always assume Assumptions~\ref{as:Sigma},~\ref{as:link}, and~\ref{as:proportion} without explicitly mentioning them unless stated otherwise.
\end{remark}

Additionally, we state the following boundedness assumption on $\bbeta$; the second part is imposed to rule out the degenerate case $\bbeta \equiv \bm{0}$. Some further comments on this assumption can be found in Remark~\ref{rem:boundary} later.
\begin{customas}{$\mathsf{B}$}\leavevmode
\label{as:bounded}
\begin{itemize}
\item[(1)] There exists a universal constants $0 < \bar{B} < \infty$ such that $\Vert \bbeta \Vert \leq \bar{B}$;
\item[(2)] There exists a universal constants $0 < \ubar{B} \leq \bar{B}$ such that $\Vert \bbeta \Vert \geq \ubar{B}$.
\end{itemize}
\end{customas}

Next, we generally need to impose the first part of the following condition on the conditional second moment $Y | \bX$. The second part will be needed when establishing the CAN property of our proposed estimator.
\begin{customas}{$\mathsf{V}$}\leavevmode
\label{as:var-cov}
\begin{itemize}
\item[(1)] We assume that $\Vert \sigma^{2} \Vert_{2}$ is bounded, where $\sigma^{2} (\cdot) \coloneqq \bbE [Y^{2} | \bX = \cdot]$ is the conditional second moment function of $Y$ given $\bX$;
\item[(2)] Let $\sigma^{k} (\bx) \coloneqq \bbE [Y^{k} | \bX = \bx]$ for $k = 2, 4$. We assume that $\Vert \sigma^{k} \Vert_{2}$ is bounded and $\sigma^{k} (\cdot)$ is a GLM sharing the same regression coefficients $\bbeta$, but with possibly different three-times differentiable link functions belonging to $L_{2} (\bbP)$.
\end{itemize}
\end{customas}

\begin{remark}
Assumption~\ref{as:var-cov}(2) is mainly made to ease exposition when we establish the CAN property of our proposed estimator (see e.g. Proposition~\ref{prop:glm_clt} and Theorem~\ref{thm:GLM CLT}). But it holds for many popular GLMs encountered in practice. For example, when $Y | \bX \sim \mathrm{Ber} (\phi (\bbeta^{\top} \bX))$, $\bbE [Y^{2} | \bX] = \bbE [Y^{4} | \bX] = \phi (\bbeta^{\top} \bX)$; when $Y | \bX \sim \mathrm{Pois} (\phi (\bbeta^{\top} \bX))$, $\bbE [Y^{2} | \bX] = \phi (\bbeta^{\top} \bX) + \phi^{2} (\bbeta^{\top} \bX)$ and $\bbE [Y^{4} | \bX] = \phi^{4} (\bbeta^{\top} \bX) + 6 \phi^{3} (\bbeta^{\top} \bX) + 7 \phi^{2} (\bbeta^{\top} \bX) + \phi (\bbeta^{\top} \bX)$.
\end{remark}

Finally, it is worth noting that the symbols for the link function, the regression coefficients and the conditional second moment functions in the above assumptions shall be interpreted as generic notations, as in the sequel we may specialize to problem-specific symbols. For example, later in Section~\ref{sec:obs}, we also use $\eta$ for the link function and $\balpha$ for the regression coefficients. 

\subsection{Results for designs that are known to have zero mean}
\label{sec:zero mean}

To gather intuition for our method, it is instructive first to consider the following assumptions.

\begin{customas}{$\mathsf{G_0}$}
\label{as:normal mean zero}
$\bX \sim \N_p (\bm{0}, \bSigma)$ or equivalently $\bZ \sim \N_{p} (\bm{0}, \bI)$ and $\bm{\mu}$ is known to equal $\bm{0}$.
\end{customas}

Our method then essentially relies on the following result, a direct consequence of Stein's lemma or Gaussian Integration by Parts.

\begin{lemma}
\label{lem:glm mean zero}
Under Model~\ref{GLM}, Assumptions~\ref{as:normal mean zero},~\ref{as:bounded}{\rm(1)} and~\ref{as:var-cov}{\rm(1)}, the following hold:
\begin{enumerate}
\item[\emph{(1)}] Given any fixed vector $\bm{\upsilon} \in \bbR^p$ and fixed matrix $\bbM \in \bbR^{p\times p}$, the following system of moment equations holds:
\begin{subequations}
\label{original chain}
\begin{align}
& \bbE [Y \bX^{\top}] \bbM \bbE [\bX Y] = \bbeta^{\top} \bSigma \bbM \bSigma \bbeta \cdot \bbE^2 [\phi' (\bX^{\top} \bbeta)], \\
& \bbE [Y \bX^{\top}] \bbM \bm{\upsilon} = \bbE [\phi'(\bX^{\top} \bbeta)] \cdot \bbeta^{\top} \bSigma \bbM \bm{\upsilon}.
\end{align}
\end{subequations}
Consequently, choosing $\bbM = \bSigma^{-1}$, we have
\begin{subequations}
\label{GLM mean zero chain}
\begin{align}
& m_{\bX Y, 2} \coloneqq \bbE [Y \bX^{\top}] \bSigma^{-1} \bbE [\bX Y] = \bbE^2 [\phi' (\bX^{\top} \bbeta)] \cdot \gamma_{\beta}^{2} = \f_{1}^{2} (\gamma_{\beta}^{2}) \cdot \gamma_{\beta}^{2}, \label{mm} \\
& m_{\bX Y, \bm{\upsilon}} \coloneqq \bbE [Y \bX^{\top}] \bSigma^{-1} \bm{\upsilon} = \bbE [\phi'(\bX^{\top} \bbeta)] \cdot \bbeta^{\top} \bm{\upsilon} \equiv \f_1 (\gamma_{\beta}^{2}) \cdot \bbeta^{\top} \bm{\upsilon}, \label{mm2}
\end{align}
\end{subequations}
where $\f_{1} (t) \coloneqq \bbE [\phi' (Z)]$ with $Z \sim \N (0, t)$ for $t \geq 0$. Denote the map induced by \eqref{mm} as $\Psi_{\GLM_{0}, \beta}: \gamma_{\beta}^{2} \mapsto m_{\bX Y, 2}$ and the map induced by \eqref{GLM mean zero chain} as $\Psi_{\GLM_{0}}: (\gamma_{\beta}^{2}, \bbeta^{\top} \bm{\upsilon}) \mapsto (m_{\bX Y, 2}, m_{\bX Y, \bm{\upsilon}})$.

\item[\emph{(2)}] Further, $\Psi_{\GLM_{0}, \beta}$ is a diffeomorphism with $\nabla (\Psi_{\GLM_{0}, \beta}^{-1})$ bounded; and the same holds for $\Psi_{\GLM_{0}}$. Consequently, $\gamma_{\beta}^{2}$ and $\bbeta^{\top} \bm{\upsilon}$ are identifiable in the sense that the LHS of \eqref{GLM chain} uniquely determines the value of $(\gamma_{\beta}^{2}, \bbeta^{\top} \bm{\upsilon})$.
\end{enumerate}
\end{lemma}

We now unpack Lemma~\ref{lem:glm mean zero}, with its proof deferred to the Appendix. Based on the Gaussian design and Stein's lemma, the moment equations \eqref{original chain} marry the moments on the LHS with certain nonlinear transformation of $\bbeta$ on the RHS. The most important moment equation here is $\Psi_{\GLM_{0}, \beta}$ induced by \eqref{mm}, that maps the quadratic form $\gamma_{\beta}^{2}$ to the moment $m_{\bX Y, 2}$. Assumption~\ref{as:link} on the link function and Assumption~\ref{as:bounded} on $\gamma_{\beta}^{2}$ together ensures that $\Psi_{\GLM_{0}, \beta}$ is a diffeomorphism, so $\gamma_{\beta}^{2} = \Psi_{\GLM_{0}, \beta}^{-1} (m_{\bX Y, 2})$ is identified. It will also be made clear later that $\Psi_{\GLM_{0}, \beta}$ being a diffeomorphism entails that $\sqrt{n}$-consistent and CAN estimators of $\gamma_{\beta}^{2}$ can be constructed. After identifying $\gamma_{\beta}^{2}$, by solving \eqref{mm2}, $\bbeta^{\top} \bm{\upsilon} = m_{\bX Y, \bm{\upsilon}} / \f_{1} (\gamma_{\beta}^{2})$ is as well identified from the moments. Taking $\bm{\upsilon} = e_{j}$, the $j$-th standard basis in $\bbR^{p}$, the same strategy identifies $\beta_{j}$, for any $j \in [p]$.

The conclusions in Lemma~\ref{lem:glm mean zero}, together with all the other identification results under Gaussian designs in this paper, do not require Assumption~\ref{as:proportion}. However, the above moment equations critically rely on the Gaussianity of $\bX$. It is natural to ask if, similar to a growing body of work studying universality for regression models under proportional asymptotics (see e.g. \cite{bayati2015universality, montanari2022universality, montanari2023universality, hu2022universality, dudeja2023spectral, lahiry2023universality} and references therein), one could move beyond Gaussian designs and demonstrate the universality of the above identification result. We provide a positive answer to this question following a relaxed identification criterion and shifting the burden of assumption from $\bX$ to $\bbeta$.

\begin{definition}[$\sqrt{n}$-identifiability]
\label{def:id}
We say that a low-dimensional target parameter $\psi \in \bbR^{k}$, where $k$ is strictly bounded, of the underlying statistical model $\bbP$ (e.g. Model~\ref{GLM}) is $\sqrt{n}$-identifiable if there exists a (possibly) nonlinear map $\Psi$ from $\psi$ to certain moments defined by $\bbP$ induced by $\psi$, such that if given two different values of the target parameter, $\psi$ and $\psi'$, such that $\Vert \psi - \psi' \Vert \gtrsim n^{- 1 / 2}$, then $\Vert \Psi (\psi) - \Psi (\psi') \Vert \gtrsim n^{- 1 / 2}$, for sufficiently large $n$.
\end{definition}

\begin{customas}{$\mathsf{U_0}$}\leavevmode
\label{as:beta mean zero}
\begin{enumerate}[label = (\arabic*)]
\item $\bX = \bSigma^{1 / 2} \bZ$, where $\bZ = (Z_{1}, \cdots, Z_{p})^{\top}$ has independent coordinates with zero mean and unit variance, and there exists a universal constant $M > 0$ such that $\Vert Z_{j} \Vert_{\psi_{2}} \leq M$ for $j = 1, \cdots, p$;
\item $\sqrt{p} \bSigma^{1 / 2} \bbeta \overset{\calW_{8}}{\rightarrow} \mathsf{b}$ where $\mathsf{b} \sim \rho$ for some probability measure $\rho$ supported on $\bbR$ and we assume that $\rho$ has bounded first and second moments\footnote{Here, given a random vector $\mathsf{A} \in \bbR^{p}$ and a random variable $\mathsf{a} \in \bbR$, the notation $\mathsf{A} \overset{\calW_{8}}{\rightarrow} \mathsf{a}$ means that the empirical distribution over the coordinates of $\mathsf{A}$ converges in $\calW_{8}$-distance (8-Wasserstein distance) to the distribution of $\mathsf{a}$, as $p \rightarrow \infty$.}.
\end{enumerate}
\end{customas}

\begin{remark}
When $\bSigma = \bI_{p}$, Assumption~\ref{as:beta mean zero}(2) reduces to $\sqrt{p} \bbeta \overset{\calW_{8}}{\rightarrow} \mathsf{b}$. This type of assumptions are commonly imposed . For general population covariance matrix $\bSigma$, under $\sqrt{p} \bbeta \overset{\calW_{8}}{\rightarrow} \mathsf{b}$, the more general assumption will be met under additional assumptions on $\bSigma$.
\end{remark}

We then have the following parallel result of Lemma~\ref{lem:glm mean zero}, without assuming that $\bX$ is Gaussian.

\begin{lemma}
\label{lem:glm mean zero, universality}
Under the same assumptions as in Lemma~\ref{lem:glm mean zero}, except with Assumption~\ref{as:normal mean zero} replaced by Assumption~\ref{as:beta mean zero}, the system of moment equations appeared in Lemma~\ref{lem:glm mean zero} holds approximately with approximation error $O (p^{-3/4}) = O (n^{- 3 / 4})$ as $p \rightarrow \infty$. Thus $\gamma_{\beta}^{2}$ and $\bbeta^{\top} \bm{\upsilon}$ are $\sqrt{n}$-identifiable for any fixed vector $\bm{\upsilon} \in \bbR^{p}$.
\end{lemma}

The proof of this lemma can be found in Appendix~\ref{app:universality}. Taken together, the above results motivate the following MoM-based estimator of $\gamma_{\beta}^{2}$ and $\bbeta^{\top} \bm{\upsilon}$:
\begin{equation}
\label{key estimator}
\begin{split}
& \hat{\gamma}_{\beta}^{2} \coloneqq \Psi_{\GLM_{0}, \beta}^{-1} \left(\hat{m}_{\bX Y, 2} \mathbbm{1} (\hat{m}_{\bX Y, 2} \in \mathcal{R}_{\GLM_{0}, \beta}) \right), \quad \hat{m}_{\bX Y, 2} \coloneqq \bbU_{n, 2} [Y_1 \bX_1^{\top} \bSigma^{-1} \bX_2 Y_2], \\
& \hat{\bbeta}^{\top} \bm{\upsilon} \coloneqq \frac{\hat{m}_{\bX Y, \bm{\upsilon}}}{\f_{1} (\hat{\gamma}_{\beta}^{2})}, \quad \hat{m}_{\bX Y, \bm{\upsilon}} \coloneqq \bbU_{n, 1} [Y \bX^{\top}] \bSigma^{-1} \bm{\upsilon},
\end{split}
\end{equation}
where $\mathcal{R}_{\GLM_{0}, \beta}$ denotes the range of $\Psi_{\GLM_{0}, \beta}$.

Since $\Psi_{\GLM_{0}, \beta}$ is a diffeomorphism, it is clear from the construction above that to prove $\sqrt{n}$-consistency of $\hat{\gamma}_{\beta}^{2}$ and $\hat{\bbeta}^{\top} \bm{\upsilon}$ one needs to verify $\max \{\var (\hat{m}_{\bX Y, 2}), \var (\hat{m}_{\bX Y, \bm{\upsilon}})\} = O (1 / n)$, which is indeed the case (see Appendix~\ref{app:universality}). We next summarize the above reasoning as the following proposition. 

\begin{proposition}
\label{prop:glm_rate}
Under the Assumptions of Lemma~\ref{lem:glm mean zero} or Lemma~\ref{lem:glm mean zero, universality}, the following hold:
\begin{align*}
& \sqrt{n} (\hat{\gamma}_{\beta}^{2} - \gamma_{\beta}^{2}) = O_{\bbP} (1), \text{ and for any $j = 1, \ldots, p$, } \sqrt{n} (\hat{\bbeta}^{\top} \bm{\upsilon} - \bbeta^{\top} \bm{\upsilon}) = O_{\bbP} (1).
\end{align*}
\end{proposition}

In fact, as $n \rightarrow \infty$, we can consider a more precise result and record that the above MoM-based estimators are CAN under the Gaussian design and some additional regularity conditions.


\begin{proposition}
\label{prop:glm_clt}
Under Model~\ref{GLM}, Assumptions~\ref{as:normal mean zero},~\ref{as:bounded} and~\ref{as:var-cov}, if we further assume that $\Vert \bm{\beta} \Vert_{f (\bSigma)}^{2}$ converges to some nontrivial limit for $f (\bSigma) = \bSigma, \bSigma^2, \bSigma^3$, we have
\begin{align*}
\sqrt{n} (\hat{\beta} - \beta_{j}) \overset{\mathcal{L}}{\rightarrow} \N (0, \nu_{j}^{2})
\end{align*}
for some constant $\nu_{j}^{2} > 0$ for $j = 1, \cdots, p$ and
\begin{align*}
\sqrt{n} (\hat{\gamma}_{\beta}^{2} - \gamma_{\beta}^{2}) \overset{\mathcal{L}}{\rightarrow} \N (0, \nu^{2})
\end{align*}
for some constant $\nu^{2} > 0$.
\end{proposition}

\begin{remark}
\label{rem:boundary}
It is noteworthy that Assumption~\ref{as:bounded}(2) is imposed mainly to ensure that the event $\mathbbm{1} (\hat{m}_{\bX Y, 2} \in \mathcal{R}_{\GLM_{0}, \beta})$ holds with probability converging to 1 such that the constraint does not affect the asymptotic distribution of the $U$-statistic estimator $\hat{m}_{\bX Y, 2}$. We conjecture that it is possible to relax this assumption by a more precise analysis of the asymptotic behavior of $\hat{m}_{\bX Y, 2}$ close to the boundary $\Vert \bbeta \Vert = 0$, which is left for future work.
\end{remark}

\begin{remark}
\label{rem:clt}
The CAN property of the proposed MoM-based estimators relies on two separate results: (1) the CLT of first-order and second-order $U$-statistics, followed from the results in \citet{bhattacharya1992class} (see Appendix~\ref{app:CLT} for a complete proof) (2) $\Psi_{\GLM_{0}, \beta}$ is a diffeomorphism, and its inverse map, $\Psi_{\GLM_{0}, \beta}^{-1}$, has bounded derivative so the Delta Method can be applied.
\end{remark}

\begin{remark}
\label{rem:clt condition}
In Proposition~\ref{prop:glm_clt} (and similar results related to the CAN property of our proposed estimator in the sequel), we need some extra assumptions on the convergence of inner products such as $\bbeta^{\top} f (\bSigma) \bbeta$. It is worth mentioning that the assumption imposed in the main text might not be tight. By speculating the derivations in Appendix~\ref{app:CLT}, one only needs either $\bbeta^{\top} \bSigma \bbeta$ or $\bbeta^{\top} \bSigma^{3} \bbeta$ to converge. But to avoid unnecessary technical complications that are irrelevant to the main theme of the paper, we decide not to pursue further in this direction. Also, one can easily find sufficient conditions to establish the convergence of such quantities. As a simple example, when $\bSigma = \bI_{p}$ and $\bbeta$ satisfies Assumption~\ref{as:beta mean zero}, we immediately have $\bbeta^{\top} f (\bSigma) \bbeta$ to converge.

In addition, we also impose an extra condition on $\bbE [Y^{4} | \bX]$ via Assumption~\ref{as:var-cov}. This assumption is to ensure that certain re-scaled fourth moments of the $U$-statistics vanish to zero as $n \rightarrow \infty$, which is required based on the proof strategy that we currently employ (see Lemma~\ref{lem:clt} and Proposition~\ref{prop:general clt} in Appendix~\ref{app:CLT}).

Finally, we do not explicitly specify the form of the asymptotic variance, which is complicated due to the use of second-order $U$-statistics. Nonetheless, in our previous work (Appendix A.8 of \citet{liu2024assumption}), consistent variance estimators based on tweaking the nonparametric bootstrap have been developed and can be used to conduct inference; also see Section~\ref{sec:bootstrap} for a brief discussion and Appendix~\ref{app:bootstrap} for its finite-sample performance.
\end{remark}

\begin{remark}
\label{rem:debunk}
At this point, readers might wonder why we consider moments such as $m_{\bX Y, 2}$ that involves $\bX$. When $\bm{\mu} \equiv \bm{0}$, $Y \in \{0, 1\}$ and $\phi (\cdot) = \expit (\cdot)$, it is obvious that all the moments of $Y$ reduce to $\bbE [Y] \equiv 0.5$. Thus without leveraging information of $\bX$, it is generally impossible to identify functionals of $\bbeta$. This is why the moment-based estimator in \citet{sawaya2023moment}, based only on $\bbE [Y]$, cannot be directly applied to logistic or probit regression.
\end{remark}

\subsection{Results for designs with unknown and possibly non-zero means}
\label{sec:mu}

Next, we show that the zero covariate-mean condition is not essential to our moment-based approach, but a larger system of moment equations is required to identify relevant parameters in GLMs. The development closely mirrors that of the previous section when $\bm{\mu}$ is known to equal $\bm{0}$. When $\bm{\mu}$ is unknown, however, because the map between the moments and the functionals of the regression coefficients is from $\bbR^{2}$ to $\bbR^{2}$, it is not as straightforward to show that this map is a global diffeomorphism as in the previous section. As before, we first set the stage under Gaussian designs.

\begin{customas}{$\mathsf{G}$}
\label{as:normal}
$\bX \sim \N_p (\bm{\mu}, \bSigma)$ or equivalently $\bZ \sim \N_{p} (\bm{0}, \bI)$.
\end{customas}

For short, we let $\bm{\nu} = (\nu_{1}, \cdots, \nu_{p})^{\top} \coloneqq \bSigma^{-1} \bm{\mu}$. The following lemma then generalizes Lemma~\ref{lem:glm mean zero} to the case of Gaussian designs with unknown mean $\bm{\mu}$.
\begin{lemma}
\label{lem:mu}
Under Model~\ref{GLM}, Assumptions~\ref{as:normal},~\ref{as:bounded}{\rm(1)}, and~\ref{as:var-cov}{\rm(1)}:
\begin{enumerate}
\item[\emph{(1)}] The following system of moment equations holds:
\begin{subequations}
\label{GLM chain}
\begin{align}
& m_{Y} \coloneqq \bbE [Y] = \f_{0} (\unknowna{\lambda_{\beta}}, \unknownb{\gamma_{\beta}^{2}}), \label{GLM m1} \\
& m_{\bX, 2} \coloneqq \bbE [\bX^{\top}] \bSigma^{-1} \bbE [\bX] = \bm{\mu}^{\top} \bSigma^{-1} \bm{\mu}, \label{GLM m2} \\
& m_{\bX Y, \bX} \coloneqq \bbE [Y \bX^{\top}] \bSigma^{-1} \bbE [\bX] = m_{Y} \cdot m_{\bX, 2} + \f_{1} (\unknowna{\lambda_{\beta}}, \unknownb{\gamma_{\beta}^{2}}) \cdot \unknowna{\lambda_{\beta}}, \label{GLM m3} \\
& m_{\bX Y, 2} \coloneqq \bbE [Y \bX^{\top}] \bSigma^{-1} \bbE [\bX Y] = m_{Y}^{2} \cdot m_{\bX, 2} + \f_{1}^{2} (\unknowna{\lambda_{\beta}}, \unknownb{\gamma_{\beta}^{2}}) \cdot \unknownb{\gamma_{\beta}^{2}} + 2 \cdot m_{Y} \cdot \f_{1} (\unknowna{\lambda_{\beta}}, \unknownb{\gamma_{\beta}^{2}}) \cdot \unknowna{\lambda_{\beta}}, \label{GLM m4} \\
& m_{\nu_{j}} \coloneqq \bbE [\bX]^{\top} \bSigma^{-1} \be_{j} = \bm{\mu}^{\top} \bSigma^{-1} \be_{j} = \bm{\nu}^{\top} \be_{j} = \nu_{j}, \label{GLM m5} \\
& m_{\beta_{j}} \coloneqq \bbE [Y \bX^{\top}] \bSigma^{-1} \be_{j} = \f_{0} (\unknowna{\lambda_{\beta}}, \unknownb{\gamma_{\beta}^{2}}) \cdot \nu_{j} + \f_{1} (\unknowna{\lambda_{\beta}}, \unknownb{\gamma_{\beta}^{2}}) \cdot \beta_{j}. \label{GLM m6}
\end{align}
\end{subequations}
where $\f_{k} (s, t) \coloneqq \bbE [\phi^{(k)} (Z)]$ with $Z \sim \N (s, t)$. Denote the forward map induced by this system as
\begin{align*}
\Psi_{\GLM} = (\Psi_{\GLM, 1}, \Psi_{\GLM, 2}, \cdots \Psi_{\GLM, 6})^{\top}: (\lambda_{\beta}, \gamma_{\beta}^{2})^{\top} \mapsto (m_{Y}, m_{\bX, 2}, \cdots, m_{\beta_{j}})^{\top}.
\end{align*}

\item[\emph{(2)}] Further, the first four equations of \eqref{GLM chain}, denoted as $\Psi_{\GLM, [4]}$, can be reduced to
\begin{subequations}
\label{GLM chain reduced}
\begin{align}
& m_{1} \coloneqq m_{Y} = \f_{0} (\lambda_{\beta}, \gamma_{\beta}^{2}), \\
& m_{2} \coloneqq m_{\bX Y, 2} + m_{Y}^{2} \cdot m_{\bX, 2} - 2 \cdot m_{Y} \cdot m_{\bX Y, \bX} = \f_{1}^{2} (\lambda_{\beta}, \gamma_{\beta}^{2}) \cdot \gamma_{\beta}^{2}.
\end{align}
Denote the forward map induced by \eqref{GLM chain reduced} as $\Psi_{\GLM, \beta} = (\Psi_{\GLM, \beta, 1}, \Psi_{\GLM, \beta, 2})^{\top}: (\lambda_{\beta}, \gamma_{\beta}^{2})^{\top} \mapsto (m_{1}, m_{2})^{\top}$. Then $\Psi_{\GLM, \beta}$ is a diffeomorphism with $\nabla (\Psi_{\GLM, \beta}^{-1})$ bounded. Consequently, $\lambda_{\beta}, \gamma_{\beta}^{2}$, and $\beta_{j}$ are identifiable.
\end{subequations}
\end{enumerate}
\end{lemma}

\begin{remark}
\label{rem:difficult}
The proof of Lemma~\ref{lem:mu} can be found in Appendix~\ref{app:mu} and Appendix~\ref{app:non-zero}. As mentioned in the beginning of this section, the system of moment equations in \eqref{GLM chain reduced} is from $\bbR^{2}$ to $\bbR^{2}$. As a result, showing that \eqref{GLM} is a global diffeomorphism is nontrivial, which involves the use of Hadamard global inverse function theorem. For details, see Lemma~\ref{lem:inverse} and Appendix~\ref{app:non-zero}.
\end{remark}

To establish universality of Lemma~\ref{lem:mu} beyond Gaussian designs, we need to first generalize Assumption~\ref{as:beta mean zero} to Assumption~\ref{as:beta} below.

\begin{customas}{$\mathsf{U}$}\leavevmode
\label{as:beta}
\begin{enumerate}[label = (\arabic*)]
\item $\bX = \bSigma^{1 / 2} \bZ + \bm{\mu}$, where $\bZ = (Z_{1}, \cdots, Z_{p})^{\top}$ has independent coordinates with zero mean, unit variance, and $\max_{j=1}^p \|Z_j\|_{\psi_2} \leq M$ for some universal constant $M > 0$;
\item $\sqrt{p} \bSigma^{1 / 2} \bbeta \overset{\calW_{8}}{\rightarrow} \mathsf{b}$ and $\sqrt{p} \bSigma^{- 1 / 2} \bm{\mu} \overset{\calW_{8}}{\rightarrow} \mathsf{u}$ where $\mathsf{b} \sim \rho$ and $\mathsf{u} \sim \varrho$ respectively for some probability measures $\rho$ and $\varrho$ supported on $\bbR$ and we assume that both $\rho$ and $\varrho$ have bounded first and second moments.
\end{enumerate}
\end{customas}

Lemma~\ref{lem:glm mean zero, universality} can then be generalized as follows, the proof of which can be found in Appendix~\ref{app:universality}.

\begin{lemma}
\label{lem:mu, universality}
Under the same assumptions as in Lemma~\ref{lem:mu}, except with Assumption~\ref{as:normal} replaced by Assumption~\ref{as:beta}, the systems of moment equations appeared in Lemma~\ref{lem:mu} hold approximately with approximation error $O (p^{-3 / 4}) = O (n^{-3 / 4})$ as $n \rightarrow \infty$.
\end{lemma}

When the distribution of $\bX$ is unknown and one passes to universality after using a Gaussian identification strategy, we conjecture that certain delocalization conditions, such as Assumption~\ref{as:beta}\rm{(2)} on the (transformed) regression coefficients and covariate mean vector, are necessary. In fact, universality can fail when one starts from a Gaussian identification strategy and Assumption~\ref{as:beta}\rm{(2)} is violated -- as demonstrated via the numerical experiments related to Figures~\ref{fig:GLM, sparse-converge, rademacher, univ} and~\ref{fig:GLM, hist, sparse-converge, rademacher, univ}; see Section~\ref{sec:sims} for more details.

\begin{remark}
\label{rem:zero mean}
Compared to the special case of knowing $\bm{\mu} \equiv \bm{0}$ in Section~\ref{sec:zero mean}, it requires extra moment equations \eqref{GLM m1} to \eqref{GLM m3}  to identify $\gamma_{\beta}^{2}$, together with the linear form $\lambda_{\beta}$. When it is known that $\bm{\mu} = \bm{0}$, \eqref{GLM m1}, \eqref{GLM m2}, and \eqref{GLM m3}, respectively, reduce to constants $1 / 2$, $0$, and $0$.
\end{remark}

\begin{remark}
\label{rem:non-gaussian}
When the law of $\bX$ is absolutely continuous with density $p$ with respect to the Lebesgue measure is (partially) known but non-Gaussian, one could leverage the following generalized Stein's identity (and its higher-order analogues) to obtain similar moment equations
\begin{align*}
\bbE [f (\bX) s (\bX)] + \bbE [f' (\bX)] = 0,
\end{align*}
where $s (\bx) \coloneqq \nabla p (\bx) / p (\bx): \bbR^{p} \rightarrow \bbR^{p}$ is the Stein's score function and $f: \bbR^{p} \rightarrow \bbR$ is any differentiable function such that both terms in the above identity exist. We do not explore these generalizations in this paper and keep it as potential future directions.
\end{remark}

Gathering the development thus far, we can construct the following estimator of $(\lambda_{\beta}, \gamma_{\beta}^{2})$ based on $\Psi_{\GLM, \beta}$ and its inverse map $\Psi_{\GLM, \beta}^{-1}$:
\begin{equation}
\label{joint estimator}
\begin{split}
(\hat{\lambda}_{\beta}, \hat{\gamma}_{\beta}^{2}) \coloneqq \Psi_{\GLM, \beta}^{-1} \left( \hat{m}_{1} \mathbbm{1} \{\hat{m}_{1} \in \calR_{\GLM, \beta, 1}\}, \hat{m}_{2} \mathbbm{1} \{\hat{m}_{2} \in \calR_{\GLM, \beta, 2}\} \right),
\end{split}
\end{equation}
where
\begin{equation}
\label{GLM moments estimators}
\begin{split}
\hat{m}_{1} \coloneqq \hat{m}_{Y} & \coloneqq \bbU_{n, 1} [Y], \quad \hat{m}_{2} \coloneqq \hat{m}_{\bX Y, 2} + \hat{m}_{Y}^{2} \cdot \hat{m}_{\bX, 2} - 2 \cdot \hat{m}_{Y} \cdot \hat{m}_{\bX Y, \bX}, \\
\text{and } \hat{m}_{\bX, 2} \coloneqq \bbU_{n, 2} [\bX_{1}^{\top} \bSigma^{-1} \bX_{2}], & \, \hat{m}_{\bX Y, \bX} \coloneqq \bbU_{n, 2} [Y_{1} \bX_{1}^{\top} \bSigma^{-1} \bX_{2}], \, \hat{m}_{\bX Y, 2} \coloneqq \bbU_{n, 2} [Y_{1} \bX_{1}^{\top} \bSigma^{-1} \bX_{2} Y_{2}].
\end{split}
\end{equation}
With $(\hat{\lambda}_{\beta}, \hat{\gamma}_{\beta}^{2})$, one can estimate $\beta_{j}$ by simply solving \eqref{GLM m5} to \eqref{GLM m6}:
\begin{align*}
\hat{\beta}_{j} \coloneqq \frac{\hat{m}_{\beta_{j}} - \f_{0} (\hat{\lambda}_{\beta}, \hat{\gamma}_{\beta}^{2}) \cdot \hat{m}_{\nu_{j}}}{\f_{1} (\hat{\lambda}_{\beta}, \hat{\gamma}_{\beta}^{2})}
\end{align*}
where
\begin{align*}
\hat{m}_{\nu_{j}} \coloneqq \bbU_{n, 1} [\bX^{\top}] \bSigma^{-1} \be_{j}, \quad \hat{m}_{\beta_{j}} \coloneqq \bbU_{n, 1} [Y \bX^{\top}] \bSigma^{-1} \be_{j}.
\end{align*}

\begin{theorem}
\label{thm:GLM}
Under the Assumptions of Lemma~\ref{lem:mu} or Lemma~\ref{lem:mu, universality}, the following hold:
\begin{align*}
\sqrt{n} \left( \hat{\lambda}_{\beta} - \lambda_{\beta} \right) = O_{\bbP} (1), \sqrt{n} \left( \hat{\gamma}_{\beta}^{2} - \gamma_{\beta}^{2} \right) = O_{\bbP} (1), \text{ and for any $j = 1, \cdots, p$}, \sqrt{n} (\hat{\beta}_{j} - \beta_{j}) = O_{\bbP} (1).
\end{align*}
\end{theorem}

The final result in this section generalizes Proposition~\ref{prop:glm_clt} for the CAN property of our proposed estimators to the case where $\bm{\mu}$ is unknown and possibly non-zero. 


\begin{theorem}
\label{thm:GLM CLT}
Under Model~\ref{GLM}, Assumptions~\ref{as:normal},~\ref{as:bounded} and~\ref{as:var-cov}, if we further assume that $\langle \bv_{1}, \bv_{2} \rangle_{f (\bSigma)}$ converges to some nontrivial limit for $f (\bSigma) = \bSigma^{-1}, \bSigma, \bSigma^2, \bSigma^3$ for $\bv_{1}, \bv_{2} \in \{\bm{\mu}, \bbeta\}$, we have
\begin{align*}
\sqrt{n} (\hat{\beta}_{j} - \beta_{j}) \overset{\mathcal{L}}{\rightarrow} \N (0, \nu_{j}^{2})
\end{align*}
for some constant $\nu_{j}^{2} > 0$ for $j = 1, \cdots, p$ and
\begin{align*}
\sqrt{n} (\hat{\gamma}_{\beta}^{2} - \gamma_{\beta}^{2}) \overset{\mathcal{L}}{\rightarrow} \N (0, \nu^{2})
\end{align*}
for some constant $\nu^{2} > 0$.
\end{theorem}
In the proof of the above theorem in Appendix~\ref{app:GLM functional CLT}, we unpack the conditions of the theorem further. On a higher level, convergence of inner products $\langle \bv_{1}, \bv_{2} \rangle_{f (\bSigma)}$ to certain nontrivial limits relates to whether the asymptotic variances of moment estimators based on $U$-statistics, after scaled by $\sqrt{n}$, converge to nontrivial limits. For statistical inference, one could construct standard Wald intervals by estimating the asymptotic variances in Theorem~\ref{thm:GLM CLT} by the bootstrap method developed in \citet{liu2024assumption}. The performance of bootstrap variance estimators will be assessed in Appendix~\ref{app:bootstrap}.

\begin{remark}[On the scalings between $p$ and $n$]
\label{rem:p}
Even though we focus on the asymptotic regime where $p$ scales proportionally with $n$, it is worth noting that our proposed estimators can achieve consistency even when $p \gg n$, as long as $p = o (n^{2})$. This is because the variances of our $U$-statistic-based moment estimators are generally of order $\frac{1}{n} \vee \frac{p}{n^{2}}$. Asymptotic normality also holds with a different scaling factor $n / \sqrt{p}$ instead of $n^{1 / 2}$, by applying results from \citet{bhattacharya1992class}.
\end{remark}

\begin{remark}[Data-driven Approximate Message Passing Schemes for GLMs]
\label{rem:amp}
Before proceeding, we take a detour to present one immediate application of our MoM-based method for inference in GLMs. Further applications of our method can be found in later Section~\ref{sec:obs}. For logistic regression, i.e. Model \eqref{GLM} with $\phi \equiv \expit$, \citet{sur2019modernb} (and in a more general form, \citet{salehi2019impact}) showed the following: for $\hat{\bbeta}_{\MLE} \equiv (\hat{\beta}_{\MLE, 1}, \cdots, \hat{\beta}_{\MLE, p})^{\top}$ the Maximum Likelihood Estimator (MLE) of $\bbeta$, letting $(Z_{0}, Z_{1}, Z_{2}, Z_{3})^{\top} \sim \N_{4} (\bm{0}, \bI_{4})$, one has
\begin{align*}
\frac{1}{p} \sum_{j = 1}^{p} h \left( \hat{\beta}_{\MLE, j} - \bar{\alpha} \beta_{j}, \beta_{j} \right) \overset{\bbP}{\rightarrow} \bbE [h (\bar{\sigma} Z_{0}, \sfb)],
\end{align*}
where $(\bar{\alpha}, \bar{\sigma}, \bar{\gamma})$ is the solution to the following fixed point equations
\begin{equation}
\label{fixed point equation}
\begin{split}
\left\{ \begin{array}{rl}
\delta^{2} \bar{\sigma}^{2} & = \, 2 \bbE \left[ \phi (- \kappa Z_{1}) \left( \bar{\gamma} \phi \left( \prox [\bar{\gamma} \Phi] \left( \kappa \bar{\alpha} Z_{1} + \sqrt{\delta} \bar{\sigma} Z_{2} \right) \right) \right)^{2} \right], \\
0 & = \, 2 \bbE \left[ \phi (- \kappa Z_{1}) Z_{1} \bar{\gamma} \phi \left( \prox [\bar{\gamma} \Phi] \left( \kappa \bar{\alpha} Z_{1} + \sqrt{\delta} \bar{\sigma} Z_{2} \right) \right) \right], \\
1 - \delta & = \, 2 \bbE \left[ \phi (- \kappa Z_{1}) \prox [\bar{\gamma} \Phi]' \left( \kappa \bar{\alpha} Z_{1} + \sqrt{\delta} \bar{\sigma} Z_{2} \right) \right].
\end{array} \right.
\end{split}
\end{equation}
Here $\Phi (\cdot) \coloneqq \int_{-\infty}^{\cdot} \phi (t) \diff t$ is the anti-derivative of $\phi$, $\kappa \coloneqq \bbE \sfb^{2} \equiv p^{-1} \sum_{j = 1}^{p} \bbE [\beta_{j}^{2}]$ and $\prox$ denotes the proximal operator with $\prox [h]' (\cdot)$ being the first derivative of $\prox [h] (\cdot)$. Since $\kappa$ is unknown, two different methods have been proposed: \texttt{ProbeFrontier} by \citet{sur2019modernb} leveraging and \texttt{SLOE} by \citet{yadlowsky2021sloe} using a reparameterization of \eqref{fixed point equation} in terms of $p^{-1} \sum_{j = 1}^{p} \hat{\beta}_{\MLE, j}^{2}$ instead of $\kappa$. The proposed MoM-based method evidently offers another approach to estimating $\kappa$ from data. This is also the problem that \citet{sawaya2023moment} try to solve. However, as mentioned in the Introduction, the approach in \citet{sawaya2023moment} only applies to asymmetric link functions because it does not rely on $\bSigma$.
\end{remark}

\section{The Case of Unknown Population Covariance}
\label{sec:unknown}

In this section, we demonstrate that knowing $\bSigma$ is not essential to the proposed MoM method in the previous section. To ease presentation, we focus on the case of knowing $\bmu \equiv \bm{0}$. We will discuss briefly in Appendix~\ref{app:unknown complete} for completeness how to estimate $\bSigma$ in the general case where both $\bmu$ and $\bSigma$ are unknown. In the \href{https://github.com/cxy0714/Method-of-Moments-Inference-for-GLMs}{accompanying GitHub repository}, we also implement the proposed method for this more general scenario.

\subsection{A Gaussian-centric approach when \texorpdfstring{$p < n$}{well}}
\label{sec:unknown well-posed}

As mentioned in the Introduction, our parameter identification and estimation strategies rely on the knowledge of $\bSigma$ under the proportional asymptotic regime. In this section, we consider relaxing this assumption under Gaussian designs. To simplify our argument, we assume that $n$ is even and partition all $n$ samples into two equal-sized parts $I_{1} \coloneqq [n / 2]$ and $I_{2} \coloneqq [(n / 2 + 1):2n]$. We use the second partition $I_{2}$ to estimate $\bSigma$ by the following rescaled sample Gram matrix
\begin{align*}
\tilde{\bSigma} \coloneqq \frac{1}{\frac{n}{2} - p - 1} \sum_{j \in I_{2}} \bX_{j} \bX_{j}^{\top}.
\end{align*}
We conjecture that it is possible to attain $\sqrt{n}$-consistency without using sample splitting to estimate $\bSigma$ \citep{liu2023hoif}, but we decide to leave the analysis to a future work.

When $\bSigma$ is unknown, we propose to estimate $\gamma_{\beta}^{2}$ as in \eqref{key estimator}, except that $\hat{m}_{\bX Y, 2}$ is replaced by
\begin{equation}
\label{alternative}
\hat{m}_{\bX Y, 2} \coloneqq \frac{1}{\frac{n}{2} (\frac{n}{2} - 1)} \sum_{i_1 \neq i_2 \in I_{1}} Y_{i_1} \bX_{i_1}^{\top} \tilde{\bSigma}^{-1} \bX_{i_2} Y_{i_2}.
\end{equation}
It is worth mentioning that, \textit{only in this Section and Appendix~\ref{app:unknown GLMs}}, $\hat{m}_{\bX Y, 2}$ takes the above form.

It is easy to see that $\hat{m}_{\bX Y, 2}$ is an unbiased estimator of $m_{\bX Y, 2}$. Moreover, we have the following parallel results to Proposition~\ref{prop:glm_rate}, the proof of which can be found in Appendix~\ref{app:unknown GLMs}.

\begin{proposition}
\label{prop:unknown}
Under Assumptions of Lemma~\ref{lem:glm mean zero}, when $p + 3 < n / 2$ and $p / n \rightarrow c$ for some fixed $c < 1$, the following hold:
\begin{align*}
\sqrt{n} (\hat{\gamma}_{\beta}^{2} - \gamma_{\beta}^{2}) = O_{\bbP} (1), \text{ and for any $j = 1, \ldots, p$, } \sqrt{n} (\hat{\beta}_j - \beta_j) = O_{\bbP} (1).
\end{align*}
\end{proposition}

The additional assumption $p + 3 < n / 2$ is a result of computing the covariance between any two elements of a random matrix drawn from the Inverse-Wishart distribution; see Lemma~\ref{lem:Wishart} for more details. It is worth noting that Proposition~\ref{prop:unknown} can also be generalized beyond Gaussian designs by using universality-type results on (inverse) sample Gram matrices, e.g. related advances in \citet{derezinski2021sparse} and \citet{derezinski2022unbiased}.

\begin{remark}
\label{rem:unknown linear}
Under linear models, one can in fact adapt the estimators of \citet{guo2022moderate} for unknown $\bSigma$ by leveraging the linearity structure; for details, see Appendix~\ref{app:unknown linear}.

Furthermore, when $\bSigma$ is unknown, without estimating $\bSigma$, one can still test the null hypothesis $H_{0}: \bbeta = \bm{0}$ by simply estimating the quantity $\bbE [Y \bX^{\top}] \bbE [\bX Y] = \bbeta^{\top} \bSigma^{2} \bbeta$, which equals zero if and only if $H_{0}$ is true, using a $U$-statistic $\bbU_{n, 2} (Y_{1} \bX_{1}^{\top} \bX_{2} Y_{2})$ without the knowledge of $\bSigma$.
\end{remark}

\subsection{An alternative approach that works when \texorpdfstring{$p \geq n$}{ill}}
\label{sec:unknown ill-posed}

The approach considered in the previous section no longer applies to the case of $p \geq n$ (or more precisely, $p \geq n / 2$ when sample splitting is employed), as the sample Gram matrix is not invertible. Here we discuss an alternative approach using Higher-Order $U$-statistics, first considered in \citet{kong2018estimating}.

The idea is to approximate $m_{\bX Y, 2} = \bbE [Y \bX^{\top}] \bSigma^{-1} \bbE [\bX Y]$, the moment with reciprocal, by Chebyshev polynomials \citep{devore1993constructive}:
\begin{align*}
\tilde{m}_{\bX Y, 2, J} \coloneqq \sum_{l = 0}^{J} c_{l} m_{\bX Y, 2}^{(l)}, \text{ where } m_{\bX Y, 2}^{(l)} \coloneqq \bbE [Y \bX^{\top}] \bSigma^{l} \bbE [\bX Y],
\end{align*}
up to order $J$ to be specified later. How to specify the Chebyshev polynomial coefficients $\{c_{l}, l = 0, \cdots, J\}$ will be deferred to Appendix~\ref{app:ill-posed}. Note that $m_{\bX Y, 2}^{(l)}$ can be unbiasedly estimated by the following $(l + 2)$-th order $U$-statistic:
\begin{align*}
\hat{m}_{\bX Y, 2}^{(l)} \coloneqq \frac{(n - (l + 2))!}{n!} \sum_{1 \leq i_{1} \neq \cdots \neq i_{l + 2} \leq n} Y_{i_{1}} \bX_{i_{1}}^{\top} \left\{ \prod_{s = 3}^{l + 2} \bX_{i_{s}} \bX_{i_{s}}^{\top} \right\} \bX_{i_{2}} Y_{i_{2}}
\end{align*}
with variance of order $\frac{1}{n} \left( \frac{p}{n} \right)^{l + 1}$. Hence we consider the following estimator of $\gamma_{\beta}^{2}$:
\begin{align*}
\hat{m}_{\bX Y, 2, J (n)} \coloneqq \sum_{l = 0}^{J (n)} c_{l} \hat{m}_{\bX Y, 2}^{(l)},
\end{align*}
where $J (n) \asymp (\log n)^{c}$ for some $c$ strictly less than 1. One could approximate the other involved moments and define $\hat{\gamma}_{\beta, J (n)}^{2}$ and $\hat{\beta}_{j, J (n)}$ for $j = 1, \cdots, p$ similarly.

This alternative approach based on higher-order $U$-statistics also works in principle when $p < n$, but has inflated variance compared to the proposal in the previous section solely based on 2nd-order $U$-statistics. When $p \geq n$, $\hat{\gamma}_{\beta, J (n)}^{2}$ and $\hat{\beta}_{j, J (n)}$ are still consistent for $\gamma_{\beta}^{2}$ and $\beta_{j}$ as shown in the Proposition below. The proof is deferred to Appendix~\ref{app:ill-posed}. Here we only record the consistency without providing the concrete convergence rate. We conjecture that the tight convergence rate should be slower than $n^{-1 / 2}$ but a more precise characterization demands a much more careful control of the variance of higher-order $U$-statistics and is thus beyond the scope of this article.

\begin{proposition}
\label{prop:ill posed}
Under Assumptions of Lemma~\ref{lem:glm mean zero} or Lemma~\ref{lem:glm mean zero, universality}, the following hold:
\begin{align*}
& \hat{\gamma}_{\beta, J (n)}^{2} - \gamma_{\beta}^{2} = o_{\bbP} (1), \text{ and for any $j = 1, \ldots, p$, } \hat{\beta}_{j, J (n)} - \beta_j = o_{\bbP} (1).
\end{align*}
\end{proposition}


\begin{remark}
\label{rem:hoif}
$\hat{m}_{\bX Y, 2}$ defined in \eqref{alternative} with the prefactor $(\frac{n}{2} - p - 1) / \frac{n}{2}$ removed is nothing but the empirical second-order influence function of the functional $\bbE [Y \bX^{\top}] \bSigma^{-1} \bbE [\bX Y]$, to our knowledge first developed in \citet{liu2017semiparametric}. In \citet{liu2017semiparametric}, however, they did not make distributional assumptions on $\bX$ such as Assumption~\ref{as:normal mean zero} or even Assumption~\ref{as:normal}, resulting in more complicated higher-order $U$-statistics for correcting the bias due to estimating $\bSigma$ by $\tilde{\bSigma}$. In particular, the order can diverge to infinity as $n \rightarrow \infty$. Gaussian designs significantly simplify the diverging-order $U$-statistic estimator to a rescaled second-order $U$-statistic. An interesting question to further explore is which other types of assumptions on $\bX$, e.g. right rotationally invariant designs \citep{li2023spectrum}, can also reduce higher-order $U$-statistics to simpler second-order $U$-statistics.

When $p \geq n$, here we follow the higher-order $U$-statistic construction of \citet{kong2018estimating} to avoid estimating $\bSigma^{-1}$. \citet{kong2018estimating} also developed a polynomial-time algorithm for computing higher-order $U$-statistics, but only when the $U$-statistic kernels are symmetric. When $\bmu$ is unknown, the involved $U$-statistics have asymmetric kernels (see Section~\ref{sec:mu} and Appendix~\ref{app:unknown complete}), making the computation a much more challenging problem that we plan to address in a separate paper. We mention in passing that \citet{robins2016technical} also constructed higher-order $U$-statistics, with order diverging at rate $\sqrt{\log n}$, to estimate quantities such as $m_{\bX Y, 2}$. Their estimators, however, first estimate the density function of $\bX$ and then estimate $\bSigma^{-1}$ by numerical integration with respect to the estimated density function. Higher-order $U$-statistics are employed to reduce the bias due to density estimation, rather than approximating $\bSigma^{-1}$ by polynomials. We leave the study of statistical implications of this subtle difference to a future paper.
\end{remark}

\section{Inference in Observational Studies}
\label{sec:obs}

In this section, we appeal to the discussions in previous parts to develop estimators for some popular quantities arising in the context of observational studies. We only present theoretical results when the covariates under study have a Gaussian distribution with known covariance $\bSigma$. Results with unknown $\bSigma$ can be derived similarly by following the arguments in Section~\ref{sec:unknown}. The universality of the proposed procedure can be established by analogous arguments in the proof of Lemma~\ref{lem:mu, universality}, and hence omitted to avoid repetition and simplify exposition. The examples we discuss here are popular members of what is now known as the class of Doubly-Robust Functionals \citep{robins2008higher, rotnitzky2021characterization, chernozhukov2022locally} where in high dimensional instances of the problem practitioners often aim to model the two high dimensional relevant nuisance parameters of the underlying model through flexible GLMs. Although we consider specific and most popular examples in this class, a potential future research avenue is to extend the results in this section to the entire class of Doubly-Robust Functionals characterized in \citet{rotnitzky2021characterization}.

We divide our discussions into three subsections: estimation of causal effects of binary treatment in linear structural models, estimating quantities under missing data, and estimating generalized covariance measure (or equivalently expected conditional covariance \citep{liu2024assumption}) that have gained popularity in recent literature on conditional independence testing \citep{shah2020hardness, christgau2023nonparametric}. Throughout this section, we assume that we have access to $n$ i.i.d. copies of the triples $(\bX_{i}, A_{i}, Y_{i})_{i = 1}^{n} \overset{\rm i.i.d.}{\sim} \bbP$ where $\bX$ is the covariates or the design matrix, except for Section~\ref{sec:MAR} with a minor modification. The approaches considered in Section~\ref{sec:unknown} to estimate $\bSigma^{-1}$ from data apply to all examples here without any conceptual challenges. Consequently, we only present results assuming that $\bSigma$ is known.

\subsection{Causal Effect of a Binary Treatment under Linear Structural Causal Models} 
\label{sec:CE}

First, we consider the task of estimating the causal effect of a binary treatment $A \in \{0, 1\}$ on an outcome $Y$ in a linear structural model. Our goal is to estimate the parameter $\psi$ in the following data generating model with the outcome regression a linear model and the propensity score a GLM with link $\phi$.
\begin{equation}
\label{CE} \tag{$\CE$}
Y = \psi \cdot A + \bbeta^{\top} \bX + \varepsilon, \quad A | \bX \sim \mathrm{Ber} \left( \phi (\balpha^{\top} \bX) \right),
\end{equation}
where $\varepsilon$ has mean zero and bounded second moment given $\bX, A$.

For convenience, we let $\bm{\mu}_{1} \coloneqq \bbE [\bX A]$, $\lambda_{\alpha, 1} \coloneqq \balpha^{\top} \bm{\mu}_{1}$ and $\lambda_{\beta, 1} \coloneqq \bbeta^{\top} \bm{\mu}_{1}$. The following lemma characterizes the moment equations under Model~\ref{CE}, and the derivation can be found in Appendix~\ref{app:identification}.

\begin{lemma}
\label{lem:CE}
Under Model~\ref{CE}, Assumptions~\ref{as:normal},~\ref{as:bounded}{\rm(1)} on both $\balpha$ and $\bbeta$, and~\ref{as:var-cov}{\rm(1)} on both $A | \bX$ and $Y | \bX, A$, the following system of moment equations holds:
\begin{subequations}
\label{CE chain}
\begin{align}
& m_{A} \coloneqq \bbE [A], \\
& m_{Y} \coloneqq \bbE [Y] = \psi \cdot m_{A} + \lambda_{\beta}, \\
& m_{AY} \coloneqq \bbE [A Y] = (\psi + \lambda_{\beta}) \cdot m_{A} + \gamma_{\alpha, \beta} \cdot \f_{1} (\lambda_{\alpha}, \gamma_{\alpha}^{2}), \\
& m_{\bX, 2} \coloneqq \bbE [\bX^{\top}] \bSigma^{-1} \bbE [\bX] = \bm{\mu}^{\top} \bSigma^{-1} \bm{\mu}, \\
& m_{\bX A, \bX} \coloneqq \bbE [A \bX^{\top}] \bSigma^{-1} \bbE [\bX] = \bm{\mu}_{1}^{\top} \bSigma^{-1} \bm{\mu}, \\
& m_{\bX A, 2} \coloneqq \bbE [A \bX^{\top}] \bSigma^{-1} \bbE [\bX A] = m_{A} \cdot m_{\bX A, \bX} + \f_{1} (\lambda_{\alpha}, \gamma_{\alpha}^{2}) \cdot \lambda_{\alpha, 1}, \\
& m_{\bX A, \bX Y} \coloneqq \bbE [A \bX^{\top}] \bSigma^{-1} \bbE [\bX Y] = \psi \cdot m_{\bX A, 2} + \lambda_{\beta, 1}, \\
& m_{\bX A Y, \bX A} \coloneqq \bbE [Y A \bX^{\top}] \bSigma^{-1} \bbE [\bX A] = \psi \cdot m_{\bX A, 2} + m_{A} \cdot \lambda_{\beta, 1} + \lambda_{\beta} \cdot m_{A} \cdot m_{\bX A, \bX} \\
& + \f_{1} (\lambda_{\alpha}, \gamma_{\alpha}^{2}) \cdot \left( \lambda_{\beta} \cdot \lambda_{\alpha, 1} + \gamma_{\alpha, \beta} \cdot m_{\bX, \bX A} \right) + \f_{2} (\lambda_{\alpha}, \gamma_{\alpha}^{2}) \cdot \gamma_{\alpha, \beta} \cdot \lambda_{\alpha, 1}, \nonumber
\end{align}
\end{subequations}
where the definition of $\f_{k}$ appears in the statement of Lemma~\ref{lem:mu}.

In addition, denote the forward map induced by the RHS of the above system as
\begin{align*}
\Psi_{\CE} \coloneqq & \left( \Psi_{\CE, 1}, \cdots, \Psi_{\CE, 8} \right): \left( \psi, \lambda_{\alpha}, \gamma_{\alpha}^{2}, \lambda_{\beta}, \gamma_{\alpha, \beta}, \lambda_{\alpha, 1}, \lambda_{\beta, 1} \right) \\
& \mapsto \left( m_{A}, m_{Y}, m_{A Y}, m_{\bX, \bX A}, m_{\bX A, 2}, m_{\bX A, \bX Y}, m_{\bX A Y, \bX A} \right).
\end{align*}
In particular, $\Psi_{\CE}$ is a diffeomorphism. As a consequence, $\psi$, together with $\left( \lambda_{\alpha}, \gamma_{\alpha}^{2}, \lambda_{\beta}, \gamma_{\alpha, \beta}, \lambda_{\alpha, 1}, \lambda_{\beta, 1} \right)$, is identifiable from the above system of moment equations.
\end{lemma}
The proof of the above lemma can be found in Appendix~\ref{app:CE}. As a direct corollary of Lemma~\ref{lem:CE}, we can construct $\sqrt{n}$-consistent and CAN estimator of $\psi$ as follows.
\begin{align*}
\hat{\psi} = \Psi_{\CE, 1}^{-1} \left( \hat{m}_{A}, \hat{m}_{Y}, \hat{m}_{A Y}, \hat{m}_{\bX, \bX A}, \hat{m}_{\bX A, 2}, \hat{m}_{\bX A, \bX Y}, \hat{m}_{\bX A Y, \bX A} \right)
\end{align*}
where all the moment estimators are constructed in a similar fashion to those in \eqref{GLM moments estimators}.

\begin{theorem}
\label{thm:CE}
Under Model~\ref{CE}, Assumptions~\ref{as:normal},~\ref{as:bounded}{\rm(1)} on both $\balpha$ and $\bbeta$, and~\ref{as:var-cov}{\rm(1)} on both $A | \bX$ and $Y | \bX, A$, we have
\begin{align*}
\sqrt{n} (\hat{\psi} - \psi) = O_{\bbP} (1).
\end{align*}
Alternatively, under Model~\ref{CE} with $\varepsilon$ being independent from $\bX, A$, Assumptions~\ref{as:normal},~\ref{as:bounded} on both $\balpha$ and $\bbeta$, and~\ref{as:var-cov} on both $A | \bX$ and $Y | \bX, A$, if we further assume that $\langle \bv_{1}, \bv_{2} \rangle_{f (\bSigma)}$ converges to some nontrivial limit for $f (\bSigma) = \bSigma^{-1}, \bSigma, \bSigma^2, \bSigma^3$ for $\bv_{1}, \bv_{2} \in \{\bm{\mu}, \balpha, \bbeta\}$, we have
\begin{align*}
\sqrt{n} (\hat{\psi} - \psi) \overset{\calL}{\rightarrow} \N (0, \nu^{2})
\end{align*}
for some constant $\nu^{2} > 0$.
\end{theorem}
To our knowledge, the above is the first result for CAN estimation for average treatment effect type quantities under the proportional asymptotic regime where $p > n$ is allowed. In this regard, \cite{jiang2025new, yadlowsky2022explaining} operate under $p < n$ regime, without knowing $\bSigma$, owing to their reliance on ordinary least squares type estimator for the outcome regression. Our results close this gap in the literature under the knowledge of $\bSigma$. The proofs of CAN for all the examples in Section~\ref{sec:obs}, including Theorem~\ref{thm:CE}, can be found in Appendix~\ref{app:CLT}. The CAN property of $\hat{\psi}$ is a corollary of the CAN property established for the estimators of the linear and quadratic forms in Theorem~\ref{thm:GLM CLT}.

\subsection{Mean Estimation with Missing Data under Missing-At-Random (MAR)}
\label{sec:MAR}

In this section, we take $A \in \{0, 1\}$ and only observe $Y$ if $A = 1$. Our goal is to estimate $\psi \coloneqq \bbE [Y]$. For $\psi$ to be identifiable from the observed data, the missing data mechanism is assumed to be Missing-At-Random (MAR). More specifically, we assume that $\bbP$ is defined and parameterized as follows:
\begin{equation}
\label{MAR} \tag{$\mathsf{MAR}$}
A | \bX = \bx \sim \mathrm{Ber} \left( \eta (\balpha^{\top} \bx) \right) \text{ with } \eta (\balpha^{\top} \cdot) \in (\ubar{c}, \bar{c}), \, \, Y = \bbeta^{\top} \bX + \varepsilon,
\end{equation}
where $\varepsilon$ has mean zero and bounded second moment and is independent of $A$ and $\bX$ and $\ubar{c}, \bar{c}$ are two universal constants satisfying $0 < \ubar{c} < \bar{c} < 1$. As in \citet{celentano2023challenges}, we assume that only $\bSigma$ is known. Under Model~\ref{MAR},
\begin{align*}
\psi \equiv \bbeta^{\top} \bm{\mu} \equiv \bbE \left[ \frac{A Y}{\eta (\balpha^{\top} \bX)} \right].
\end{align*}
This problem is isomorphic to that of estimating treatment specific mean from observational studies under no unmeasured confounding.

Let $\g_{j} (t) \coloneqq \bbE [\eta^{(j)} (Z)]$, where $Z \sim \N (\lambda_{\alpha}, \gamma_{\alpha}^{2})$ with $\lambda_{\alpha} \coloneqq \balpha^{\top} \bm{\mu}$ and $\gamma_{\alpha}^{2} \coloneqq \Vert \balpha \Vert_{\bSigma}^{2}$. We also define $\gamma_{\alpha, \beta} \coloneqq \balpha^{\top} \bSigma \bbeta = \langle \balpha, \bbeta \rangle_{\bSigma}$. The following lemma characterizes the moment equations under Model~\ref{MAR}, and the derivation can again be found in Appendix~\ref{app:identification}.
\begin{lemma}
\label{lem:MAR}
Under Model~\ref{MAR}, Assumptions~\ref{as:normal},~\ref{as:bounded}{\rm(1)} on both $\balpha$ and $\bbeta$, and~\ref{as:var-cov}{\rm(1)} on both $A | \bX$ and $Y | \bX$, the following system of moment equations holds:
\begin{subequations}
\label{MAR chain}
\begin{align}
& m_{A} \coloneqq \bbE [A], \\
& m_{\bX, 2} \coloneqq \bbE [\bX^{\top}] \bSigma^{-1} \bbE [\bX] = \bm{\mu}^{\top} \bSigma^{-1} \bm{\mu}, \\
& m_{\bX A, \bX} \coloneqq \bbE [A \bX^{\top}] \bSigma^{-1} \bm{\mu} = m_{A} \cdot m_{\bX, 2} + \f_{1} (\unknowna{\lambda_{\alpha}}, \unknownb{\gamma_{\alpha}^{2}}) \cdot \unknowna{\lambda_{\alpha}} \\
& m_{\bX A, 2} \coloneqq \bbE [A \bX^{\top}] \bSigma^{-1} \bbE [\bX A] = m_{A}^{2} \cdot m_{\bX, 2} + \f_{1}^{2} (\unknowna{\lambda_{\alpha}}, \unknownb{\gamma_{\alpha}^{2}}) \cdot \unknownb{\gamma_{\alpha}^{2}} + 2 \cdot m_{A} \cdot \f_{1} (\unknowna{\lambda_{\alpha}}, \unknownb{\gamma_{\alpha}^{2}}) \cdot \unknowna{\lambda_{\alpha}} \\
& m_{A Y} \coloneqq \bbE [A Y] = m_{A} \cdot \unknownc{\psi} + \f_{1} (\unknowna{\lambda_{\alpha}}, \unknownb{\gamma_{\alpha}^{2}}) \cdot \unknownd{\gamma_{\alpha, \beta}}, \label{mar e} \\
& m_{\bX A Y, \bX} \coloneqq \bbE [Y A \bX^{\top}] \bSigma^{-1} \bm{\mu} = (m_{A} + m_{\bX A, \bX}) \cdot \unknownc{\psi} + \left\{ m_{\bX, 2} \cdot \f_{1} (\unknowna{\lambda_{\alpha}}, \unknownb{\gamma_{\alpha}^{2}}) + \f_{2} (\unknowna{\lambda_{\alpha}}, \unknownb{\gamma_{\alpha}^{2}}) \cdot \unknowna{\lambda_{\alpha}} \right\} \cdot \unknownd{\gamma_{\alpha, \beta}}, \label{mar f}
\end{align}
\end{subequations}
where the definition of $\f_{k}$ appears in the statement of Lemma~\ref{lem:mu}.

In addition, denote the forward map induced by the RHS of the above system as
\begin{align*}
\Psi_{\MAR} \coloneqq & (\Psi_{\MAR, 1}, \cdots, \Psi_{\MAR, 6}): (\psi, \lambda_{\alpha}, \gamma_{\alpha}^{2}, \gamma_{\alpha, \beta}) \\
& \mapsto (m_{A}, m_{\bX, 2}, m_{\bX A, \bX}, m_{\bX A, 2}, m_{AY}, m_{\bX A Y, \bX}).
\end{align*}
In particular, $\Psi_{\MAR}$ is a diffeomorphism. As a consequence, $\psi$, together with $\left( \lambda_{\alpha}, \gamma_{\alpha}^{2}, \gamma_{\alpha, \beta} \right)$, is identifiable from the above system of moment equations.
\end{lemma}

\begin{remark}
The system of moment equations for identifying $\psi$ is not unique. For example, one could also replace $m_{\bX A Y, \bX}$ by $m_{\bX A Y, \bX A}$, which leads to the following identity:
\begin{equation}
\label{alt}\tag{11f'}
\begin{split}
& \ m_{\bX A Y, \bX A} \coloneqq \bbE [A Y \bX^{\top}] \bSigma^{-1} \bbE [\bX A] \\
= & \ \left\{ m_{A}^{2} \cdot m_{\bX, 2} + 2 m_{A} \cdot \f_{1} (\unknowna{\lambda_{\alpha}}, \unknownb{\gamma_{\alpha}^{2}}) \cdot \unknowna{\lambda_{\alpha}} + m_{A}^{2} + \f_{1}^{2} (\unknowna{\lambda_{\alpha}}, \unknownb{\gamma_{\alpha}^{2}}) \cdot \unknownb{\gamma_{\alpha}^{2}} \right\} \cdot \unknownc{\psi} \\
& + \left\{ \begin{array}{c}
m_{A} \cdot \left( \f_{1} (\unknowna{\lambda_{\alpha}}, \unknownb{\gamma_{\alpha}^{2}}) \cdot m_{\bX, 2} + \f_{2} (\unknowna{\lambda_{\alpha}}, \unknownb{\gamma_{\alpha}^{2}}) \cdot \unknowna{\lambda_{\alpha}} \right) \\
+ \ \f_{1}^{2} (\unknowna{\lambda_{\alpha}}, \unknownb{\gamma_{\alpha}^{2}}) \cdot \unknowna{\lambda_{\alpha}} + \f_{1} (\unknowna{\lambda_{\alpha}}, \unknownb{\gamma_{\alpha}^{2}}) \f_{2} (\unknowna{\lambda_{\alpha}}, \unknownb{\gamma_{\alpha}^{2}}) \cdot \unknownb{\gamma_{\alpha}^{2}}
\end{array} \right\} \cdot \unknownd{\gamma_{\alpha, \beta}}.
\end{split}
\end{equation}
Combining \eqref{alt} with either \eqref{mar e} or \eqref{mar f} identifies $\psi$ by solving the corresponding two linear equations. Which combination leads to more efficient estimators of $\psi$ is left for future work to study.

If one is also interested in $\var (Y)$, the system \eqref{MAR chain} can be further augmented by adding the moment equation for $m_{\bX A Y, 2} \coloneqq \bbE [Y A \bX^{\top}] \bSigma^{-1} \bbE [\bX A Y]$.
\end{remark}

The proof of the above lemma can be found in Appendix~\ref{app:MAR}. As a direct corollary of Lemma~\ref{lem:MAR}, we can construct $\sqrt{n}$-consistent and CAN estimator of $\psi$ as follows.
\begin{align*}
\hat{\psi} = \Psi_{\MAR, 1}^{-1} \left( \hat{m}_{A}, \hat{m}_{\bX, 2}, \hat{m}_{\bX A, \bX}, \hat{m}_{\bX A, 2}, \hat{m}_{AY}, \hat{m}_{\bX A Y, \bX} \right)
\end{align*}
where all the moment estimators are constructed in a similar fashion to those in \eqref{GLM moments estimators}.

\begin{theorem}
\label{thm:MAR}
Under Model~\ref{MAR}, Assumptions~\ref{as:normal},~\ref{as:bounded}{\rm(1)} on both $\balpha$ and $\bbeta$, and~\ref{as:var-cov}{\rm(1)} on both $A | \bX$ and $Y | \bX$, we have
\begin{align*}
\sqrt{n} (\hat{\psi} - \psi) = O_{\bbP} (1).
\end{align*}
Alternatively, under Model~\ref{MAR}, Assumptions~\ref{as:normal},~\ref{as:bounded} on both $\balpha$ and $\bbeta$, and~\ref{as:var-cov} on both $A | \bX$ and $Y | \bX$, if we further assume that $\langle \bv_{1}, \bv_{2} \rangle_{f (\bSigma)}$ converges to some nontrivial limit for $f (\bSigma) = \bSigma^{-1}, \bSigma, \bSigma^2, \bSigma^3$ for $\bv_{1}, \bv_{2} \in \{\bm{\mu}, \balpha, \bbeta\}$, we have
\begin{align*}
\sqrt{n} (\hat{\psi} - \psi) \overset{\calL}{\rightarrow} \N (0, \nu^{2})
\end{align*}
for some constant $\nu^{2} > 0$.
\end{theorem}

For estimating $\psi$ in Model~\ref{MAR}, as we mentioned in the Introduction, \citet{celentano2023challenges} proposed estimators that resemble debiased Lasso in the proportional asymptotic regime, under the Gaussian design with known population covariance matrix $\bSigma$. We compare the numerical performance of our estimator $\hat{\psi}$ and those of \citet{celentano2023challenges} later in Section~\ref{sec:sims mar}.

\subsection{Estimation of the Generalized Covariance Measure (GCM)}
\label{sec:GCM}

In this section, we assume that $\bbP$ is defined and parameterized as follows:
\begin{equation}
\label{GCM} \tag{$\mathsf{GCM}$}
\begin{split}
& \bbE [A | \bX = \bx] = \eta (\balpha^{\top} \bx), \bbE [Y | \bX = \bx] = \phi (\bbeta^{\top} \bx),
\end{split}
\end{equation}
where both $A$ and $Y$ have bounded second moments. We also let $\sigma_{A}^{2} (\cdot) \coloneqq \var (A | \bX = \cdot)$, $\sigma_{Y}^{2} (\cdot) \coloneqq \var (Y | \bX = \cdot)$ and $\sigma_{A, Y} (\cdot) \coloneqq \cov (A, Y | \bX = \cdot)$. \citet{shah2020hardness} proposed to test $H_{0}: Y \indep A | \bX$ by estimating the parameter $\bbE [(Y - \phi (\bbeta^{\top} \bX)) (A - \eta (\bX^{\top} \balpha))]$, often referred to as the expected conditional covariance or Generalized Covariance Measure (GCM). Without loss of generality, we take the target of inference as 
\begin{align*}
\psi \coloneqq \bbE [\eta (\balpha^{\top} \bX) \phi (\bX^{\top} \bbeta)].
\end{align*}

The following lemma characterizes the moment equations under Model~\ref{GCM}, and the derivation can again be found in Appendix~\ref{app:identification}.

\begin{lemma}
\label{lem:GCM}
Under Model~\ref{GCM}, Assumptions~\ref{as:normal},~\ref{as:bounded}{\rm(1)} on both $\balpha$ and $\bbeta$, and~\ref{as:var-cov}{\rm(1)} on both $A | \bX$ and $Y | \bX$, the following system of moment equations holds:
\begin{subequations}
\label{GCM chain}
\begin{align}
& m_{A} \coloneqq \bbE [A], \\
& m_{Y} \coloneqq \bbE [Y], \\
& m_{\bX, 2} \coloneqq \bbE [\bX^{\top}] \bSigma^{-1} \bbE [\bX] = \bm{\mu}^{\top} \bSigma^{-1} \bm{\mu}, \\
& m_{\bX A, \bX} \coloneqq \bbE [A \bX^{\top}] \bSigma^{-1} \bbE [\bX] = m_{A} \cdot m_{\bX, 2} + \g_{1} (\lambda_{\alpha}, \gamma_{\alpha}^{2}) \cdot \lambda_{\alpha}, \\
& m_{\bX Y, \bX} \coloneqq \bbE [Y \bX^{\top}] \bSigma^{-1} \bbE [\bX] = m_{Y} \cdot m_{\bX, 2} + \f_{1} (\lambda_{\beta}, \gamma_{\beta}^{2}) \cdot \lambda_{\beta}, \\
& m_{\bX A, 2} \coloneqq \bbE [A \bX^{\top}] \bSigma^{-1} \bbE [\bX A] = m_{A}^{2} \cdot m_{\bX, 2} + \g_{1}^{2} (\lambda_{\alpha}, \gamma_{\alpha}^{2}) \cdot \gamma_{\alpha}^{2} + 2 \cdot m_{A} \cdot \g_{1} (\lambda_{\alpha}, \gamma_{\alpha}^{2}) \cdot \lambda_{\alpha}, \\
& m_{\bX Y, 2} \coloneqq \bbE [Y \bX^{\top}] \bSigma^{-1} \bbE [\bX Y] = m_{Y}^{2} \cdot m_{\bX, 2} + \f_{1}^{2} (\lambda_{\beta}, \gamma_{\beta}^{2}) \cdot \gamma_{\beta}^{2} + 2 \cdot m_{Y} \cdot \f_{1} (\lambda_{\beta}, \gamma_{\beta}^{2}) \cdot \lambda_{\beta}, \\
& m_{\bX A, \bX Y} \coloneqq \bbE [A \bX^{\top}] \bSigma^{-1} \bbE [\bX Y] = m_{A} \cdot m_{Y} \cdot m_{\bX, 2} + m_{A} \cdot \f_{1} (\lambda_{\beta}, \gamma_{\beta}^{2}) \cdot \lambda_{\beta} \label{mm3} \\
& + m_{Y} \cdot \g_{1} (\lambda_{\alpha}, \gamma_{\alpha}^{2}) \cdot \lambda_{\alpha} + \g_{1} (\lambda_{\alpha}, \gamma_{\alpha}^{2}) \cdot \f_{1} (\lambda_{\beta}, \gamma_{\beta}^{2}) \cdot \gamma_{\alpha, \beta}, \nonumber
\end{align}
\end{subequations}
where the definition of $\f_{k}$ appears in the statement of Lemma~\ref{lem:mu}.

In addition, denote the forward map induced by the RHS of the above system as
\begin{align*}
\Psi_{\GCM} \coloneqq & (\Psi_{\GCM, 1}, \cdots, \Psi_{\GCM, 8}): (\lambda_{\alpha}, \gamma_{\alpha}^{2}, \lambda_{\beta}, \gamma_{\beta}^{2}, \gamma_{\alpha, \beta}) \\
& \mapsto (m_{A}, m_{Y}, m_{\bX, 2}, m_{\bX A, \bX}, m_{\bX Y, \bX}, m_{\bX A, 2}, m_{\bX Y, 2}, m_{\bX A, \bX Y}).
\end{align*}
In particular, $\Psi_{\GCM}$ is a diffeomorphism. As a consequence, $\psi$, together with $(\lambda_{\alpha}, \gamma_{\alpha}^{2}, \lambda_{\beta}, \gamma_{\beta}^{2}, \gamma_{\alpha, \beta})$, is identifiable from the above system of moment equations.
\end{lemma}

The proof of the above lemma can be found in Appendix~\ref{app:GCM}. As a direct corollary of Lemma~\ref{lem:GCM}, we can construct $\sqrt{n}$-consistent and CAN estimator of $\psi$ as follows. Let
\begin{align*}
\hat{\psi} \coloneqq \bbE [\eta (Z_{1}) \phi (Z_{2}], & \text{ where } \left( \begin{matrix}
Z_{1} \\
Z_{2}
\end{matrix} \right) \sim \N_{2} \left( \left( \begin{matrix}
\hat{\lambda}_{\alpha} \\
\hat{\lambda}_{\beta}
\end{matrix} \right), \left( \begin{matrix}
\hat{\gamma}_{\alpha}^{2} & \hat{\gamma}_{\alpha, \beta} \\
\hat{\gamma}_{\alpha, \beta} & \hat{\gamma}_{\beta}^{2}
\end{matrix} \right) \right), \text{ and} \\
(\hat{\lambda}_{\alpha}, \hat{\gamma}_{\alpha}^{2}, \hat{\lambda}_{\beta}, \hat{\gamma}_{\beta}^{2}, \hat{\gamma}_{\alpha, \beta}) & \coloneqq \Psi_{\GCM}^{-1} \left( \hat{m}_{A}, \hat{m}_{Y}, \hat{m}_{\bX, 2}, \hat{m}_{\bX A, \bX}, \hat{m}_{\bX Y, \bX}, \hat{m}_{\bX A, 2}, \hat{m}_{\bX Y, 2}, \hat{m}_{\bX A, \bX Y} \right),
\end{align*}
where all the moment estimators are constructed in a similar fashion to those in \eqref{GLM moments estimators}.

To establish $\sqrt{n}$-consistency and CAN of our proposed estimator, we also need to impose the following condition on the conditional covariance function of $A, Y$ given $\bX$.
\begin{customas}{$\mathsf{Cov}$}\leavevmode
\label{as:cov}
\begin{itemize}
\item[(1)] We assume that $\Vert \sigma_{A, Y} (\cdot) \Vert_{2}$ is bounded;
\item[(2)] We assume that the conditional covariance of $A, Y$ given $\bX$ is also a GLM with link function $\zeta$:
\begin{align*}
\sigma_{A, Y} (\bX) = \zeta (\bm{\theta}^{\top} \bX) \text{ almost surely,}
\end{align*}
such that $\zeta$ is three-times differentiable and the first to third derivatives of $\zeta$, together with $\zeta$ itself, are integrable with respect to the law of $\bX$ and the integrals are all strictly bounded by some universal constant.
\end{itemize}
\end{customas}

\begin{theorem}
\label{thm:GCM}
Under Model~\ref{GCM}, Assumptions~\ref{as:normal},~\ref{as:bounded}{\rm(1)} on both $\balpha$ and $\bbeta$,~\ref{as:var-cov}{\rm(1)} on both $A | \bX$ and $Y | \bX$, and~\ref{as:cov}\rm{(1)}, we have
\begin{align*}
\sqrt{n} (\hat{\psi} - \psi) = O_{\bbP} (1).
\end{align*}
Alternatively, under Model~\ref{GCM}, Assumptions~\ref{as:normal},~\ref{as:bounded} on both $\balpha$ and $\bbeta$,~\ref{as:var-cov} on both $A | \bX$ and $Y | \bX$, and~\ref{as:cov}, if we further assume that $\langle \bv_{1}, \bv_{2} \rangle_{f (\bSigma)}$ converges to some nontrivial limit for $f (\bSigma) = \bSigma^{-1}, \bSigma, \bSigma^2, \bSigma^3$ for $\bv_{1}, \bv_{2} \in \{\bm{\mu}, \balpha, \bbeta, \btheta\}$, we have
\begin{align*}
\sqrt{n} (\hat{\psi} - \psi) \overset{\calL}{\rightarrow} \N (0, \nu^{2})
\end{align*}
for some constant $\nu^{2} > 0$.
\end{theorem}

Since we do not assume $Y \indep A | \bX$ and GLMs are only specified for $A | \bX$ and $Y | \bX$, Assumption~\ref{as:cov}(1) is needed for $\sqrt{n}$-consistency because the variances of the moment estimators based on $U$-statistics involve terms such as $\bbE [A Y \bX]$. This is in contrast to the previous two examples. In Section~\ref{sec:CE}, GLMs are specified for variational independent components $A | \bX$ and $Y | \bX, A$ of the joint observed-data distribution, whereas in Section~\ref{sec:MAR}, we assume that $Y \indep A | \bX$.

\section{Numerical Experiments}
\label{sec:sims}

In this section, we verify our theory via extensive numerical experiments. We consider two types of problems. The first problem, studied in Section~\ref{sec:sims glms}, is on estimating linear and quadratic forms of regression coefficients $\balpha$ of a GLM between the response $A$ and baseline covariates $\bX$. The second problem, studied in Section~\ref{sec:sims mar}, is on estimating the mean of a response $Y$ subject to missingness under Model~\ref{MAR}. As mentioned, the latter problem is also isomorphic to estimating treatment specific mean from observational studies under no unmeasured confounding. All the results in the numerical experiments are based on 500 Monte Carlos. The dimension-sample size ratio $p / n$ is fixed at 1.2 in Section~\ref{sec:sims glms} and 1.25 in Section~\ref{sec:sims mar}. Appendix~\ref{app:sim} contains additional simulation results that are complementary to those reported here in the main text. We always assume that $\bSigma$ is known except that in Appendix~\ref{app:sims unknown} we conduct simulations to evaluate the performance of our proposed estimators in Section~\ref{sec:unknown well-posed} when knowing neither $\bmu$ nor $\bSigma$. The bootstrap variance estimators mentioned right after Theorem~\ref{thm:GLM CLT} are assessed in Appendix~\ref{app:bootstrap}. 

For all the histograms and normal quantile-quantile plots reported in this section and in Appendix~\ref{app:sim}, we choose the numerical experiment with $n = 5000$. The R codes for replicating the numerical experiments can be found in
\href{https://github.com/cxy0714/Method-of-Moments-Inference-for-GLMs}{this GitHub repository}.

\subsection{Linear and quadratic forms of GLMs}
\label{sec:sims glms}

We consider several different experiment settings. But across different settings, the common Data Generating Process (DGP) can be described as follows: $A | \bX \sim \mathrm{Ber} \left( \expit \left( \balpha^{\top} \bX \right) \right)$ where the goal is to estimate a single coordinate $\alpha_{1}$ and the quadratic form $\gamma_{\alpha}^{2} = \balpha^{\top} \bSigma \balpha$. Throughout this section, we set the true value of the quadratic form as $\bbE(\gamma_{\alpha}^{2}) \equiv 1$. The true value of $\alpha_{1}$ may vary across different settings. The DGPs vary in the following aspects:
\begin{itemize}
\item $\bX_{i} \overset{\rm i.i.d.}{\sim} \N_{p} (\bm{\mu}, \bSigma)$ with $\bm{\mu} \equiv \bm{0}$ and $\bSigma \equiv \bI_{p}$ for $i = 1, \cdots, n$ (Settings 1 \& 2); $\bX_{i, j} \overset{\rm i.i.d.}{\sim} \mathrm{Rad} (1 / 2)$ for $i = 1, \cdots, n$ and $j = 1, \cdots, p$ where $\mathrm{Rad}$ denotes the Rademacher distribution (Settings 3 \& 4). The two cases, though, share the same population mean and covariance matrix.
\item $\balpha$ is dense in the sense that we draw $\balpha = (\alpha_{1}, \cdots, \alpha_{p}) \overset{\rm i.i.d.}{\sim} \mathrm{Uniform} ([-\sqrt{3 / p}, \sqrt{3 / p}])$ (Settings 1 \& 3); $\balpha$ is sparse in the sense that $s = \sqrt{p}$ out of $p$ coordinates are non-zero and each non-zero coefficient equals $p^{- 1 / 4}$ (Settings 2 \& 4). Due to the similarity between results for dense and sparse configurations, we defer the figures for the sparse configuration case to Appendix~\ref{app:sim fig}.
\item We also vary the sample sizes as $n = 1000, 2000, \cdots, 5000$. We note that Setting 2 corresponds to the simulation setting described in Section 6.3 of \citet{bellec2025observable}.
\end{itemize}

In each setting, for the problem of estimating certain single coordinate, we compare our MoM-based estimators against the methods of \citet{bellec2025observable}, which debiases the initial estimator $\hat{\balpha}_{\rm init}$ of $\balpha$ using their main Theorem 4.1. We do not further compare different estimators of $\lambda_{\alpha}$ and $\gamma_{\alpha}^{2}$ as the methodology in \citet{bellec2025observable} is for linear form $\bv^{\top} \balpha$ with known direction $\bv$, not directly applicable here. For simplicity, we use ridge regression to compute $\hat{\balpha}_{\rm init}$. For the tuning parameter $\lambda$ in ridge regression, due to large running time, We choose twelve different values of $\lambda$ ranging from 0.05 to 10, equally spaced on a logarithmic scale. 

We summarize the results of the numerical experiments below. For single coordinates of $\balpha$, we arbitrarily pick $\alpha_{1}$ and $\alpha_{100}$ to present the simulation results.
\begin{itemize}
\item When the design is Gaussian (Settings 1 \& 2, Figures~\ref{fig:GLM, dense-converge, gaussian，1.2},~\ref{fig:GLM, dense-converge, gaussian，1.2, hist},~\ref{fig:GLM, sparse-converge, gaussian，1.2} and~\ref{fig:GLM, sparse-converge, gaussian，1.2, hist}), Figures~\ref{fig:GLM, dense-converge, gaussian，1.2} shows that our proposed MoM-based estimator of $\alpha_{1}$ and $\alpha_{100}$ generally have similar $\sqrt{n} \times \mathrm{bias}$, variance, and mean squared error to the estimator proposed in \citet{bellec2025observable}. Moreover, Figures~\ref{fig:GLM, dense-converge, gaussian，1.2, hist} and~\ref{fig:GLM, sparse-converge, gaussian，1.2, hist} showcase the histograms and normal quantile-quantile plots of the $U$-statistic-based moment estimators and estimators of $\lambda_{\alpha}$, $\gamma_{\alpha}^{2}$, $\alpha_{1}$ and $\alpha_{100}$ by solving the moment system \eqref{GLM chain}. It is clear from these figures that the sampling distributions of both the $U$-statistic-based moment estimators and our estimators $\alpha_{1}$ and $\gamma_{\alpha}^{2}$ are close to the Gaussian distribution, further confirming by our theoretical results on the GAN property of our proposed MoM-based estimators.

\item When the design is Rademacher (Settings 3 \& 4, Figures~\ref{fig:GLM, dense-converge, rademacher，1.2},~\ref{fig:GLM, dense-converge, rademacher，1.2, hist},~\ref{fig:GLM, sparse-converge, rademacher，1.2} and~\ref{fig:GLM, sparse-converge, rademacher，1.2, hist}), we observe similar results to those in the Gaussian settings. Therefore, under certain conditions on the design and the regression coefficients, our identification and estimation strategies based on Gaussian designs continue to be applicable -- demonstrating the universality of our proposed procedure. Interestingly, the debiased estimator also exhibits universality, which deserves a further theoretical investigation. It is worth noting that although Setting 4 concerns sparse regression coefficients, the number of non-zero coefficients is large and the values of non-zero coefficients are the same, so numerically our MoM-based estimators work still fine. In Appendix~\ref{app:sim univ}, we showcase a different simulation setting, in which only one coordinate of the coefficients is non-zero; there the results show that the Gaussian-based identification and estimation strategies no longer produce consistent estimators.
\end{itemize}

\begin{figure}[htbp]
\centering
\includegraphics[width = 0.65\textwidth]{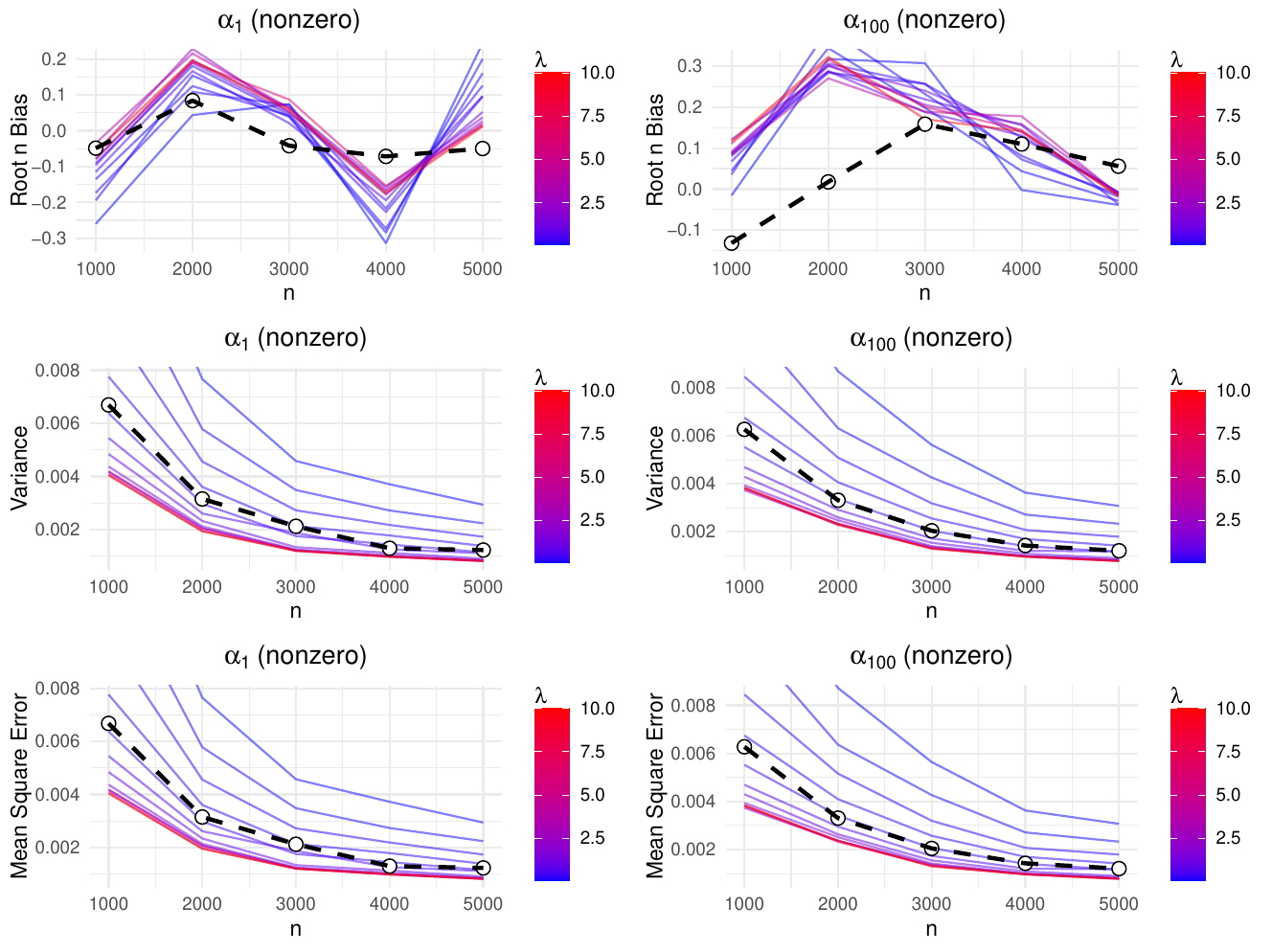}
\caption{Setting 1 of Section~\ref{sec:sims glms} (Gaussian design and dense regression coefficients).}
\label{fig:GLM, dense-converge, gaussian，1.2}
\end{figure}

\begin{figure}[htbp]
\centering
\includegraphics[width = 0.65\textwidth]{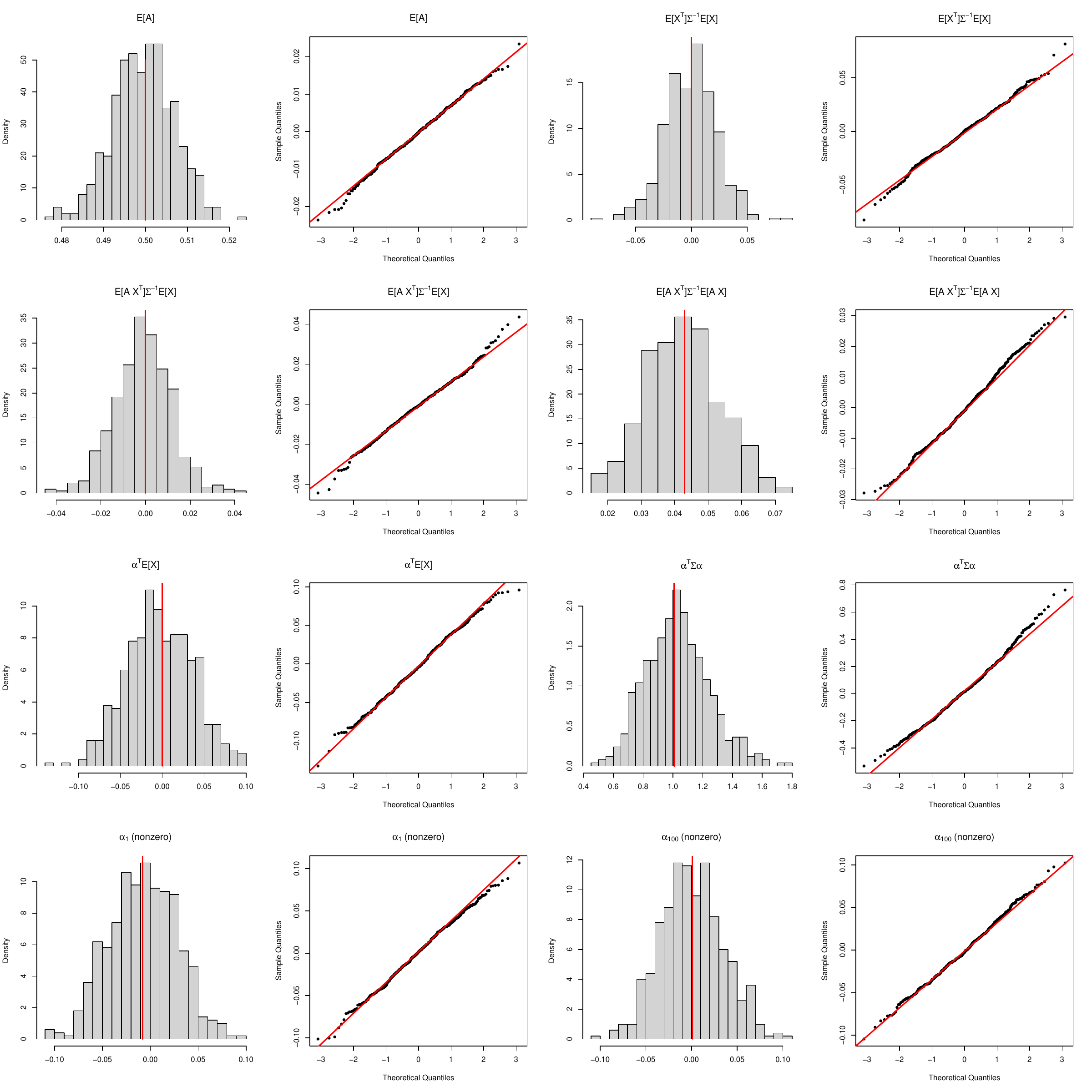}
\caption{Setting 1 of Section~\ref{sec:sims glms} (Gaussian design and dense regression coefficients): Sampling distributions of the moment estimators and the parameter estimators, over 500 Monte Carlos are displayed for the case of $n = 5000$.}
\label{fig:GLM, dense-converge, gaussian，1.2, hist}
\end{figure}

\begin{figure}[htbp]
\centering
\includegraphics[width = 0.65\textwidth]{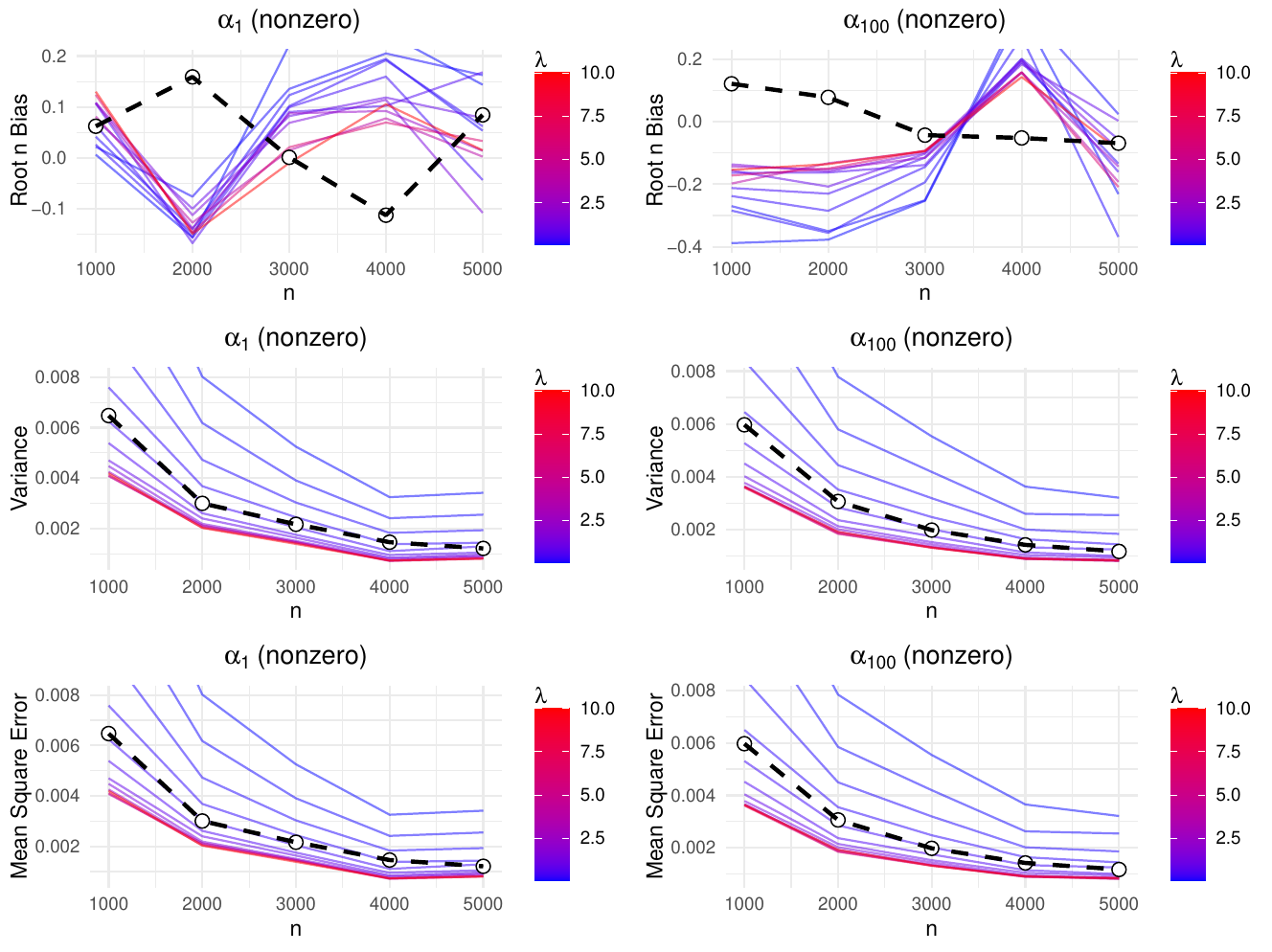}
\caption{Setting 3 of Section~\ref{sec:sims glms} (Rademacher design and dense regression coefficients).}
\label{fig:GLM, dense-converge, rademacher，1.2}
\end{figure}

\begin{figure}[htbp]
\centering
\includegraphics[width = 0.65\textwidth]{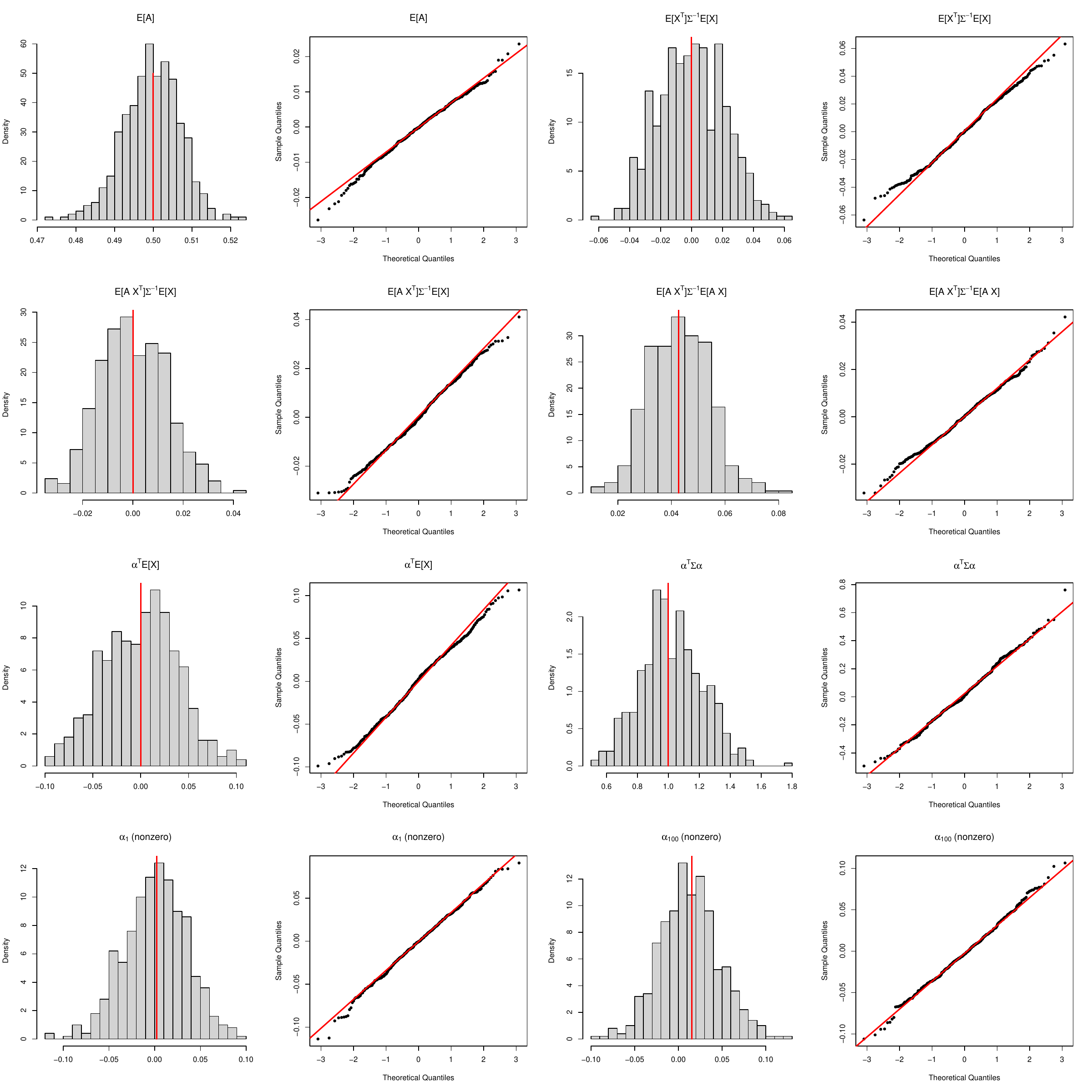}
\caption{Setting 3 of Section~\ref{sec:sims glms} (Rademacher design and dense regression coefficients): Sampling distributions of the moment estimators and the parameter estimators, over 500 Monte Carlos are displayed for the case of $n = 5000$.}
\label{fig:GLM, dense-converge, rademacher，1.2, hist}
\end{figure}

\subsection{Mean estimation with missing data under MAR}
\label{sec:sims mar}

For the problem of estimating the mean of a response $Y$ under Model~\ref{MAR}, we also consider several different settings. We first describe the common part of the DGP: the data is generated according to Model~\ref{MAR}, with $\bX_{i} \overset{\rm i.i.d.}{\sim} \N_{p} (\bm{0}, \bI_{p})$ for $i = 1, \cdots, n$. We take $\bbE [A | \bX] = 0.1 + 0.9 \cdot \expit (\bX^{\top} \balpha)$, as in \citet{celentano2023challenges}. Here the true value of the target parameter $\psi$ is 0 because $\bX$ has mean zero. We take the outcome noise $\varepsilon \sim \N (0, \sigma^{2})$ with $\sigma = 0.2$. The sample size varies from 100 to 10000. We compare our proposed MoM-based estimator $\hat{\psi}$ with the two estimators proposed in \citet{celentano2023challenges}, respectively termed as ``Oracle ASCW'' and ``Empirical SCA''. For these alternative estimators, initial estimates based on Ridge are also used, With the tuning parameter $\lambda$ taking 50 values ranging from $[0.01, 10]$ with equal steps on a logarithmic scale. We refer readers to \citet{celentano2023challenges} for the details of these two alternative estimators; here we simply implement the R code provided in the \href{https://github.com/mcelentano/Debiasing_for_missing_data}{GitHub repository} provided by the authors of \citet{celentano2023challenges}.

We also consider two different configurations for the regression coefficients $\balpha$ and $\bbeta$. In Setting 1 (Figures~\ref{fig:MAR, dense, gaussian} --~\ref{fig:MAR, parameters, dense, gaussian}), we consider dense regression coefficients, where both $\balpha$ and $\bbeta$ are drawn i.i.d. coordinate-wise from $\mathrm{Uniform} ([- \sqrt{3 / p}, \sqrt{3 / p}])$. In Appendix~\ref{app:sim mar}, we report results when changing the covariates distribution from Gaussian to $\bX_{i, j} \overset{\rm i.i.d.}{\sim} \mathrm{Rad} (1 / 2)$ for $i = 1, \cdots, n$ and $j = 1, \cdots, p$. In Setting 2 (Figures~\ref{fig:MAR, sparse, gaussian} --~\ref{fig:MAR, parameters, sparse, gaussian}), we consider sparse regression coefficients, where only $\alpha_{1}$ and $\beta_{1}$ are non-zero and are both equal to 1. Again, to save space, we defer the figures for Setting 2 to Appendix~\ref{app:sim mar}. 

As can be seen from Figures~\ref{fig:MAR, dense, gaussian} and~\ref{fig:MAR, sparse, gaussian}, the MoM-based estimators generally have lower $\sqrt{n} \times \mathrm{bias}$, variance, and mean squared error than the two estimators proposed in \citet{celentano2023challenges}, over a range of values of the tuning parameter $\lambda$, regardless of the configurations of the regression coefficients. Figures~\ref{fig:MAR, parameters, dense, gaussian} and~\ref{fig:MAR, parameters, sparse, gaussian} display the histograms and normal quantile-quantile plots of the $U$-statistic-based moment estimators and the estimators of the target parameter $\psi$, together with $\lambda_{\alpha}, \gamma_{\alpha}^{2}, \gamma_{\alpha, \beta}$ (the byproducts of the system \eqref{MAR chain}). It is clear from these figures that the sampling distributions of both the $U$-statistic-based moment estimators and of the estimators of the target parameter $\psi$, together with $\lambda_{\alpha}, \gamma_{\alpha}^{2}, \gamma_{\alpha, \beta}$ (the byproducts of the system \eqref{MAR chain}), are close to the Gaussian distribution.

\begin{figure}[htbp]
\centering
\includegraphics[width = 0.65\textwidth]{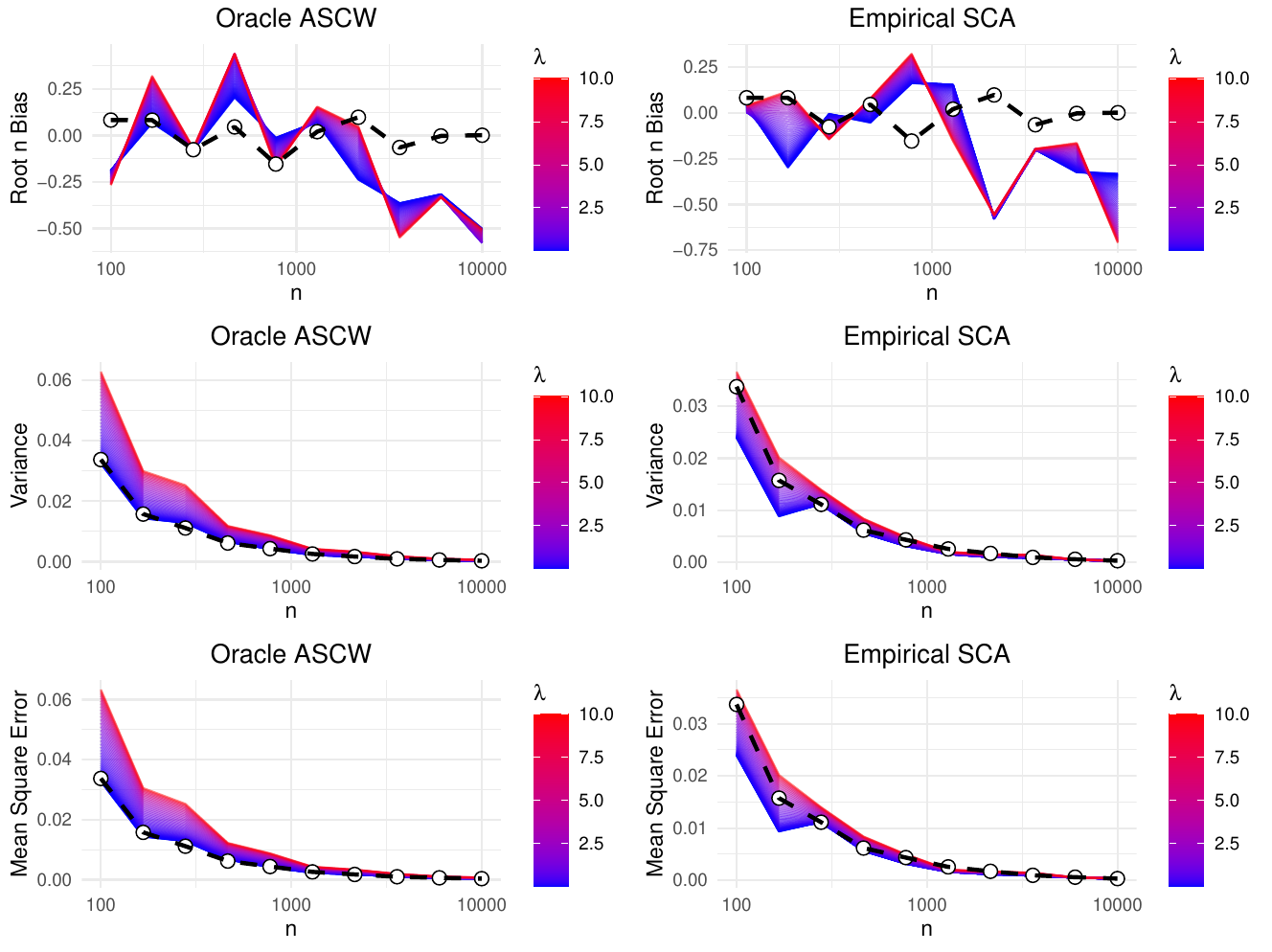}
\caption{Simulation results for Setting 1 (dense regression coefficients) in Section~\ref{sec:sims mar}. The two methods proposed in \citet{celentano2023challenges} are plotted separately in two columns of the figure, with color gradients from blue to red representing the increasing value of the tuning parameter $\lambda$. The MoM-based estimators are plotted with white circles and dashed black lines.}
\label{fig:MAR, dense, gaussian}
\end{figure}

\begin{figure}[htbp]
\centering
\includegraphics[width = 0.65\textwidth]{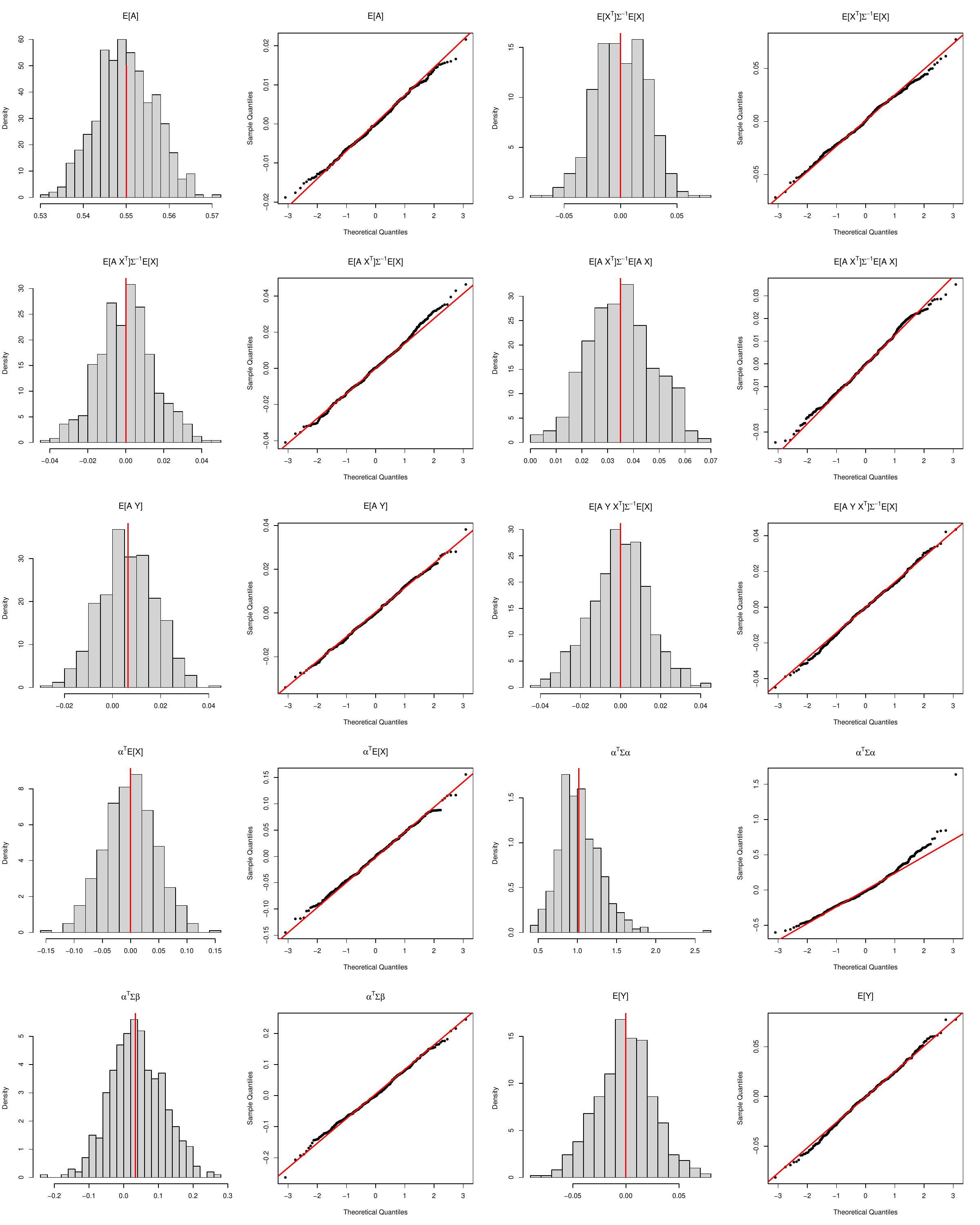}
\caption{Simulation results for Setting 1 (dense regression coefficients) in Section~\ref{sec:sims mar}: Sampling distributions of the moment estimators and the parameter estimators, over 500 Monte Carlos are displayed for the case of $n = 5000$.}
\label{fig:MAR, parameters, dense, gaussian}
\end{figure}

\section{Discussion}
\label{sec:discussion}

This paper proposes Method-of-Moments (MoM) estimators of functionals of the regression coefficients in high-dimensional GLMs under the proportional asymptotic regime, in the most part assuming Gaussian designs with known population covariance matrix $\bSigma$, following the current trend of literature. We demonstrate promising theoretical and numerical results about the MoM estimators. However, we emphasize that a more delicate comparison between our proposed approach and those based on debiasing \citep{bellec2025observable, celentano2023challenges} is warranted in a future work: for instance, which estimator has better efficiency or which estimator performs better when $\bSigma$ is unknown or when sufficient conditions (e.g. Assumption~\ref{as:beta}) for universality are violated with non-Gaussian designs. We end our article with the following discussion.

\subsection{Statistical Inference}
\label{sec:bootstrap}


A problem that is worth pursuing in future work relates to the potentially non-Gaussian limiting distribution. When $\Vert \bbeta \Vert$ is small, the corresponding $U$-statistic estimator may have positive probability of being negative and the truncation used in \eqref{key estimator} could lead to non-Gaussian limiting distributions of $\hat{\gamma}_{\beta}^{2}$. It will be an interesting problem to investigate whether the bootstrap approach also delivers asymptotically valid inference in such a setting.

\subsection{Unknown Population Covariance Matrix}

In this paper, we mostly assume that $\bSigma$ is known, an assumption ubiquitously adopted in the current literature (see Section~\ref{sec:review}). Progress can be made when $\bSigma$ is unknown, but under specific conditions; see e.g. Section~\ref{sec:unknown} and Appendix~\ref{app:unknown}. It will be of greater interest to examine the case if we weaken the assumptions on $\bSigma$ or $\bX$ imposed in Section~\ref{sec:unknown}. 

\subsection{Single-Index Models (SIM) and Model Misspecification}

Compared to GLMs, isotonic Single-Index Models (SIMs) offer a more flexible modeling option as they allow the link function to be unknown and nonparametric. To the best of our knowledge, \citet{bellec2025observable} is among the first to systematically study inference of functionals of regression coefficients in SIMs under proportional asymptotics. It is interesting to extend our framework to single-index models but it requires estimating the link function and hence higher-order moments, a complication that we plan to address in a separate paper. Another direction is to investigate if similar results for ATE still hold under the assumptions of \citet{su2023estimated}, where only the propensity score is modeled by high-dimensional logistic regression but the outcome regression model is arbitrary. Finally, whenever parametric models such as GLMs are used, model misspecification is of potential concerns. The semiparametric partially nonlinear regression parameters considered in \citet{vansteelandt2022assumption}, which reduce to functionals of regression coefficients when GLMs are correctly specified, can be another promising future avenue for our work.


\section*{Acknowledgments}
The authors would like to thank Subhabrata Sen, Pragya Sur, and Nicolas Verzelen for helpful discussions. The computations in this paper were run on the Siyuan-1 and $\pi$-2.0 cluster supported by the Center for High Performance Computing at Shanghai Jiao Tong University. Xingyu Chen and Lin Liu were supported by NSFC Grant No. 12471274. Lin Liu was also partly supported by NSFC Grant No. 12090024.

\putbib[\myreferences]
\end{bibunit}

\newpage

\appendix

\begin{appendices}
\begin{bibunit}[plainnat]

\section{Preparatory Results}
\label{app:stein}

\begin{lemma}[First- and Second-Order Stein's lemma]
\label{lem:stein}
For $\bZ \sim \N_{p} (\bm{0}, \bI)$, we have for $f: \bbR^{p} \rightarrow \bbR$ that is differentiable,
\begin{align*}
\bbE \left[ \bZ f (\bZ) \right] = \bbE \left[ \nabla f (\bZ) \right].
\end{align*}
Similarly, we have for $h: \bbR^{p} \rightarrow \bbR$ and $f: \bbR^{p} \rightarrow \bbR$ that are twice differentiable,
\begin{align*}
\bbE \left[ \left( \bZ \bZ^{\top} - \bI \right) h (\bZ) \right] & = \bbE \left[ \nabla^{2} h (\bZ) \right], \\
\bbE \left[ \left( \bZ^{\top} f (\bZ) - \div f (\bZ) \right)^{2} \right] & = \bbE \left[ \Vert f (\bZ) \Vert^{2} + \trace \left( \nabla^{2} f (\bZ) \right) \right].
\end{align*}
\end{lemma}

\begin{corollary}
\label{cor:stein}
For $\bX \sim \N_{p} (\bm{\mu}, \bSigma)$, we have, for $f: \bbR^{p} \rightarrow \bbR$ that is differentiable,
\begin{align*}
\bbE \left[ \bX f (\bX) \right] = \bm{\mu} \bbE [f (\bX)] + \bSigma \bbE [\nabla f (\bX)].
\end{align*}
In addition, we have, for $\mathbf{f}: \bbR^{p} \rightarrow \bbR^{p}$ that is coordinate-wisely differentiable, 
\begin{equation}
\label{trace}
\bbE [\bX^{\top} \mathbf{f} (\bX)] = \bm{\mu}^{\top} \bbE [\mathbf{f} (\bX)] + \trace \left[ \bSigma \bbE [\nabla \mathbf{f} (\bX)] \right].
\end{equation}
Similarly, we have for $h: \bbR^{p} \rightarrow \bbR$ that is twice differentiable,
\begin{equation}
\label{matrix}
\bbE \left[ \left( \bX \bX^{\top} - \bSigma \right) h (\bX) \right] = \bm{\mu} \bm{\mu}^{\top} \bbE [h (\bX)] + \bm{\mu} (\bSigma \bbE [\nabla h (\bX)])^{\top} + (\bSigma \bbE [\nabla h (\bX)]) \bm{\mu}^{\top} + \bSigma \bbE [\nabla^{2} h (\bX)] \bSigma.
\end{equation}
\end{corollary}

\begin{proof}
The proof for the first statement is trivial and hence omitted. Let $\bZ \sim \N_{p} (\bm{0}, \bI)$ then $\bX \overset{d}{=} \bSigma^{1 / 2} \bZ + \bm{\mu}$.
\begin{align*}
& \ \bbE \left[ \left( \bX \bX^{\top} - \bSigma \right) h (\bX) \right] \\
= & \ \bbE \left[ \left( (\bX - \bm{\mu} + \bm{\mu}) (\bX - \bm{\mu} + \bm{\mu})^{\top} - \bSigma \right) h (\bX) \right] \\
= & \ \bm{\mu} \bm{\mu}^{\top} \bbE [h (\bX)] + \bm{\mu} \bbE \left[ (\bX - \bm{\mu})^{\top} h (\bX) \right] + \bbE \left[ (\bX - \bm{\mu}) h (\bX) \right] \bm{\mu}^{\top} + \bbE \left[ \left\{ (\bX - \bm{\mu}) (\bX - \bm{\mu})^{\top} - \bSigma \right\} h (\bX) \right] \\
= & \ \bm{\mu} \bm{\mu}^{\top} \bbE [h (\bX)] + \bm{\mu} (\bSigma \bbE [\nabla h (\bX)])^{\top} + (\bSigma \bbE [\nabla h (\bX)]) \bm{\mu}^{\top} + \bSigma^{1 / 2} \bbE \left[ (\bZ \bZ^{\top} - \bI) h (\bSigma^{1 / 2} \bZ + \bm{\mu}) \right] \bSigma^{1 / 2} \\
= & \ \bm{\mu} \bm{\mu}^{\top} \bbE [h (\bX)] + \bm{\mu} (\bSigma \bbE [\nabla h (\bX)])^{\top} + (\bSigma \bbE [\nabla h (\bX)]) \bm{\mu}^{\top} + \bSigma \bbE [\nabla^{2} h (\bX)] \bSigma.
\end{align*}
\end{proof}

\begin{corollary}
\label{cor:stein 2}
For $Z \sim \N (\lambda, \gamma^{2})$, given a twice-differentiable function $f: \bbR \rightarrow \bbR$, we have
\begin{align*}
\bbE f (Z) (Z - \lambda)^{2} = \gamma^{2} \bbE f (Z) + (\gamma^{2})^{2} \bbE f'' (Z).
\end{align*}
\end{corollary}

\begin{proof}
Repeatedly applying Lemma~\ref{lem:stein}, we have
\begin{align*}
\bbE f (Z) (Z - \lambda)^{2} & = \bbE f (Z) (Z - \lambda) (Z - \lambda) \\
& = \gamma^{2} \bbE \{f' (Z) (Z - \lambda) + f (Z)\} \\
& = \gamma^{2} \bbE f (Z) + (\gamma^{2})^{2} \bbE f'' (Z).
\end{align*}
\end{proof}

As our proposed MoM-based estimators critically rely on inverting a nonlinear map (generally without an analytic form), particularly in Appendix~\ref{app:identification}, we also need the following result that is a consequence of the Inverse Function Theorem and a result believed to be due to Hadamard (Theorem 6.2.8 of \citet{krantz2002implicit}).

\begin{lemma}[Hadamard Global Inverse Function Theorem]
\label{lem:inverse}
Suppose $\calR$ and $\calM$ are smooth, simply connected open sets of $\bbR^{k}$ for some positive integer $k$ and $\Psi: \calR \rightarrow \calM$ is a twice-differentiable and proper map. If the Jacobian $\nabla \Psi$ has determinant bounded away from 0 over $\calR$, then $\Psi$ is a diffeomorphism.
\end{lemma}

\begin{proof}
By Theorem 6.2.8 of \citet{krantz2002implicit}, since $\calR$ and $\calM$ are smooth, simply connected open sets of $\bbR^{k}$ and $\Psi$ is proper and has non-singular Jacobian, $\Psi$ is a homeomorphism and hence bijective. By the Inverse Function Theorem, since the Jacobian $\nabla \Psi$ has determinant bounded away from zero, $\Psi$ is a local diffeomorphism. A bijective map that is a local diffeomorphism is a (global) diffeomorphism.
\end{proof}

Taking $x \in \calR$. $\Psi$ is a proper map if $x \rightarrow \partial \calR$ implies that $\Psi (x) \rightarrow \partial \calM$: see Definition 6.2.2 of \citet{krantz2002implicit} and discussions thereafter.

\section{Proof for the system of moment equations for GLM}
\label{app:mu}

After the preparation in the previous section, we are ready to prove the first part of Lemma~\ref{lem:mu}. We only need to verify \eqref{GLM m3}, \eqref{GLM m4}, and \eqref{GLM m6} of \eqref{GLM chain}.

Regarding \eqref{GLM m3}, we can simply apply Corollary~\ref{cor:stein} and obtain:
\begin{align*}
m_{\bX Y, \bX} & = \bbE [\phi (\bbeta^{\top} \bX) \bX^{\top}] \bSigma^{-1} \bbE [\bX] = \bbE [\phi (\bbeta^{\top} \bX)] \bm{\mu}^{\top} \bSigma^{-1} \bm{\mu} + \bbE [\phi' (\bbeta^{\top} \bX)] \bbeta^{\top} \bSigma \bSigma^{-1} \bm{\mu} \\
& = \bbE [Y] \bm{\mu}^{\top} \bSigma^{-1} \bm{\mu} + \bbE [\phi' (\bbeta^{\top} \bX)] \bbeta^{\top} \bm{\mu} \\
& = m_{Y} \cdot m_{\bX, 2} + \f_{1} (\lambda_{\beta}, \gamma_{\beta}^{2}) \cdot \lambda_{\beta}.
\end{align*}
The derivations of \eqref{GLM m4} and \eqref{GLM m6} follow the same strategy, hence omitted.

\section{Derivations of the system of moment equations for the examples in Section~\ref{sec:obs}}

This section is devoted to deriving the system of moment equations for the three examples considered in Section~\ref{sec:obs}.

\subsection{Derivation of \texorpdfstring{\eqref{CE chain}}{CE chain}}
\label{app:CE}

In this section, we derive the system of moment equations \eqref{CE chain} in Lemma~\ref{lem:CE}.

\begin{proof}[Derivation of \eqref{CE chain}]
The system of moment equations in \eqref{CE chain} follows directly from Corollary~\ref{cor:stein} and some elementary calculations. We only show the following three identities as the others are trivial:
\begin{enumerate}
\item Derivation related to $\bbE [A \bX^{\top}] \bSigma^{-1} \bbE [\bX A]$:
\begin{align*}
& \ \bbE [A \bX^{\top}] \bSigma^{-1} \bbE [\bX A] \\
= & \ \bbE [\phi (\balpha^{\top} \bX) \bX^{\top}] \bSigma^{-1} \bbE [\bX A] \\
= & \ \bbE [A] \bm{\mu}^{\top} \bSigma^{-1} \bbE [\bX A] + \bbE [\phi' (\balpha^{\top} \bX)] \balpha^{\top} \bbE [\bX A].
\end{align*}

\item Derivation related to $\bbE [Y \bX^{\top}] \bSigma^{-1} \bbE [\bX A]$:
\begin{align*}
\bbE [Y \bX^{\top}] \bSigma^{-1} \bbE [\bX A] = \psi \cdot \bbE [A \bX^{\top}] \bSigma^{-1} \bbE [\bX A] + \bbeta^{\top} \bbE [\bX A].
\end{align*}

\item Derivation related to $\bbE [Y A \bX^{\top}] \bSigma^{-1} \bbE [\bX A]$:
\begin{align*}
& \ \bbE [Y A \bX^{\top}] \bSigma^{-1} \bbE [\bX A] \\
= & \ \psi \cdot \bbE [A \bX^{\top}] \bSigma^{-1} \bbE [\bX A] + \bbeta^{\top} \bbE [\bX \bX^{\top} A] \bSigma^{-1} \bbE [\bX A] \\
= & \ \psi \cdot m_{\bX A, 2} + \bbeta^{\top} \bbE [\bX \bX^{\top} \phi (\balpha^{\top} \bX)] \bSigma^{-1} \bbE [\bX A] \\
= & \ \psi \cdot m_{\bX A, 2} + \bbeta^{\top} \left\{ \begin{array}{c} \bSigma \bbE [A] + \bm{\mu} \bm{\mu}^{\top} \bbE [A] + \bbE [\phi' (\balpha^{\top} \bX)] \left( \bm{\mu} \balpha^{\top} \bSigma + \bSigma \balpha \bm{\mu}^{\top} \right) \\
+ \, \bbE [\phi'' (\balpha^{\top} \bX)] \bSigma \balpha \balpha^{\top} \bSigma
\end{array} \right\} \bSigma^{-1} \bbE [\bX A] \\
= & \ \psi \cdot m_{\bX A, 2} + \bbeta^{\top} \bbE [\bX A] \bbE [A] + \lambda_{\beta} \bbE [A] \bm{\mu}^{\top} \bSigma^{-1} \bbE [\bX A] \\
& + \f_{1} (\lambda_{\alpha}, \gamma_{\alpha}^{2}) \left( \lambda_{\beta} \balpha^{\top} \bbE [\bX A] + \gamma_{\alpha, \beta} \bm{\mu}^{\top} \bSigma^{-1} \bbE [\bX A] \right) + \bbE [\phi'' (\balpha^{\top} \bX)] \gamma_{\alpha, \beta} \balpha^{\top} \bbE [\bX A].
\end{align*}
\end{enumerate}
\end{proof}

\subsection{Derivation of \texorpdfstring{\eqref{MAR chain}}{MAR chain}}
\label{app:MAR}

In this section, we derive the system of moment equations \eqref{MAR chain} in Lemma~\ref{lem:MAR}.

\begin{proof}[Derivation of \eqref{MAR chain}]
Again, the system of moment equations in \eqref{MAR chain} follows directly from Corollary~\ref{cor:stein} and
some elementary calculation. We only show the following identities as the others are trivial.
\begin{enumerate}
\item Derivation related to $\bbE [A Y]$:
\begin{align*}
\bbE [A Y] & = \bbE [\eta (\balpha^{\top} \bX) \bX^{\top}] \bbeta \\
& = m_{0} \cdot \psi + \bbE [\eta' (\balpha^{\top} \bX)] \cdot \balpha^{\top} \bSigma \bbeta.
\end{align*}
\item Derivation related to $\bbE [A \bX^{\top}] \bSigma^{-1} \bm{\mu}$ and $\bbE [A Y \bX^{\top}] \bSigma^{-1} \bm{\mu}$: Given any direction $\bv$, applying Corollary~\ref{cor:stein}, we have
\begin{align*}
\bbE [A \bX^{\top}] \bv = & \ \bbE [\eta (\balpha^{\top} \bX) \bX^{\top}] \bv = \left( \bbE [\eta (\balpha^{\top} \bX)] \bm{\mu}^{\top} + \bbE [\eta' (\balpha^{\top} \bX)] \balpha^{\top} \bSigma \right) \bv \\
= & \ m_{0} \cdot \bm{\mu}^{\top} \bv + \bbE [\eta' (\balpha^{\top} \bX)] \cdot \balpha^{\top} \bSigma \bv, \text{ and } \\
\bbE [A Y \bX^{\top}] \bv = & \ \bbeta^{\top} \bbE [\bX \bX^{\top} \eta (\balpha^{\top} \bX)] \bv \\
= & \ \bbE [\bbeta^{\top} \bX \eta (\balpha^{\top} \bX) \bX^{\top}] \bv \\
= & \ \left( \bbE [\bbeta^{\top} \bX \eta (\balpha^{\top} \bX)] \bm{\mu}^{\top} + \bbE [\eta (\balpha^{\top} \bX)] \bbeta^{\top} \bSigma + \bbE [\bbeta^{\top} \bX \eta' (\balpha^{\top} \bX)] \balpha^{\top} \bSigma \right) \bv \\
= & \ \bbE [\eta (\balpha^{\top} \bX) \bX^{\top}] \bbeta \cdot \bm{\mu}^{\top} \bv + m_{0} \cdot \bbeta^{\top} \bSigma \bv + \bbE [\eta' (\balpha^{\top} \bX) \bX^{\top}] \bbeta \cdot \balpha^{\top} \bSigma \bv \\
= & \ \bbE [\eta (\balpha^{\top} \bX)] \bm{\mu}^{\top} \bbeta \cdot \bm{\mu}^{\top} \bv + \bbE [\eta' (\balpha^{\top} \bX)] \balpha^{\top} \bSigma \bbeta \cdot \bm{\mu}^{\top} \bv + m_{0} \cdot \bbeta^{\top} \bSigma \bv \\
& + \bbE [\eta' (\balpha^{\top} \bX)] \bm{\mu}^{\top} \bbeta \cdot \balpha^{\top} \bSigma \bv + \bbE [\eta'' (\balpha^{\top} \bX)] \balpha^{\top} \bSigma \bbeta \cdot \balpha^{\top} \bSigma \bv \\
= & \ \left\{ m_{0} \cdot \bm{\mu}^{\top} \bv + \bbE [\eta' (\balpha^{\top} \bX)] \cdot \balpha^{\top} \bSigma \bv \right\} \cdot \bbeta^{\top} \bm{\mu} + m_{0} \cdot \bbeta^{\top} \bSigma \bv \\
& + \left\{ \bbE [\eta' (\balpha^{\top} \bX)] \cdot \bm{\mu}^{\top} \bv + \bbE [\eta'' (\balpha^{\top} \bX)] \cdot \balpha^{\top} \bSigma \bv \right\} \cdot \balpha^{\top} \bSigma \bbeta \\
= & \ \bbE [A \bX^{\top}] \bv \cdot \psi + m_{0} \cdot \bbeta^{\top} \bSigma \bv + \left\{ \bbE [\eta' (\balpha^{\top} \bX)] \cdot \bm{\mu}^{\top} \bv + \bbE [\eta'' (\balpha^{\top} \bX)] \cdot \balpha^{\top} \bSigma \bv \right\} \cdot \balpha^{\top} \bSigma \bbeta.
\end{align*}
Taking $\bv$ to certain specific direction, we further obtain the following list of identities:
\begin{align*}
\bbE [A \bX^{\top}] \bSigma^{-1} \bm{\mu} = & \ m_{0} \cdot \bm{\mu}^{\top} \bSigma^{-1} \bm{\mu} + \bbE [\eta' (\balpha^{\top} \bX)] \balpha^{\top} \bm{\mu}, \text{ and} \\
\bbE [A Y \bX^{\top}] \bSigma^{-1} \bm{\mu} = & \ \bbE [A \bX^{\top}] \bSigma^{-1} \bm{\mu} \cdot \psi + m_{0} \cdot \bbeta^{\top} \bm{\mu} \\
& + \left\{ \bbE [\eta' (\balpha^{\top} \bX)] \cdot \bm{\mu}^{\top} \bSigma^{-1} \bm{\mu} + \bbE [\eta'' (\balpha^{\top} \bX)] \cdot \balpha^{\top} \bm{\mu} \right\} \cdot \balpha^{\top} \bSigma \bbeta \\
= & \ \left( \bbE [A \bX^{\top}] \bSigma^{-1} \bm{\mu} + m_{0} \right) \cdot \psi \\
& + \left\{ \bbE [\eta' (\balpha^{\top} \bX)] \cdot \bm{\mu}^{\top} \bSigma^{-1} \bm{\mu} + \bbE [\eta'' (\balpha^{\top} \bX)] \cdot \balpha^{\top} \bm{\mu} \right\} \cdot \balpha^{\top} \bSigma \bbeta.
\end{align*}
\item Derivation related to $\bbE [A \bX^{\top}] \bSigma^{-1} \bbE [\bX A]$: Again, choosing $\bv$ appropriately, we have
\begin{align*}
& \ \bbE [A \bX^{\top}] \bSigma^{-1} \bbE [\bX A] = m_{0}^{2} \cdot \bm{\mu}^{\top} \bSigma^{-1} \bm{\mu} + \bbE^{2} [\eta' (\balpha^{\top} \bX)] \cdot \balpha^{\top} \bSigma \balpha + 2 \cdot m_{0} \cdot \bbE [\eta' (\balpha^{\top} \bX)] \cdot \balpha^{\top} \bm{\mu}.
\end{align*}
\end{enumerate}

\end{proof}

\subsection{Derivation of \texorpdfstring{\eqref{GCM chain}}{GCM chain}}
\label{app:GCM}

In this section, we derive the system of moment equations \eqref{GCM chain} in Lemma~\ref{lem:GCM}.

\begin{proof}[Derivation of \eqref{GCM chain}]
All the moment equations of \eqref{GCM chain} are equivalent to those given in Lemma~\ref{lem:mu} except the last one, which we now derive by again applying Corollary~\ref{cor:stein}:
\begin{align*}
& \ \bbE [A \bX^{\top}] \bSigma^{-1} \bbE [\bX Y] \\
= & \ \bbE \left[ \eta (\balpha^{\top} \bX) \bX^{\top} \right] \bSigma^{-1} \bbE \left[ \bX \phi (\bX^{\top} \bbeta) \right] \\
= & \ \left\{ \bbE [A] \bm{\mu}^{\top} + \bbE [\eta' (\balpha^{\top} \bX)] \balpha^{\top} \bSigma \right\} \bSigma^{-1} \left\{ \bm{\mu} \bbE [Y] + \bSigma \bbeta \bbE [\phi' (\bX^{\top} \bbeta)] \right\} \\
= & \ m_{A} \cdot m_{Y} \cdot m_{\bX, 2} + m_{A} \cdot \f_{1} (\lambda_{\beta}, \gamma_{\beta}^{2}) \cdot \lambda_{\beta} + m_{Y} \cdot \g_{1} (\lambda_{\alpha}, \gamma_{\alpha}^{2}) \cdot \lambda_{\alpha} + \g_{1} (\lambda_{\alpha}, \gamma_{\alpha}^{2}) \cdot \f_{1} (\lambda_{\beta}, \gamma_{\beta}^{2}) \cdot \gamma_{\alpha, \beta}.
\end{align*}
\end{proof}

\section{Proofs Related to Identifications under Gaussian Designs}
\label{app:identification}

\subsection{GLMs with non-zero covariate mean \texorpdfstring{$\bm{\mu}$}{m}}
\label{app:non-zero}

In this section, we first prove Lemma~\ref{lem:mu}, which states that the linear form $\lambda_{\beta} = \bbeta^{\top} \bm{\mu}$ and the quadratic form $\gamma_{\beta}^{2} = \bbeta^{\top} \bSigma \bbeta$ of the regression coefficient vector $\bbeta$ in Model~\ref{GLM} are simultaneously identifiable from the system of (population) moment equations \eqref{GLM chain}.

\begin{proof}[Proof of Lemma~\ref{lem:mu}]
Let $Z \sim \N (\lambda_{\beta}, \gamma_{\beta}^{2})$. Recall that we reduce the first four moment equations \eqref{GLM chain} as
\begin{align*}
& \f_{0} (\lambda_{\beta}, \gamma_{\beta}^{2}) - m_{1} = 0, \\
& \gamma_{\beta}^{2} \f_{1}^{2} (\lambda_{\beta}, \gamma_{\beta}^{2}) - m_{2} = 0,
\end{align*}
and denote the forward map as $\Psi_{\GLM, \beta}$.

We first compute the Jacobian $\nabla \Psi_{\GLM, \beta}$ as
\begin{align*}
\nabla \Psi_{\GLM, \beta} = & \ \left( \begin{matrix}
\bbE \phi' (Z) & 2 \bbE \phi' (Z) \bbE \phi' (Z) (Z - \lambda_{\beta}) \\
\dfrac{1}{2 \gamma_{\beta}^{2}} \bbE \phi' (Z) (Z - \lambda_{\beta}) & \dfrac{1}{\gamma_{\beta}^{2}} \bbE \phi' (Z) \bbE \phi' (Z) (Z - \lambda_{\beta})^{2}
\end{matrix} \right) \\
= & \left( \begin{matrix}
\bbE \phi' (Z) & 2 \gamma_{\beta}^{2} \bbE \phi' (Z) \bbE \phi'' (Z) \\
\frac{1}{2} \bbE \phi'' (Z) & \bbE \phi' (Z) \left( \bbE \phi' (Z) + \gamma_{\beta}^{2} \bbE \phi''' (Z) \right)
\end{matrix} \right).
\end{align*}

Without loss of generality, we take the link function $\phi$ to be monotonically strictly increasing. The determinant of $\nabla \Psi_{\GLM, \beta}$ is
\begin{align*}
|\nabla \Psi_{\GLM, \beta}| & = \frac{1}{\gamma_{\beta}^{2}} \bbE \phi' (Z) \left\{ \bbE \phi' (Z) \bbE \phi' (Z) (Z - \lambda_{\beta})^{2} - \left( \bbE \phi' (Z) (Z - \lambda_{\beta}) \right)^{2} \right\} \\
& = \bbE \phi' (Z) \left\{ \bbE^{2} \phi' (Z) + \gamma_{\beta}^{2} \left( \bbE \phi' (Z) \bbE \phi''' (Z) - \bbE^{2} \phi'' (Z) \right) \right\} \geq 0
\end{align*}
where $|\nabla \Psi_{\GLM, \beta}| \geq 0$ follows from applying Cauchy-Schwarz inequality to the first line of the above display and $\phi$ monotonically increasing. Here, we note that monotonicity is but a sufficient condition for the above inequality to hold -- which leaves the door open to extend our approach to SIMs with general non-monotonic link functions.

Lemma~\ref{lem:inverse} implies that we are only left to show that the above inequality is strict and the corresponding map $\Psi$ is proper. Suppose on the contrary, the above inequality is an equality. Then we have, with probability one,
\begin{align*}
\sqrt{\phi' (Z)} (Z - \lambda_{\beta}) \equiv \frac{\bbE \phi' (Z) (Z - \lambda_{\beta})}{\bbE \phi' (Z)} \sqrt{\phi' (Z)},
\end{align*}
which is equivalent to, with probability one,
\begin{align*}
Z - \lambda_{\beta} \equiv \frac{\bbE \phi' (Z) (Z - \lambda_{\beta})}{\bbE \phi' (Z)} \text{ or } \sqrt{\phi' (Z)} \equiv 0,
\end{align*}
a contradiction unless $Z$ is degenerate (i.e. $\gamma_{\beta}^{2} = 0$). Thus we have proved that the forward map $\Psi_{\GLM, \beta}$ is a local diffeomorphism. $\Psi_{\GLM, \beta}$ is also proper under Assumption~\ref{as:link}.

In fact, for linear and log-linear link functions, it is easy to see that $|\nabla \Psi|$ is strictly positive for strictly bounded $\lambda_{\beta}$ and $\gamma_{\beta}^{2}$. For probit link, we observe that
\begin{align*}
|\nabla \Psi_{\GLM, \beta}| = & \ \bbE \phi' (Z) \left\{ \bbE^{2} \phi' (Z) + \gamma_{\beta}^{2} \left( \bbE \phi' (Z) \bbE \phi''' (Z) - \bbE^{2} \phi'' (Z) \right) \right\} \\
= & \ \frac{1}{(2 \pi (1 + \gamma_{\beta}^{2}))^{3 / 2}} \exp \left\{ - \frac{3 \lambda_{\beta}^{2}}{2 (1 + \gamma_{\beta}^{2})} \right\} \left( 1 + \gamma_{\beta}^{2} \left\{ \frac{\lambda_{\beta}^{2}}{(1 + \gamma_{\beta}^{2})^{2}} - \frac{1}{1 + \gamma_{\beta}^{2}} - \frac{\lambda_{\beta}^{2}}{(1 + \gamma_{\beta}^{2})^{2}} \right\} \right) \\
= & \ \frac{1}{(2 \pi (1 + \gamma_{\beta}^{2}))^{3 / 2}} \exp \left\{ - \frac{3 \lambda_{\beta}^{2}}{2 (1 + \gamma_{\beta}^{2})} \right\} \frac{1}{1 + \gamma_{\beta}^{2}} > 0
\end{align*}
if $\lambda_{\beta}^{2}$ and $\gamma_{\beta}^{2}$ are strictly bounded.

For logit link, there is no analytic expression for $|\nabla \Psi_{\GLM, \beta}|$ but it is easy to numerically show that $|\nabla \Psi_{\GLM, \beta}|$ is strictly larger than zero, and hence $\Psi_{\GLM, \beta}$ is a diffeomorphism. The remaining part of Lemma~\ref{lem:mu} is straightforward to prove and hence omitted.

\end{proof}

\subsection{Identification under Model \texorpdfstring{\ref{CE}}{ce}}
\label{app:CE id}

\begin{proof}[Proof of Lemma~\ref{lem:CE}]\leavevmode
The proof follows directly from the proof of Lemma~\ref{lem:mu} and the linearity in $\psi$ of \eqref{CE chain}.
\end{proof}

\subsection{Identification under Model \texorpdfstring{\ref{MAR}}{mar}}
\label{app:MAR id}

\begin{proof}[Proof of Lemma~\ref{lem:MAR}]
The proof follows directly from the proof of Lemma~\ref{lem:mu} and the linearity in $\psi$ of \eqref{MAR chain}.
\end{proof}

\subsection{Identification under Model \texorpdfstring{\ref{GCM}}{gcm}}
\label{app:GCM id}

\begin{proof}[Proof of Lemma~\ref{lem:GCM}]
Denote the forward map induced by the first seven equations in system \eqref{GCM chain} as 
\begin{align*}
\Psi_{\GLM}: (\lambda_{\alpha}, \gamma_{\alpha}^{2}, \lambda_{\beta}, \gamma_{\beta}^{2})^{\top} \mapsto (m_{A}, m_{Y}, m_{\bX, 2}, m_{\bX A, \bX}, m_{\bX Y, \bX}, m_{\bX A, 2}, m_{\bX Y, 2})^{\top}.
\end{align*}
As an immediate consequence of Lemma~\ref{lem:mu}, $\Psi_{\GLM}$ is a diffeomorphism and thus $(\lambda_{\alpha}, \gamma_{\alpha}^{2}, \lambda_{\beta}, \gamma_{\beta}^{2})$ is identifiable by inverting $\Psi_{\GLM}$. $\gamma_{\alpha, \beta}$ is then identified by solving \eqref{mm3}. Finally, since the target parameter $\psi$ can be written as a known function of $(\lambda_{\alpha}, \gamma_{\alpha}^{2}, \lambda_{\beta}, \gamma_{\beta}^{2}, \gamma_{\alpha, \beta})$, the proof is complete.
\end{proof}

\section{Proofs Related to CAN}
\label{app:CLT}

\subsection{General results for limiting distributions of moment estimators}
\label{app:clt}

We first state a general result. Consider the following pair of random variables, comprised of a first-order and a second-order degenerate $U$-statistic:
\begin{equation}
\label{pair}
\left( \begin{matrix}
\sum\limits_{i = 1}^{n} g_{n} (O_{i}) & \sum\limits_{1 \leq i_{1} < i_{2} \leq n} h_{n} (O_{i_{1}}, O_{i_{2}})
\end{matrix} \right)^{\top},
\end{equation}
with $\bbE g_{n} (O) = 0$ and $\bbE h_{n} (O, o) = 0$ almost surely. Our goal is to establish their joint limiting distribution as the sample size $n$ approaches infinity.

The following result in \citet{bhattacharya1992class} provides a set of generic sufficient conditions under which \eqref{pair} has a Gaussian limit.
\begin{lemma}[Corollary 1.4 of \citet{bhattacharya1992class}]
\label{lem:clt}
Suppose that the following conditions hold: as $n \rightarrow \infty$,
\begin{enumerate}[label = (\arabic*)]
\item $n \bbE g_{n} (O)^{2} \rightarrow \sigma_{g}^{2}$;
\item $n^{2} \bbE h_{n} (O_{1}, O_{2})^{2} \rightarrow \sigma_{h}^{2}$;
\item $n \bbE g_{n} (O)^{4} \rightarrow 0$;
\item $n^{3} \bbE \left( \int_{o} h_{n} (O, o)^{2} \diff \bbP (o) \right)^{2} \rightarrow 0$;
\item $n^{4} \bbE \left( \int_{o} h_{n} (O_{1}, o) h_{n} (O_{2}, o) \diff \bbP (o) \right)^{2} \rightarrow 0$;
\item $n^{2} \bbE h_{n} (O_{1}, O_{2})^{4} \rightarrow 0$;
\item $n^{3} \bbE \left( \int_{o} g_{n} (o) h_{n} (O, o) \diff \bbP (o) \right)^{2} \rightarrow 0$.
\end{enumerate}
We then have
\begin{align*}
\left( \begin{matrix}
\sum\limits_{i = 1}^{n} g_{n} (O_{i}) \\ \sum\limits_{1 \leq i_{1} \neq i_{2} \leq n} h_{n} (O_{i_{1}}, O_{i_{2}})
\end{matrix} \right) \overset{\calL}{\rightarrow} \N_{2} \left( \left( \begin{matrix}
0 \\
0
\end{matrix} \right), \left( \begin{matrix}
\sigma_{g}^{2} & 0 \\
0 & \sigma_{h}^{2} / 2
\end{matrix} \right) \right).
\end{align*}
\end{lemma}

Let $\calO \ni O \coloneqq (\bX, W)$ where $\bX \in \bbR^{p}$ and $W \in \calW$. In the sequel, we assume that there exist two functions $f, g: \calW \rightarrow \bbR$ of $W$. We further let $\mu_{q} (\bX) = \bbE [q (W) | \bX]$, $\nu^{2}_{q} (\bX) \coloneqq \bbE [q (W)^{2} | \bX]$, $\sigma^{2}_{q} (\bX) \coloneqq \var [q (W) | \bX]$, and $\bSigma_{q} \coloneqq \bbE [\nu^{2}_{q} (\bX) \bX \bX^{\top}]$ for $q \in \{f, g, \sqrt{f g}\}$. Here without loss of generality, we take $f g \geq 0$. Obviously, $\nu_{q}^{2} \equiv \sigma_{q}^{2}  + \mu_{q}^{2}$ almost surely.

We have the following representations of $\nu_{q}^{2}$ and $\bSigma_{q}$ that will be useful for later development.

\begin{lemma}
\label{lem:stein quad}
Under Assumption~\ref{as:Sigma} and~\ref{as:normal}, given a n.n.s.d. matrix $\bbM$, if $q$ is chosen such that $\bbE [\nabla \nu_{q}^{2} (\bX)]$ and $\bbE [\nabla^{2} \nu_{q}^{2} (\bX)]$ exist, we have
\begin{equation}
\label{stein quad}
\begin{split}
\bSigma_{q} = \bSigma \bbE [q (W)^{2}] & + \bm{\mu} \bm{\mu}^{\top} \bbE [q (W)^{2}] + \bm{\mu} \bbE [\nabla \nu_{q}^{2} (\bX)]^{\top} \bSigma + \bSigma \bbE [\nabla \nu_{q}^{2} (\bX)] \bm{\mu}^{\top} + \bSigma \bbE [\nabla^{2} \nu_{q}^{2} (\bX)] \bSigma, \\
\bbE [\bX^{\top} \bbM \bX \nu_{q}^{2} (\bX)] = & \ \bm{\mu}^{\top} \bbM \bm{\mu} \bbE [q (W)^{2}] + 2 \bm{\mu}^{\top} \bbM \bSigma \bbE [\nabla \nu_{q}^{2} (\bX)] + \trace [\bSigma \bbM] \bbE [q (W)^{2}] \\
& + \trace [\bSigma \bbM \bSigma \bbE [\nabla^{2} \nu_{q}^{2} (\bX)]].
\end{split}
\end{equation}
\end{lemma}

\begin{proof}
First, by applying \eqref{matrix} from Corollary~\ref{cor:stein}, we have
\begin{align*}
\bSigma_{q} = & \ \bSigma \bbE [\nu_{q}^{2} (\bX)] + \bm{\mu} \bm{\mu}^{\top} \bbE [\nu_{q}^{2} (\bX)] + \bm{\mu} \bbE [\nabla \nu_{q}^{2} (\bX)]^{\top} \bSigma + \bSigma \bbE [\nabla \nu_{q}^{2} (\bX)] \bm{\mu}^{\top} + \bSigma \bbE [\nabla^{2} \nu_{q}^{2} (\bX)] \bSigma \\
= & \ \bSigma \bbE [q (W)^{2}] + \bm{\mu} \bm{\mu}^{\top} \bbE [q (W)^{2}] + \bm{\mu} \bbE [\nabla \nu_{q}^{2} (\bX)]^{\top} \bSigma + \bSigma \bbE [\nabla \nu_{q}^{2} (\bX)] \bm{\mu}^{\top} + \bSigma \bbE [\nabla^{2} \nu_{q}^{2} (\bX)] \bSigma.
\end{align*}
Second, by applying \eqref{trace} from Corollary~\ref{cor:stein}, we have
\begin{align*}
& \ \bbE [\bX^{\top} \bbM \bX \nu_{q}^{2} (\bX)] \\
= & \ \bm{\mu}^{\top} \bbM \bbE [\bX \nu_{q}^{2} (\bX)] + \trace [\bSigma \bbM \bSigma \bbE [\nabla (\bX \nu_{q}^{2} (\bX)]] \\
= & \ \bm{\mu}^{\top} \bbM \bm{\mu} \bbE [q (W)^{2}] + \bm{\mu}^{\top} \bbM \bSigma \bbE [\nabla \nu_{q}^{2} (\bX)] + \trace [\bSigma \bbM (\bbE [\nu_{q}^{2} (\bX)] + \bbE [\bX \nabla \nu_{q}^{2} (\bX)^{\top}])] \\
= & \ \bm{\mu}^{\top} \bbM \bm{\mu} \bbE [q (W)^{2}] + \bm{\mu}^{\top} \bbM \bSigma \bbE [\nabla \nu_{q}^{2} (\bX)] + \trace [\bSigma \bbM] \bbE [q (W)^{2}] \\
& + \trace [\bSigma \bbM \bm{\mu} \nabla \bbE [\nu_{q}^{2} (\bX)^{\top}]] + \trace [\bSigma \bbM \bSigma \bbE [\nabla^{2} \nu_{q}^{2} (\bX)]].
\end{align*}
\end{proof}

Define
\begin{align*}
T_{n} \coloneqq \bbU_{n, 2} \left[ f (W_{1}) \bX_{1}^{\top} \bbM \bX_{2} g (W_{2}) \right] \equiv \bbU_{n, 2} \left[ \bbS_{2} \{f (W_{1}) \bX_{1}^{\top} \bbM \bX_{2} g (W_{2})\} \right]
\end{align*}
where $\bbM$ is a fixed symmetric positive semi-definite matrix with strictly bounded eigenvalues and $\bbS_{2}$ denotes the symmetrization operator: $\bbS_{2} \{h (O_{1}, O_{2}\} \equiv \{h (O_{1}, O_{2}) + h (O_{2}, O_{1})\} / 2$ for any function $h: \calO_{1} \times \calO_{2} \rightarrow \bbR$. To simplify notation, we let $s (O_{1}, O_{2}) \coloneqq \bbS_{2} \{f (W_{1}) \bX_{1}^{\top} \bbM \bX_{2} g (W_{2})\}$.

\begin{proposition}
\label{prop:general clt}
Under Assumptions~\ref{as:Sigma},~\ref{as:link},~\ref{as:normal}, and~\ref{as:bounded}, suppose that the following are additionally satisfied: Given a n.n.s.d. matrix $\bbM$, if either the following quantities
\begin{subequations}
\label{check conditions 1}
\begin{align}
& \bm{\mu}^{\top} \bbM \bm{\mu}, \bm{\mu}^{\top} \bbM \bSigma \bbM \bm{\mu}, \bm{\mu}^{\top} \bbM \bSigma \bbE [\nabla^{2} \nu_{q}^{2} (\bX)] \bSigma \bbM \bm{\mu}, \bbE [q (W)], \bbE [q (W)^{2}], \\
& \bm{\mu}^{\top} \bbM \bSigma \bbE [\nabla \mu_{q} (\bX)], \bm{\mu}^{\top} \bbM \bSigma \bbE [\nabla \nu_{q}^{2} (\bX)], \bm{\mu}^{\top} \bbM \bSigma \bbM \bSigma \bbE [\nabla \mu_{q} (\bX)], \\
& \bm{\mu}^{\top} \bbM \bSigma \bbE [\nabla^{2} \nu_{q'}^{2} (\bX)] \bSigma \bbM \bSigma \bbE [\nabla \mu_{q} (\bX)] \\
& \bbE [\nabla \mu_{q} (\bX)]^{\top} \bSigma \bbM \bSigma \bbM \bSigma \bbE [\nabla \mu_{q'} (\bX)], \bbE [\nabla \nu_{q}^{2} (\bX)]^{\top} \bSigma \bbM \bSigma \bbE [\nabla \mu_{q'} (\bX)], \\
& \bbE [\nabla \mu_{q} (\bX)]^{\top} \bSigma \bbM \bSigma \bbE [\nabla^{2} \nu_{q''}^{2} (\bX)] \bSigma \bbM \bSigma \bbE [\nabla \mu_{q'} (\bX)],
\end{align}
\end{subequations}
or the following quantities (or both)
\begin{subequations}
\label{check conditions 2}
\begin{align}
& p^{-1} \trace [\bSigma \bbM \bSigma \bbM] \bbE [q (W)^{2}] \bbE [q' (W)^{2}], p^{-1} \trace [\bSigma \bbM \bSigma \bbM \bSigma \bbE [\nabla^{2} \nu_{q}^{2} (\bX)]], \\
& p^{-1} \trace [\bSigma \bbM \bSigma \bbE [\nabla^{2} \nu_{q}^{2} (\bX)] \bSigma \bbM \bSigma \bbE [\nabla^{2} \nu_{q'}^{2} (\bX)]],
\end{align}
\end{subequations}
converge to nontrivial limits as $n \rightarrow \infty$, where $q, q', q'' \in \{f, g, \sqrt{fg}\}$, and if there exists a universal constant $B > 0$ and a non-negative scalar function $\bar{\eta} (\bX) \in L_{2} (\bbP)$ such that for $\bX' \indep \bX$,
\begin{equation}
\label{cond:clt}
\lambda_{\max} (\bbE [\nabla^{2} \eta_{q} (\bX)]) \leq B < \infty, \text{ and } |\bX^{'\top} \nabla^{2} \eta_{q} (\bX) \bX'| \lesssim \Vert \bX' \Vert^{2} \bar{\eta} (\bX)
\end{equation}
where $\eta_{q} (\bX) \coloneqq \bbE [q (W)^{4} | \bX]$ and $\eta_{q} \in L_{2} (\bbP)$ for $q \in \{f, g, \sqrt{fg}\}$, then we have
\begin{align*}
\sqrt{n} \left( T_{n} - \bbE T_{n} \right) \overset{\calL}{\rightarrow} \N (0, \nu^{2})
\end{align*}
for some positive constant $\nu^{2} > 0$.
\end{proposition}

\begin{proof}
We first apply the Hoeffding decomposition to $T_{n}$ and get:
\begin{align*}
T_{n} - \bbE T_{n} & \equiv \bbU_{n, 2} [s (O_{1}, O_{2}) - \bbE s (O_{1}, O_{2})] \\
& = \bbU_{n, 1} [\bbE [s (O_{1}, O_{2} | O_{1}] - \bbE s (O_{1}, O_{2})] + \bbU_{n, 2} [s (O_{1}, O_{2}) - \bbE [s (O_{1}, O_{2} | O_{1}]] \\
& = \frac{1}{\sqrt{n}} \left\{ \sum_{1 \leq i_{1} < i_{2} \leq n} \frac{s (O_{i_{1}}, O_{i_{2}}) - \bbE [s (O_{i_{1}}, O) | O_{i_{1}}]}{\sqrt{n} (n - 1)} + \sum_{i = 1}^{n} \frac{\bbE [s (O_{i_{1}}, O) | O_{i_{1}}] - \bbE s (O_{1}, O_{2})}{\sqrt{n}} \right\}.
\end{align*}
We take $g_{n} (O_{i}) \equiv \frac{\bbE [s (O_{i}, O) | O_{i}] - \bbE s (O_{1}, O_{2})}{\sqrt{n}}$ and $h_{n} (O_{i_{1}}, O_{i_{2}}) \equiv \frac{s (O_{i_{1}}, O_{i_{2}}) - \bbE [s (O_{i_{1}}, O) | O_{i_{1}}]}{\sqrt{n} (n - 1)}$.

We are then left to check the seven conditions of Lemma~\ref{lem:clt} one by one. We first check Condition~(1).
\begin{align*}
& \ n \bbE g_{n} (O)^{2} \\
= & \ \bbE \left[ \left\{ \bbE [s (O, O') | O] - \bbE s (O_{1}, O_{2}) \right\}^{2} \right] \\
= & \ \bbE \left[ \left\{ \bbE [\bbS_{2} \{f (W) \bX^{\top} \bbM \bX' g (W'\} | O] - \bbE [f (W) \bX^{\top}] \bbM \bbE [\bX g (W)] \right\}^{2} \right] \\
= & \ \bbE \left[ \left\{ \frac{1}{2} \left( f (W) \bX^{\top} \bbM \bbE [\bX g (W)] + g (W) \bX^{\top} \bbM \bbE [\bX f (W)] \right) - \bbE [f (W) \bX^{\top}] \bbM \bbE [\bX g (W)] \right\}^{2} \right] \\
= & \ \frac{1}{4} \bbE \left[ \left\{ (f (W) \bX^{\top} - \bbE [f (W) \bX^{\top}]) \bbM \bbE [\bX g (W)] \right\}^{2} \right] + \frac{1}{4} \bbE \left[ \left\{ (g (W) \bX^{\top} - \bbE [g (W) \bX^{\top}]) \bbM \bbE [\bX f (W)] \right\}^{2} \right] \\
& + \frac{1}{2} \bbE \left[ (f (W) \bX^{\top} - \bbE [f (W) \bX^{\top}]) \bbM \bbE [\bX g (W)] \cdot (g (W) \bX^{\top} - \bbE [g (W) \bX^{\top}]) \bbM \bbE [\bX f (W)] \right] \\
= & \ \frac{1}{4} \bbE \left[ \left\{ f (W) \bX^{\top} \bbM \bbE [\bX g (W)] \right\}^{2} \right] + \frac{1}{4} \bbE \left[ \left\{ g (W) \bX^{\top} \bbM \bbE [\bX f (W)] \right\}^{2} \right] \\
& + \frac{1}{2} \bbE \left[ f (W) g (W) \bX^{\top} \bbM \bbE [\bX f (W)] \cdot \bX^{\top} \bbM \bbE [\bX g (W)] \right] - \left( \bbE [f (W) \bX^{\top}] \bbM \bbE [\bX g (W)] \right)^{2} \\
\eqqcolon & \ \frac{1}{4} I_{1} + \frac{1}{4} I_{2} + \frac{1}{2} I_{3} - I_{0}.
\end{align*}
We then compute the limit of $I_{0}, I_{1}, I_{2}, I_{3}$ separately.

We first handle $I_{0}$.
\begin{align*}
I_{0} & = \left( \begin{array}{c}
\bbE [f (W)] \bbE [g (W)] \bm{\mu}^{\top} \bbM \bm{\mu} + \bbE [\nabla \mu_{f} (\bX)^{\top}] \bSigma \bbM \bSigma \bbE [\nabla \mu_{g} (\bX)] \\
+ \, \bbE [f (W)] \bm{\mu}^{\top} \bbM \bSigma \bbE [\nabla \mu_{g} (\bX)] + \bbE [g (W)] \bm{\mu}^{\top} \bbM \bSigma \bbE [\nabla \mu_{f} (\bX)]
\end{array} \right)^{2}
\end{align*}

We next analyze $I_{1}$ by repeatedly applying Lemma~\ref{lem:stein quad}.
\begin{align*}
I_{1} = & \ \bbE [\bX^{\top} \bbM \bbE [\bX \mu_{g} (\bX)] \bbE [\bX \mu_{g} (\bX)]^{\top} \bbM \bX \nu_{f}^{2} (\bX)] \\
= & \ \bbE [\bX^{\top} \bbM \left( \bm{\mu} \bbE [\mu_{g} (\bX)] + \bSigma \bbE [\nabla \mu_{g} (\bX)] \right) \left( \bm{\mu} \bbE [\mu_{g} (\bX)] + \bSigma \bbE [\nabla \mu_{g} (\bX)] \right)^{\top} \bbM \bX \nu_{f}^{2} (\bX)] \\
= & \ \bbE [\bX^{\top} \bbM \bm{\mu} \bm{\mu}^{\top} \bbM \bX \nu_{f}^{2} (\bX)] \bbE^{2} [g (W)] + 2 \bbE [\bX^{\top} \bbM \bSigma \bbE [\nabla \mu_{g} (\bX)] \bm{\mu}^{\top} \bbM \bX \nu_{f}^{2} (\bX)] \bbE [g (W)] \\
& + \bbE [\bX^{\top} \bbM \bSigma \bbE [\nabla \mu_{g} (\bX)] \bbE [\nabla \mu_{g} (\bX)]^{\top} \bSigma \bbM \bX \nu_{f}^{2} (\bX)] \\
= & \ \left( \bm{\mu}^{\top} \bbM \bm{\mu} \right)^{2} \bbE [f (W)^{2}] \bbE^{2} [g (W)] + 2 \bm{\mu}^{\top} \bbM \bm{\mu} \bm{\mu}^{\top} \bbM \bSigma \bbE [\nabla \nu_{f}^{2} (\bX)] \bbE^{2} [g (W)] \\
& + \trace [\bSigma \bbM \bm{\mu} \bm{\mu}^{\top} \bbM] \bbE [f (W)^{2}] \bbE^{2} [g (W)] + \trace [\bSigma \bbM \bm{\mu} \bm{\mu}^{\top} \bbM \bSigma \bbE [\nabla^{2} \nu_{f}^{2} (\bX)]] \bbE^{2} [g (W)] \\
& + 2 \bm{\mu}^{\top} \bbM \bSigma \bbE [\nabla \mu_{g} (\bX)] \bm{\mu}^{\top} \bbM \bm{\mu} \bbE [f (W)^{2}] \bbE [g (W)] + 4 \bm{\mu}^{\top} \bbM \bSigma \bbE [\nabla \mu_{g} (\bX)] \bm{\mu}^{\top} \bbM \bSigma \bbE [\nabla \nu_{f}^{2} (\bX)] \bbE [g (W)] \\
& + 2 \trace [\bSigma \bbM \bSigma \bbE [\nabla \mu_{g} (\bX)] \bm{\mu}^{\top} \bbM] \bbE [f (W)^{2}] \bbE [g (W)] + 2 \trace [\bSigma \bbM \bSigma \bbE [\nabla \mu_{g} (\bX)] \bm{\mu}^{\top} \bbM \bSigma \bbE [\nabla^{2} \nu_{f}^{2} (\bX)]] \bbE [g (W)] \\
& + (\bm{\mu}^{\top} \bbM \bSigma \bbE [\nabla \mu_{g} (\bX)])^{2} \bbE [f (W)^{2}] + 2 \bm{\mu}^{\top} \bbM \bSigma \bbE [\nabla \mu_{g} (\bX)] \bbE [\nabla \mu_{g} (\bX)]^{\top} \bSigma \bbM \bSigma \bbE [\nabla \nu_{f}^{2} (\bX)] \\
& + \trace [\bSigma \bbM \bSigma \bbE [\nabla \mu_{g} (\bX)] \bbE [\nabla \mu_{g} (\bX)]^{\top} \bSigma \bbM] \bbE [f (W)^{2}] \\
& + \trace [\bSigma \bbM \bSigma \bbE [\nabla \mu_{g} (\bX)] \bbE [\nabla \mu_{g} (\bX)]^{\top} \bSigma \bbM \bSigma \bbE [\nabla^{2} \nu_{f}^{2} (\bX)]] \\
= & \ \left( \bm{\mu}^{\top} \bbM \bm{\mu} \right)^{2} \bbE [f (W)^{2}] \bbE^{2} [g (W)] + \bm{\mu}^{\top} \bbM \bSigma \bbM \bm{\mu} \bbE [f (W)^{2}] \bbE^{2} [g (W)] + \bm{\mu}^{\top} \bbM \bSigma \bbE [\nabla^{2} \nu_{f}^{2} (\bX)] \bSigma \bbM \bm{\mu} \cdot \bbE^{2} [g (W)] \\
& + 2 \bm{\mu}^{\top} \bbM \bm{\mu} \cdot \bm{\mu}^{\top} \bbM \bSigma \bbE [\nabla \nu_{f}^{2} (\bX)] \bbE^{2} [g (W)] + 2 \bm{\mu}^{\top} \bbM \bm{\mu} \cdot \bm{\mu}^{\top} \bbM \bSigma \bbE [\nabla \mu_{g} (\bX)] \bbE [f (W)^{2}] \bbE [g (W)] \\
& + 4 \bm{\mu}^{\top} \bbM \bSigma \bbE [\nabla \mu_{g} (\bX)] \cdot \bm{\mu}^{\top} \bbM \bSigma \bbE [\nabla \nu_{f}^{2} (\bX)] \bbE [g (W)] + 2 \bm{\mu}^{\top} \bbM \bSigma \bbM \bSigma \bbE [\nabla \mu_{g} (\bX)] \bbE [f (W)^{2}] \bbE [g (W)] \\
& + 2 \bm{\mu}^{\top} \bbM \bSigma \bbE [\nabla^{2} \nu_{f}^{2} (\bX)] \bSigma \bbM \bSigma \bbE [\nabla \mu_{g} (\bX)] \bbE [g (W)] + (\bm{\mu}^{\top} \bbM \bSigma \bbE [\nabla \mu_{g} (\bX)])^{2} \bbE [f (W)^{2}] \\
& + 2 \bm{\mu}^{\top} \bbM \bSigma \bbE [\nabla \mu_{g} (\bX)] \cdot \bbE [\nabla \mu_{g} (\bX)]^{\top} \bSigma \bbM \bSigma \bbE [\nabla \nu_{f}^{2} (\bX)] + \bbE [\nabla \mu_{g} (\bX)]^{\top} \bSigma \bbM \bSigma \bbM \bSigma \bbE [\nabla \mu_{g} (\bX)] \cdot \bbE [f (W)^{2}] \\
& + \bbE [\nabla \mu_{g} (\bX)]^{\top} \bSigma \bbM \bSigma \bbE [\nabla^{2} \nu_{f}^{2} (\bX)] \bSigma \bbM \bSigma \bbE [\nabla \mu_{g} (\bX)].
\end{align*}
By symmetry, $I_{2}$ has the same form as $I_{1}$ except that $f$ and $g$ are swapped. By the same argument, $I_{3}$ has a similar form:
\begin{align*}
I_{3} = & \ \bbE [\bX^{\top} \bbM \bbE [\bX \mu_{f} (\bX)] \bbE [\mu_{g} (\bX) \bX^{\top}] \bbM \bX \nu_{\sqrt{fg}}^{2} (\bX)] \\
= & \ \bbE [\bX^{\top} \bbM \bm{\mu} \bm{\mu}^{\top} \bbM \bX \nu_{\sqrt{fg}}^{2} (\bX)] \bbE [f (W)] \bbE [g (W)] + \bbE [\bX^{\top} \bbM \bSigma \bbE [\nabla \mu_{g} (\bX)] \bm{\mu}^{\top} \bbM \bX \nu_{\sqrt{fg}}^{2} (\bX)] \bbE [f (W)] \\
& + \bbE [\bX^{\top} \bbM \bSigma \bbE [\nabla \mu_{f} (\bX)] \bm{\mu}^{\top} \bbM \bX \nu_{\sqrt{fg}}^{2} (\bX)] \bbE [g (W)] + \bbE [\bX^{\top} \bbM \bSigma \bbE [\nabla \mu_{g} (\bX)] \bbE [\nabla \mu_{f} (\bX)]^{\top} \bSigma \bbM \bX \nu_{\sqrt{fg}}^{2} (\bX)] \\
= & \ \left( \bm{\mu}^{\top} \bbM \bm{\mu} \right)^{2} \bbE [f (W) g (W)] \bbE [f (W)] \bbE [g (W)] + \bm{\mu}^{\top} \bbM \bSigma \bbM \bm{\mu} \bbE [f (W) g (W)] \bbE [f (W)] \bbE [g (W)] \\
& + \bm{\mu}^{\top} \bbM \bSigma \bbE [\nabla^{2} \nu_{\sqrt{fg}}^{2} (\bX)] \bSigma \bbM \bm{\mu} \cdot \bbE [f (W)] \bbE [g (W)] + 2 \bm{\mu}^{\top} \bbM \bm{\mu} \cdot \bm{\mu}^{\top} \bbM \bSigma \bbE [\nabla \nu_{\sqrt{fg}}^{2} (\bX)] \bbE [f (W)] \bbE [g (W)] \\
& + \bm{\mu}^{\top} \bbM \bm{\mu} \cdot \bm{\mu}^{\top} \bbM \bSigma \bbE [\nabla \mu_{f} (\bX)] \bbE [f (W) g (W)] \bbE [g (W)] + \bm{\mu}^{\top} \bbM \bm{\mu} \cdot \bm{\mu}^{\top} \bbM \bSigma \bbE [\nabla \mu_{g} (\bX)] \bbE [f (W) g (W)] \bbE [f (W)] \\
& + 2 \bm{\mu}^{\top} \bbM \bSigma \bbE [\nabla \mu_{f} (\bX)] \cdot \bm{\mu}^{\top} \bbM \bSigma \bbE [\nabla \nu_{\sqrt{fg}}^{2} (\bX)] \bbE [g (W)] + 2 \bm{\mu}^{\top} \bbM \bSigma \bbE [\nabla \mu_{g} (\bX)] \cdot \bm{\mu}^{\top} \bbM \bSigma \bbE [\nabla \nu_{\sqrt{fg}}^{2} (\bX)] \bbE [f (W)] \\
& + \bm{\mu}^{\top} \bbM \bSigma \bbM \bSigma \bbE [\nabla \mu_{f} (\bX)] \bbE [f (W) g (W)] \bbE [g (W)] + \bm{\mu}^{\top} \bbM \bSigma \bbM \bSigma \bbE [\nabla \mu_{g} (\bX)] \bbE [f (W) g (W)] \bbE [f (W)] \\
& + \bm{\mu}^{\top} \bbM \bSigma \bbE [\nabla^{2} \nu_{\sqrt{fg}}^{2} (\bX)] \bSigma \bbM \bSigma \bbE [\nabla \mu_{f} (\bX)] \bbE [g (W)] + \bm{\mu}^{\top} \bbM \bSigma \bbE [\nabla^{2} \nu_{\sqrt{fg}}^{2} (\bX)] \bSigma \bbM \bSigma \bbE [\nabla \mu_{g} (\bX)] \bbE [f (W)] \\
& + \bm{\mu}^{\top} \bbM \bSigma \bbE [\nabla \mu_{g} (\bX)] \cdot \bm{\mu}^{\top} \bbM \bSigma \bbE [\nabla \mu_{f} (\bX)] \bbE [f (W) g (W)] \\
& + \bbE [\nabla \mu_{f} (\bX)]^{\top} \bSigma \bbM \bSigma \bbM \bSigma \bbE [\nabla \mu_{g} (\bX)] \cdot \bbE [f (W) g (W)] \\
& + \bm{\mu}^{\top} \bbM \bSigma \bbE [\nabla \mu_{f} (\bX)] \cdot \bbE [\nabla \mu_{g} (\bX)]^{\top} \bSigma \bbM \bSigma \bbE [\nabla \nu_{\sqrt{fg}}^{2} (\bX)] \\
& + \bm{\mu}^{\top} \bbM \bSigma \bbE [\nabla \mu_{g} (\bX)] \cdot \bbE [\nabla \mu_{f} (\bX)]^{\top} \bSigma \bbM \bSigma \bbE [\nabla \nu_{\sqrt{fg}}^{2} (\bX)] \\
& + \bbE [\nabla \mu_{f} (\bX)]^{\top} \bSigma \bbM \bSigma \bbE [\nabla^{2} \nu_{\sqrt{fg}}^{2} (\bX)] \bSigma \bbM \bSigma \bbE [\nabla \mu_{g} (\bX)].
\end{align*}

We next check Condition (2).
\begin{align*}
& \ n^{2} \bbE [h_{n} (O_{1}, O_{2})^{2}] \\
= & \ \frac{n}{(n - 1)^{2}} \bbE \left[ \left( s (O_{1}, O_{2}) - \bbE [s (O_{1}, O) | O_{1}] \right)^{2} \right] \\
= & \ \frac{n}{(n - 1)^{2}} \bbE \left[ \left\{ \frac{1}{2} \left( f (W_{1}) \bX_{1}^{\top} \bbM \bX_{2} g (W_{2}) + f (W_{2}) \bX_{2}^{\top} \bbM \bX_{1} g (W_{1}) \right) - \bbE [f (W) \bX^{\top}] \bbM \bbE [\bX g (W)] \right\}^{2} \right] \\
= & \ \frac{n}{2 (n - 1)^{2}} \bbE \left[ \bbE [g (W)^{2} | \bX] \bX^{\top} \bbM \bbE \left[ \bbE [f (W)^{2} | \bX] \bX \bX^{\top} \right] \bbM \bX \right] \\
& + \frac{n}{2 (n - 1)^{2}} \bbE \left[ \bbE [f (W) g (W) | \bX] \bX^{\top} \bbM \bbE \left[ \bbE [f (W) g (W) | \bX] \bX \bX^{\top} \right] \bbM \bX \right] \\
& - \frac{n}{(n - 1)^{2}} \left( \bbE [f (W) \bX^{\top}] \bbM \bbE [\bX g (W)] \right)^{2} \\
\eqqcolon & \ \frac{n}{2 (n - 1)^{2}} J_{1} + \frac{n}{2 (n - 1)^{2}} J_{2} - \frac{n}{(n - 1)^{2}} J_{3}.
\end{align*}

We first simplify $J_{1}$ by using \eqref{stein quad} of Lemma~\ref{lem:stein quad}.
\begin{align*}
J_{1} = & \ \bbE \left[ \bbE [g (W)^{2} | \bX] \bX^{\top} \bbM \bbE \left[ \bbE [f (W)^{2} | \bX] \bX \bX^{\top} \right] \bbM \bX \right] \\
= & \ \bbE \left[ \bX^{\top} \bbM \bSigma_{f} \bbM \bX \nu_{g}^{2} (\bX) \right] \\
= & \ \bm{\mu}^{\top} \bbM \bSigma_{f} \bbM \bm{\mu} \bbE [g (W)^{2}] + 2 \bm{\mu}^{\top} \bbM \bSigma_{f} \bbM \bSigma \bbE [\nabla \nu_{g}^{2} (\bX)] + \trace [\bSigma \bbM \bSigma_{f} \bbM] \bbE [g (W)^{2}] \\
& + \trace [\bSigma \bbM \bSigma_{f} \bbM \bSigma \bbE [\nabla^{2} \nu_{g}^{2} (\bX)]].
\end{align*}
Again, from \eqref{stein quad}, we have
\begin{align*}
\bSigma_{f} = \bSigma \bbE [f (W)^{2}] + \bm{\mu} \bm{\mu}^{\top} \bbE [f (W)^{2}] + \bm{\mu} \bbE [\nabla \nu_{f}^{2} (\bX)]^{\top} \bSigma + \bSigma \bbE [\nabla \nu_{f}^{2} (\bX)] \bm{\mu}^{\top} + \bSigma \bbE [\nabla^{2} \nu_{f}^{2} (\bX)] \bSigma.
\end{align*}
Combining the above, we have
\begin{align*}
J_{1} = & \ \bm{\mu}^{\top} \bbM \bSigma \bbM \bm{\mu} \bbE [f (W)^{2}] \bbE [g (W)^{2}] + (\bm{\mu}^{\top} \bbM \bm{\mu})^{2} \bbE [f (W)^{2}] \bbE [g (W)^{2}] + 2 \bm{\mu}^{\top} \bbM \bm{\mu} \bbE [\nabla \nu_{f}^{2} (\bX)^{\top}] \bSigma \bbM \bm{\mu} \bbE [g (W)^{2}] \\
& + \bm{\mu}^{\top} \bbM \bSigma \bbE [\nabla^{2} \nu_{f}^{2} (\bX)] \bSigma \bbM \bm{\mu} \bbE [g (W)^{2}] + 2 \bm{\mu}^{\top} \bbM \bSigma \bbM \bSigma \bbE [\nabla \nu_{g}^{2} (\bX)] \bbE [f (W)^{2}] \\
& + 2 \bm{\mu}^{\top} \bbM \bm{\mu} \bm{\mu}^{\top} \bbM \bSigma \bbE [\nabla \nu_{g}^{2} (\bX)] \bbE [f (W)^{2}] + 2 \bm{\mu}^{\top} \bbM \bm{\mu} \bbE [\nabla \nu_{f}^{2} (\bX)^{\top}] \bSigma \bbM \bSigma \bbE [\nabla \nu_{g}^{2} (\bX)] \\
& + 2 \bm{\mu}^{\top} \bbM \bSigma \bbE [\nabla \nu_{f}^{2} (\bX)] \bm{\mu}^{\top} \bbM \bSigma \bbE [\nabla \nu_{g}^{2} (\bX)] + 2 \bm{\mu}^{\top} \bbM \bSigma \bbE [\nabla^{2} \nu_{f}^{2} (\bX)] \bSigma \bbM \bSigma \bbE [\nabla \nu_{g}^{2} (\bX)] \\
& + \trace [\bSigma \bbM \bSigma \bbM] \bbE [f (W)^{2}] \bbE [g (W)^{2}] + \trace [\bSigma \bbM \bm{\mu} \bm{\mu}^{\top} \bbM] \bbE [f (W)^{2}] \bbE [g (W)^{2}] \\
& + \trace [\bSigma \bbM \bm{\mu} \bbE [\nabla \nu_{f}^{2} (\bX)^{\top}] \bSigma \bbM] \bbE [g (W)^{2}] + \trace [\bSigma \bbM \bSigma \bbE [\nabla \nu_{f}^{2} (\bX)] \bm{\mu}^{\top} \bbM] \bbE [g (W)^{2}] \\
& + \trace [\bSigma \bbM \bSigma \bbE [\nabla^{2} \nu_{f}^{2} (\bX)] \bSigma \bbM] \bbE [g (W)^{2}] + \trace [\bSigma \bbM \bSigma \bbM \bSigma \bbE [\nabla^{2} \nu_{g}^{2} (\bX)]] \bbE [f (W)^{2}] \\
& + \trace [\bSigma \bbM \bm{\mu} \bm{\mu}^{\top} \bbM \bSigma \bbE [\nabla^{2} \nu_{g}^{2} (\bX)]] \bbE [f (W)^{2}] + \trace [\bSigma \bbM \bm{\mu} \bbE [\nabla \nu_{f}^{2} (\bX)^{\top}] \bSigma \bbM \bSigma \bbE [\nabla^{2} \nu_{g}^{2} (\bX)]] \\
& + \trace [\bSigma \bbM \bSigma \bbE [\nabla \nu_{f}^{2} (\bX)] \bm{\mu}^{\top} \bbM \bSigma \bbE [\nabla^{2} \nu_{g}^{2} (\bX)]] + \trace [\bSigma \bbM \bSigma \bbE [\nabla^{2} \nu_{f}^{2} (\bX)] \bSigma \bbM \bSigma \bbE [\nabla^{2} \nu_{g}^{2} (\bX)]] \\
= & \ 2 \bm{\mu}^{\top} \bbM \bSigma \bbM \bm{\mu} \bbE [f (W)^{2}] \bbE [g (W)^{2}] + (\bm{\mu}^{\top} \bbM \bm{\mu})^{2} \bbE [f (W)^{2}] \bbE [g (W)^{2}] \\
& + 2 \bm{\mu}^{\top} \bbM \bm{\mu} \bbE [\nabla \nu_{f}^{2} (\bX)^{\top}] \bSigma \bbM \bSigma \bbE [\nabla \nu_{g}^{2} (\bX)] + 2 \bm{\mu}^{\top} \bbM \bSigma \bbE [\nabla \nu_{f}^{2} (\bX)] \bm{\mu}^{\top} \bbM \bSigma \bbE [\nabla \nu_{g}^{2} (\bX)] \\
& + 2 \bm{\mu}^{\top} \bbM \bm{\mu} \bbE [\nabla \nu_{f}^{2} (\bX)^{\top}] \bSigma \bbM \bm{\mu} \bbE [g (W)^{2}] + 2 \bm{\mu}^{\top} \bbM \bm{\mu} \bbE [\nabla \nu_{g}^{2} (\bX)^{\top}] \bSigma \bbM \bm{\mu} \bbE [f (W)^{2}] \\
& + \bm{\mu}^{\top} \bbM \bSigma \bbE [\nabla^{2} \nu_{f}^{2} (\bX)] \bSigma \bbM \bm{\mu} \bbE [g (W)^{2}] + \bm{\mu}^{\top} \bbM \bSigma \bbE [\nabla^{2} \nu_{g}^{2} (\bX)] \bSigma \bbM \bm{\mu} \bbE [f (W)^{2}] \\
& + 2 \bm{\mu}^{\top} \bbM \bSigma \bbM \bSigma \bbE [\nabla \nu_{g}^{2} (\bX)] \bbE [f (W)^{2}] + 2 \bm{\mu}^{\top} \bbM \bSigma \bbM \bSigma \bbE [\nabla \nu_{f}^{2} (\bX)] \bbE [g (W)^{2}] \\
& + 2 \bm{\mu}^{\top} \bbM \bSigma \bbE [\nabla^{2} \nu_{f}^{2} (\bX)] \bSigma \bbM \bSigma \bbE [\nabla \nu_{g}^{2} (\bX)] + 2 \bm{\mu}^{\top} \bbM \bSigma \bbE [\nabla^{2} \nu_{g}^{2} (\bX)] \bSigma \bbM \bSigma \bbE [\nabla \nu_{f}^{2} (\bX)] \\
& + \trace [\bSigma \bbM \bSigma \bbM] \bbE [f (W)^{2}] \bbE [g (W)^{2}] + \trace [\bSigma \bbM \bSigma \bbE [\nabla^{2} \nu_{f}^{2} (\bX)] \bSigma \bbM \bSigma \bbE [\nabla^{2} \nu_{g}^{2} (\bX)]] \\
& + \trace [\bSigma \bbM \bSigma \bbM \bSigma \bbE [\nabla^{2} \nu_{f}^{2} (\bX)]] \bbE [g (W)^{2}] + \trace [\bSigma \bbM \bSigma \bbM \bSigma \bbE [\nabla^{2} \nu_{g}^{2} (\bX)]] \bbE [f (W)^{2}] \\
= & \ O (1) + \trace [\bSigma \bbM \bSigma \bbM] \bbE [f (W)^{2}] \bbE [g (W)^{2}] + \trace [\bSigma \bbM \bSigma \bbE [\nabla^{2} \nu_{f}^{2} (\bX)] \bSigma \bbM \bSigma \bbE [\nabla^{2} \nu_{g}^{2} (\bX)]] \\
& + \trace [\bSigma \bbM \bSigma \bbM \bSigma \bbE [\nabla^{2} \nu_{f}^{2} (\bX)]] \bbE [g (W)^{2}] + \trace [\bSigma \bbM \bSigma \bbM \bSigma \bbE [\nabla^{2} \nu_{g}^{2} (\bX)]] \bbE [f (W)^{2}].
\end{align*}
By symmetry,
\begin{align*}
J_{2} = & \ O (1) + \trace [\bSigma \bbM \bSigma \bbM] \bbE^{2} [f (W) g (W)] + \trace [(\bSigma \bbM \bSigma \bbE [\nabla^{2} \nu_{\sqrt{f g}}^{2} (\bX)])^{2}] \\
& + 2 \trace [\bSigma \bbM \bSigma \bbM \bSigma \bbE [\nabla^{2} \nu_{\sqrt{f g}}^{2} (\bX)]] \bbE [f (W) g (W)].
\end{align*}
Finally, $J_{3} = O (1)$ is the same as $I_{0}$.

Conditions (3) -- (7) only concern rates instead of actual convergence to a particular (non-trivial) limit and under our conditions, they are all satisfied. We start by checking Conditions (3) \& (6). For (3), we need to show the following terms vanishing to 0:
\begin{align*}
& \ n \frac{1}{n^{2}} \bbE \left\{ f (W) \bX^{\top} \bbM \bbE [\bX g (W)] \right\}^{4} \\
= & \ \frac{1}{n} \bbE \left\{ \eta_{f} (\bX) (\bX^{\top} \bbM \bbE [\bX g (W)])^{4} \right\} \\
\leq & \ \frac{1}{n} \bbE^{1 / 2} \eta_{f} (\bX)^{2} \bbE^{1 / 2} (\bX^{\top} \bbM \bbE [\bX g (W)])^{8} \lesssim \frac{1}{n} \rightarrow 0,
\end{align*}
where the last inequality follows from Lemma~\ref{lem:p}. By symmetry, we also have
\begin{align*}
n \frac{1}{n^{2}} \bbE \left\{ g (W) \bX^{\top} \bbM \bbE [\bX f (W)] \right\}^{4} \rightarrow 0.
\end{align*}

For (6), we need to show the following term vanishing to 0. Recall that $\eta_{f} (\bX) \coloneqq \bbE [f (W)^{4} | \bX]$ and $\eta_{g} (\bX) \coloneqq \bbE [g (W)^{4} | \bX]$.
\begin{align*}
& \ n^{2} \frac{1}{n^{2} (n - 1)^{4}} \bbE \left\{ f (W_{1})^{4} g (W_{2})^{4} \{\bX_{1}^{\top} \bbM \bX_{2}\}^{4} \right\} \\
\leq & \ \frac{1}{n^{4}} \bbE \left\{ \eta_{f} (\bX_{1}) \eta_{g} (\bX_{2}) \{\bX_{1}^{\top} \bbM \bX_{2}\}^{4} \right\} \\
\lesssim & \ \frac{p^{3}}{n^{4}} \rightarrow 0
\end{align*}
where the last inequality follows from Lemma~\ref{lem:p}.

Next, we check Condition (4). We need to show the following terms vanishing to 0:
\begin{align*}
& \ n^{3} \frac{1}{n^{2} (n - 1)^{4}} \bbE \left\{ f (W_{1})^{2} \bX_{1}^{\top} \bbM \bbE [g (W_{2})^{2} \bX_{2} \bX_{2}^{\top}] \bbM \bX_{1} \right\}^{2} \\
= & \ C \frac{1}{n^{3}} \bbE \left\{ f (W_{1})^{4} \left( \bX_{1}^{\top} \bbM \bbE [g (W_{2})^{2} \bX_{2} \bX_{2}^{\top}] \bbM \bX_{1} \right)^{2} \right\} \\
= & \ C \frac{1}{n^{3}} \bbE \left\{ \eta_{f} (\bX_{1}) \left( \bX_{1}^{\top} \bbM \bbE [\{\nu_{g}^{2} (\bX_{2}) + \zeta_{g}^{2} (\bX_{2})\} \bX_{2} \bX_{2}^{\top}] \bbM \bX_{1} \right)^{2} \right\} \\
\lesssim & \ \frac{p^{2}}{n^{3}} \rightarrow 0,
\end{align*}
where the last inequality follows from a direct application of \eqref{matrix} from Corollary~\ref{cor:stein}. By symmetry, we have
\begin{align*}
n^{3} \frac{1}{n^{2} (n - 1)^{4}} \bbE \left\{ g (W_{1})^{2} \bX_{1}^{\top} \bbM \bbE [f (W_{2})^{2} \bX_{2} \bX_{2}^{\top}] \bbM \bX_{1} \right\}^{2} \rightarrow 0.
\end{align*}

Then we check Condition (5). We need to show the following terms vanishing to 0:
\begin{align*}
& \ n^{4} \frac{1}{n^{2} (n - 1)^{4}} \bbE \left\{ f (W_{1}) \bX_{1}^{\top} \bbM \bbE [f (W) g (W) \bX \bX^{\top}] \bbM \bX_{2} g (W_{2}) \right\}^{2} \\
= & \ \frac{C}{n^{2}} \bbE \left\{ \bbE [f (W_{1})^{2} | \bX_{1}] \bX_{1}^{\top} \bbM \bbE [f (W) g (W) \bX \bX^{\top}] \bbM \bbE [g (W_{2})^{2} \bX_{2} \bX_{2}^{\top}] \bbM \bbE [f (W) g (W) \bX \bX^{\top}] \bbM \bX_{1} \right\} \\
\lesssim & \ \frac{p}{n^{2}} \rightarrow 0,
\end{align*}
where the last inequality follows from a direct application of \eqref{matrix} from Corollary~\ref{cor:stein}. By symmetry, we have
\begin{align*}
n^{4} \frac{1}{n^{2} (n - 1)^{4}} \bbE \left\{ g (W_{1}) \bX_{1}^{\top} \bbM \bbE [f (W)^{2} \bX \bX^{\top}] \bbM \bX_{2} g (W_{2}) \right\}^{2} \rightarrow 0
\end{align*}
and
\begin{align*}
n^{4} \frac{1}{n^{2} (n - 1)^{4}} \bbE \left\{ f (W_{1}) \bX_{1}^{\top} \bbM \bbE [g (W)^{2} \bX \bX^{\top}] \bbM \bX_{2} f (W_{2}) \right\}^{2} \rightarrow 0.
\end{align*}

Finally, we check Condition (7). We need to show the following terms vanishing to 0:
\begin{align*}
& \ n^{3} \frac{1}{n^{2} (n - 1)^{2}} \bbE \left\{ f (W) \bX^{\top} \bbM \bbE [g (W) f (W) \bX \bX^{\top}] \bbM \bbE [\bX g (W)] \right\}^{2} \\
= & \ \frac{1}{n} \bbE f (W)^{2} \left( \bX^{\top} \bbM \bbE [\bbE [g (W) f (W) | \bX] \bX \bX^{\top}] \bbM \bbE [\bX g (W)] \right)^{2} \\
\lesssim & \ \frac{1}{n} \bbE^{1 / 2} f (W)^{4} \bbE^{1 / 2} \left( \bX^{\top} \bbM \bbE [\bbE [g (W) f (W) | \bX] \bX \bX^{\top}] \bbM \bbE [\bX g (W)] \right)^{4} \lesssim \frac{1}{n} \rightarrow 0,
\end{align*}
where the last inequality follows from Lemma~\ref{lem:p}.
\end{proof}

\begin{lemma}
\label{lem:p}
Under the assumptions of Proposition~\ref{prop:general clt}, we have, for $k \leq 8$,
\begin{align*}
& \bbE \left\{ \bX^{\top} \bbM \bbE [\bX f (W)] \right\}^{k} = O (1), \bbE \left\{ \bX^{\top} \bbM \bbE [\bX g (W)] \right\}^{k} = O (1), \\
& \bbE \left\{ \eta_{f} (\bX_{1}) \eta_{g} (\bX_{2}) \{\bX_{1}^{\top} \bbM \bX_{2}\}^{4} \right\} \lesssim p^{3}.
\end{align*}
\end{lemma}

\begin{proof}
Without loss of generality, we assume $\bm{\mu} = \bm{0}$ and $\bSigma = \bI$. We start by showing the second part of this lemma. By repeatedly applying the results in Corollary~\ref{cor:stein}, we have, with $C$ a constant independent of $n$ changing from line to line,
\begin{align*}
& \ \bbE \left\{ \eta_{f} (\bX_{1}) \eta_{g} (\bX_{2}) \{\bX_{1}^{\top} \bbM \bX_{2}\}^{4} \right\} \\
= & \ C \bbE \left\{ \bX_{2}^{\top} \bbM^{2} \bX_{2} \eta_{f} (\bX_{1}) \eta_{g} (\bX_{2}) \{\bX_{1}^{\top} \bbM \bX_{2}\}^{2} \right\} + C \bbE \left\{ \nabla \eta_{f} (\bX_{1})^{\top} \bbM \bX_{2} \eta_{g} (\bX_{2}) \{\bX_{1}^{\top} \bbM \bX_{2}\}^{3} \right\} \\
= & \ C \bbE \left\{ \eta_{g} (\bX_{2}) \bX_{2}^{\top} \bbM^{2} \bX_{2} \bX_{2}^{\top} \bbM \bbE [\eta_{f} (\bX_{1}) \bX_{1} \bX_{1}^{\top}] \bbM \bX_{2} \right\} + C \bbE \left\{ \bX_{2}^{\top} \bbM \bX_{2} \nabla \eta_{f} (\bX_{1})^{\top} \bbM \eta_{g} (\bX_{2}) \bX_{2} \bX_{2}^{\top} \bbM \bX_{1} \right\} \\
& + C \bbE \left\{ \bX_{2}^{\top} \bbM \nabla^{2} \eta_{f} (\bX_{1}) \bbM \bX_{2} \eta_{g} (\bX_{2}) \{\bX_{1}^{\top} \bbM \bX_{2}\}^{2} \right\} \\
= & \ C \bbE \left\{ \eta_{g} (\bX_{2}) \bX_{2}^{\top} \bbM^{2} \bX_{2} \bX_{2}^{\top} \bbM \bbE [\nabla^{2} \eta_{f} (\bX_{1})] \bbM \bX_{2} \right\} + C \bbE \left\{ \bX_{2}^{\top} \bbM \bX_{2} \eta_{g} (\bX_{2}) \bX_{2}^{\top} \bbM \bbE [\nabla^{2} \eta_{f} (\bX_{1})] \bbM \bX_{2} \right\} \\
& + C \bbE \left\{ \bX_{2}^{\top} \bbM \nabla^{2} \eta_{f} (\bX_{1}) \bbM \bX_{2} \eta_{g} (\bX_{2}) \{\bX_{1}^{\top} \bbM \bX_{2}\}^{2} \right\} \lesssim p^{3}.
\end{align*}
For the first part, we only need to show $k = 4$ and $k = 2$ is a direct consequence. Again, we repeatedly apply Corollary~\ref{cor:stein} and obtain
\begin{align*}
& \ \bbE [\bX^{\top} \bbM \bbE [\bX f (W)] \{\bX^{\top} \bbM \bbE [\bX f (W)]\}^{3}] \\
= & \ \bbE [\bbE [f (W) \bX^{\top}] \bbM^{2} \bbE [\bX f (W)] \{\bX^{\top} \bbM \bbE [\bX f (W)]\}^{2}] \\
= & \ \bbE \left[ \left( \bbE [f (W) \bX^{\top}] \bbM^{2} \bbE [\bX f (W)] \right)^{3} \right] = O (1).
\end{align*}
\end{proof}

For the concrete examples considered in the main text with $\mu_{f}, \mu_{g}$ being GLMs, we need to check if the terms in \eqref{check conditions 1} or \eqref{check conditions 2} in Proposition~\ref{prop:general clt} have nontrivial limits, in the special case $\bbM = \bSigma^{-1}$. Then \eqref{check conditions 1} and \eqref{check conditions 2}, respectively, reduce to
\begin{subequations}
\label{simple check conditions 1}
\begin{align}
& \bm{\mu}^{\top} \bSigma^{-1} \bm{\mu}, \bm{\mu}^{\top} \bbE [\nabla^{2} \nu_{q}^{2} (\bX)] \bm{\mu}, \bbE [q (W)], \bbE [q (W)^{2}], \\
& \bm{\mu}^{\top} \bbE [\nabla \mu_{q} (\bX)], \bm{\mu}^{\top} \bbE [\nabla \nu_{q}^{2} (\bX)], \bm{\mu}^{\top} \bbE [\nabla^{2} \nu_{q'}^{2} (\bX)] \bSigma \bbE [\nabla \mu_{q} (\bX)] \\
& \bbE [\nabla \mu_{q} (\bX)]^{\top} \bSigma \bbE [\nabla \mu_{q'} (\bX)], \bbE [\nabla \nu_{q}^{2} (\bX)]^{\top} \bSigma \bbE [\nabla \mu_{q'} (\bX)], \\
& \bbE [\nabla \mu_{q} (\bX)]^{\top} \bSigma \bbE [\nabla^{2} \nu_{q''}^{2} (\bX)] \bSigma \bbE [\nabla \mu_{q'} (\bX)]
\end{align}
\end{subequations}
and
\begin{subequations}
\label{simple check conditions 2}
\begin{align}
& \bbE [q (W)^{2}] \bbE [q' (W)^{2}], p^{-1} \trace [\bSigma \bbE [\nabla^{2} \nu_{q}^{2} (\bX)]], \\
& p^{-1} \trace [\bSigma \bbE [\nabla^{2} \nu_{q}^{2} (\bX)] \bSigma \bbE [\nabla^{2} \nu_{q'}^{2} (\bX)]].
\end{align}
\end{subequations}

Thus for the four specific cases, we just need to check the above two conditions for all the moment estimators involved.

\subsection{Linear and Quadratic Forms of GLMs}
\label{app:GLM functional CLT}

The key step is to simplify $\bbE [\nabla \mu (\bX)]$, $\bbE [\nabla \nu^{2} (\bX)]$ and $\bbE [\nabla^{2} \nu^{2} (\bX)]$; note that here we have considered the joint convergence of the $U$-statistic-based estimators of the involved moments in \eqref{GLM chain} using the Cram\'{e}r-Wold device. To this end, we have
\begin{align*}
\bbE [\nabla \mu (\bX)] & = \bbE [\nabla \phi (\bbeta^{\top} \bX)] = \bbE [\phi' (\bbeta^{\top} \bX)] \bX \\
& = \bbE [\phi' (\bbeta^{\top} \bX)] \bm{\mu} + \bbE [\phi' (\bbeta^{\top} \bX)] \bSigma \bbeta, \\
& = (\bm{\mu} + \bSigma \bbeta) \bbE [\mu' (\bX)] \\
\bbE [\nabla \nu^{2} (\bX)] & = \bbE [\nabla \{\phi^{2} (\bbeta^{\top} \bX) + \sigma^{2} (\bbeta^{\top} \bX)\}] \\
& = \{\bbE [\phi^{2}{}' (\bbeta^{\top} \bX)] + \bbE [\sigma^{2}{}' (\bbeta^{\top} \bX)]\} \cdot (\bm{\mu} + \bSigma \bbeta) \\
& = (\bm{\mu} + \bSigma \bbeta) \bbE [\nu^{2}{}' (\bX)], \\
\bbE [\nabla^{2} \nu^{2} (\bX)] & = \bbE [\nabla^{2} \{\phi^{2} (\bbeta^{\top} \bX) + \sigma^{2} (\bbeta^{\top} \bX)\}] \\
& = (\bm{\mu} + \bSigma \bbeta) \{\bbE [\nabla \phi^{2}{}' (\bbeta^{\top} \bX)] + \bbE [\nabla \sigma^{2}{}' (\bbeta^{\top} \bX)]\} \\
& = (\bm{\mu} + \bSigma \bbeta) (\bm{\mu} + \bSigma \bbeta)^{\top} \{\bbE [\phi^{2}{}'' (\bbeta^{\top} \bX)] + \bbE [\sigma^{2}{}'' (\bbeta^{\top} \bX)]\} \\
& = (\bm{\mu} + \bSigma \bbeta) (\bm{\mu} + \bSigma \bbeta)^{\top} \{\bbE [\nu^{2}{}'' (\bX)]\}.
\end{align*}

Checking Condition~\ref{simple check conditions 1}, we need the following terms to converge to certain limits or $o (1)$:
\begin{align*}
\bm{\mu}^{\top} \bSigma^{-1} \bm{\mu}, \bm{\mu}^{\top} \bm{\mu}, \bm{\mu}^{\top} \bbeta, \bm{\mu}^{\top} \bSigma \bbeta, \bm{\mu}^{\top} \bSigma^{2} \bbeta, \bbeta^{\top} \bSigma \bbeta, \bbeta^{\top} \bSigma^{3} \bbeta.
\end{align*}
Checking Condition~\ref{simple check conditions 2}, we need the following terms to converge to certain limits:
\begin{align*}
\bm{\mu}^{\top} \bbeta, \bbeta^{\top} \bSigma \bbeta.
\end{align*}
When $\bm{\mu} = \bm{0}$, the above conditions reduce to $\bbeta^{\top} \bSigma \bbeta$ and/or $\bbeta^{\top} \bSigma^{3} \bbeta$ converge to a limit.

We then need to check \eqref{cond:clt}. Under Assumptions~\ref{as:bounded} and~\ref{as:var-cov}, both requirements in \eqref{cond:clt} hold due to GLM and boundedness of $\Vert \bbeta \Vert$.

\subsection{Causal Effects in Linear Structural Equation Models}
\label{app:CE CLT}

For the causal effect estimation problem in Section~\ref{sec:CE},
\begin{align*}
\bbE [A Y | \bX] & = \bbE [A \bbE [Y | \bX, A]] = \bbE [A \bbeta^{\top} \bX + \psi \cdot A | \bX] \\
& = \bbeta^{\top} \bX \cdot \eta (\balpha^{\top} \bX) + \psi \cdot \eta (\balpha^{\top} \bX)
\end{align*}
which is a function of $(\balpha^{\top} \bX, \bbeta^{\top} \bX)$. Thus we only need to replace $\bbeta$'s in Appendix~\ref{app:GLM functional CLT} by $\balpha$ or a mix of $\balpha$ and $\bbeta$.

\subsection{Mean Estimation under MAR}
\label{app:MAR CLT}

For the problem of mean estimation under MAR in Section~\ref{sec:MAR}, we only need to additionally consider to replace $\bbeta$'s in Appendix~\ref{app:GLM functional CLT} by $\balpha$ or a mix of $\balpha$ and $\bbeta$. The derivation is hence omitted.

\subsection{Generalized Covariance Measure}
\label{app:GCM CLT}

For the problem of estimating GCM in Section~\ref{sec:GCM}, as we indicate after Theorem~\ref{thm:GCM}, we only need to introduce additional terms that involve the inner product between $\bm{\theta}^{\top} f (\bSigma)$ and $\balpha, \bbeta$ or $\bm{\mu}$ where $f$ is either the linear, the quadratic or the cubic function.

\section{Proof Related to Universality}
\label{app:universality}

To extend the results from Gaussian designs to non-Gaussian designs, the following lemma from \citet{chatterjee2006generalization} (also see \citet{han2023universality}) is pivotal.

\begin{lemma}
\label{lemma:stein_universal_identity}

Let $\bX \coloneqq (X_{1}, \cdots, X_{p})^{\top}$ and $\bZ \coloneqq (Z_{1}, \cdots, Z_{p})^{\top}$ be two collections of random vectors with independent coordinates and matching first and second moments: $\bbE [X_{j}^{k}] = \bbE [Z_{j}^{k}]$ for $j = 1, \cdots, p$ and $k \in \{1, 2\}$. For any three-time differentiable function $f$,
\begin{equation}
\left\vert \bbE f (\bX) - \bbE f (\bZ) \right\vert \leq \sum_{j = 1}^{p} \max_{W_{j} \in \{X_{j}, Z_{j}\}} \left\vert \bbE \int_{0}^{W_{j}} \partial^{3}_{j} f (X_{1}, \cdots, X_{j - 1}, t, Z_{j + 1}, \cdots, Z_{p}) (W_{j} - t)^{2} \diff t \right\vert,
\end{equation}
where $\partial^{3}_{j} f$ denotes the third-derivative of $f$ with respect to the $j$-th argument.
\end{lemma}

\begin{lemma}
\label{lem:app universality}
Suppose that Assumptions~\ref{as:link} and~\ref{as:beta} hold and $\bbM$ is a n.n.s.d. matrix such that $M^{-1} \leq \lambda_{\min} (\bbM) \leq \lambda_{\max} (\bbM) \leq M$. Let $\bZ$ be the corresponding Gaussian vector with the same first and second moments as $\bX$. We then have the following assertions:
 \begin{enumerate}
     \item  $\bbE(\bbeta^\top \bX \phi(\balpha^\top \bX)) = \bbeta^\top \bSigma \balpha \bbE [\phi'(\bZ^\top \balpha)] +\bbeta^\top \bmu \bbE [\phi(\bZ^\top \balpha)] + O(p^{-3/4})$;
     \item $\bbE [\phi (\balpha^\top \bX)] = \bbE [\phi (\balpha^\top \bZ)] + O (p^{-3 / 4})$;
     \item $\bbE [\eta (\balpha^\top \bX) \bX^{\top} \bbM \bbE [\phi (\bX^{\top} \bbeta) \bX \bX^{\top}] \bbM \bX]$;
     \item $\{Y_i, \bX_i\}_{i = 1}^{n} \overset{\rm i.i.d.}{\sim} \bbP$ with $\bbE [Y | \bX] = \phi (\balpha^\top \bX)$ and let $T_n = \bbU_{n, 2} (Y_1 \bX_1^\top \bbM \bX_2 Y_2)$. Then $\var (T_n) = O (1 / n + p / n^2)$ if $\bbE [\sigma^4 (\bX)] < \infty$ with $\sigma^2 (\bX) = \var (Y | \bX)$.
 \end{enumerate}


\end{lemma}
\begin{proof}

We prove assertion 1 and the proof of 2 follows by parallel arguments with obvious modifications. To prove 1, we first start with the case $\bSigma = \bI$ and $\bmu=\mathbf{0}$ and subsequently provide the details necessary for extending to the general mean and covariance cases. To this end define the function $f:\mathbb{R}^p \rightarrow \mathbb{R}$ as $f(\boldsymbol{\omega})=\boldsymbol{\omega}^\top\bbeta\phi(\boldsymbol{\omega}^\top \balpha)$ for $\boldsymbol{\omega}=(\omega,\ldots,\omega_p)^\top\in \bbR^p$. Therefore it is clear that
\begin{align*}
    \frac{\partial^3}{\partial \omega_j^3}f(\boldsymbol{\omega})&=\boldsymbol{\omega}^\top\bbeta\alpha_j^3\phi^{(3)}(\boldsymbol{\omega}^\top\balpha)+3\beta_j\alpha_j^2\phi^{(2)}(\boldsymbol{\omega}^\top\balpha).
\end{align*}
Therefore by Lemma~\ref{lemma:stein_universal_identity} 
\begin{align*}
\ & |\bbE(\bX^\top\bbeta\phi(\bX^\top\alpha))- \bbE(\bY^\top\bbeta\phi(\bY^\top\alpha))|\\
&\leq \sum_{j=1}^p \max_{V_j\in \{X_j,Y_j\}}\left\vert\bbE \int_{0}^{V_j}\frac{\partial^3 }{\partial\omega_j^3}f(X_1,\ldots,X_{j-1},t, Y_{j+1},\ldots,Y_p)(Z_j-t)^2 dt\right\vert
\\
&\leq T_1+T_2,
\end{align*}
where $T_1,T_2$ respect the decomposition of $\frac{\partial^3}{\partial \omega_j^3}f(\boldsymbol{\omega})$ after applying triangle inequality to the display above. We first consider $T_1$ as follows by defining $W_{l,j}(\kappa,t)=\kappa(X_l\mathbbm{1}(l< j)+Y_l\mathbbm{1}(l> j)+ t\mathbbm{1}(l=j))$
\begin{align*}
    T_1&=\sum_{j=1}^p |\alpha_j^3| \max\limits_{V_j \in \{X_j,Y_j\}} \left\vert \bbE\int_{0}^{V_j} \sum_{l = 1}^p W_{l, j} (\beta_l, t) \phi^{(3)} \left(\sum_{l = 1}^p W_{l,j} (\alpha_l, t)\right) (V_j - t)^2 \diff t \right\vert.
\end{align*}
Now by Cauchy-Schwarz inequality,
\begin{equation}
\label{cs}
\begin{split}
  \ & \left\vert \bbE\int_{0}^{V_j} \sum_{l\neq j} W_{l, j} (\beta_l, t) \phi^{(3)} \left(\sum_{l = 1}^p W_{l,j} (\alpha_l, t)\right) (V_j - t)^2 \diff t \right\vert \\
  & \leq \left( \left\vert \bbE \int_{0}^{V_j}\left\{\phi^{(3)} \left(\sum_{l = 1}^p W_{l,j} (\alpha_l, t)\right)\right\}^2\diff t \right\vert \right)^{1/2} \left( \left\vert \bbE \int_{0}^{V_j}\left\{\sum_{l\neq j} W_{l, j} (\beta_l, t)(V_j-t)^2\right\}^2 \diff t \right\vert \right)^{1/2}.
\end{split}
\end{equation}
We analyze the two terms of the product in the last display as follows.
\begin{align*}
   \ & \left\vert\bbE\int_{0}^{V_j}\left\{\phi^{(3)} \left(\sum_{l = 1}^p W_{l,j} (\alpha_l, t)\right)\right\}^2\diff t\right\vert\\
   &\leq \left\vert\bbE\int_{0}^{V_j}\left\{\phi^{(3)} \left(\sum_{l \neq j}^p W_{l,j} (\alpha_l,0)\right)\right\}^2\diff t\right\vert +
   |\alpha_j|\left\vert\bbE\int_{0}^{V_j}tf_{\phi}\left(\sum_{l\neq j}^p W_{l,j} (\alpha_l, 0)+\alpha_j\chi (t)\right)\diff t\right\vert\\
   &=|\alpha_j|\left\vert\bbE\int_{0}^{V_j}tf_{\phi}\left(\sum_{l\neq j}^p W_{l,j} (\alpha_l, 0)+\alpha_j\chi (t)\right)\diff t\right\vert
\end{align*}
where $\chi (t)$ is a $\sum_{l \neq j}^{p} W_{l, j} (\alpha_{l}, 0)$-measurable quantity between $\sum_{l \neq j}^{p} W_{l, j} (\alpha_{l}, 0)$ and $t$ by the exact form of Taylor remainder theorem, and we used the facts that $W_{l,j}(\kappa,t)=W_{l,j}(\kappa,0)$ for $l\neq j$, $\bbE(V_j)=0$ and $V_j \indep \sum_{l\neq j} W_{l,j}(\alpha_l,0)$, and also the short-hand notation $f_{\phi}(\cdot) \coloneqq \{\phi^{(3)} (\cdot)^2\}'$. Now note that by the property of the function $\phi$ and the range of integration over $t$,
\begin{align*}
\ & \left\vert\bbE\int_{0}^{V_j}tf_{\phi}\left(\sum_{l\neq j}^p W_{l,j} (\alpha_l, 0)+\alpha_j\chi (t)\right)\diff t\right\vert \\
& \leq 2\left\vert\bbE\int_{0}^{V_j}|t|\exp\left\{2\left(\sum_{l\neq j}^p W_{l,j} (\alpha_l, 0)+\alpha_j\chi (t)\right)^2 f \left( \sum_{l\neq j}^p W_{l,j} (\alpha_l, 0)+\alpha_j\chi (t) \right) \right\} \diff t \right\vert \\
&\leq 2\left\vert\bbE\int_{0}^{V_j}|V_j|\exp\left\{2\left(\sum_{l\neq j}^p W_{l,j} (\alpha_l, 0)+\alpha_j\chi (t)\right)^2 f \left( \sum_{l\neq j}^p W_{l,j} (\alpha_l, 0)+\alpha_j\chi (t) \right)\right\}\diff t\right\vert.
\end{align*}
Next, we again note that by the exact form of Taylor remainder theorem it is clear that $\chi (t)$ is $\sum_{l\neq j}^p W_{l,j} (\alpha_l, 0)$-measurable and hence $$\exp\left\{2\left(\sum_{l\neq j}^p W_{l,j} (\alpha_l, 0)+\alpha_j\chi (t)\right)^2 f \left( \sum_{l\neq j}^p W_{l,j} (\alpha_l, 0)+\alpha_j\chi (t) \right) \right\} \indep V_j.$$ 

Now by direct calculations and the fact that $\max_{l=1}^p\|V_l\|_{\psi_2}\leq M$, $|\chi (t)| \leq |V_j|$ and $\lim_{|x|\rightarrow \infty} f(x)=0$ it is straightforward to verify that $$2\left\vert\bbE\int_{0}^{V_j}|V_j|\exp\left\{2\left(\sum_{l\neq j}^p W_{l,j} (\alpha_l, 0)+\alpha_j\chi (t)\right)^2 f \left( \sum_{l\neq j}^p W_{l,j} (\alpha_l, 0)+\alpha_j\chi (t) \right)\right\}\diff t\right\vert$$ is bounded. Similarly using the fact that $\|\sum_{l\neq j}W_{l,j}(\beta_l,t)\|_{\psi_2}\leq  M$ and it is independent of $V_j$ which also has $\psi_2$-norm bounded by $M$ one has by straightforward algebra that $$\left\vert\int_{0}^{V_j}\bbE\left\{\sum_{l\neq j} W_{l, j} (\beta_l, t)(V_j-t)^2\right\}^2\diff t\right\vert$$ is bounded as well. Consequently there exists a constant $C(M)$ depending on $M$ such that 
\begin{align*}
     \left\vert \bbE\int_{0}^{V_j} \sum_{l\neq j} W_{l, j} (\beta_l, t) \phi^{(3)} \left(\sum_{l = 1}^p W_{l,j} (\alpha_l, t)\right) (V_j - t)^2 \diff t \right\vert \leq C(M)|\alpha_j|.
\end{align*}
Therefore for some constant $C'(M)$ depending on $M$
\begin{align*}
    T_1 &\leq C'(M)\sum_{j=1}^p |\alpha_j|^{7/2}=O(p^{-3/4}).
\end{align*}
By a parallel argument, it is easy to show that for some constant $C''(M)$ depending on $M$
\begin{align*}
    T_2 &\leq C''(M)\sum_{j=1}^p |\beta_j||\alpha_j|^{5/2}=O(p^{-3/4}).
\end{align*}
This completes the proof of the first statement. The second statement follows directly from the same proof strategy and is thus omitted. It is noteworthy that if $\bbE [V_{j}] = \mu_{j}$ is not necessarily 0 when $\bm{\mu} \neq \bm{0}$, we can still apply the same arguments by further controlling the first term on the LHS of \eqref{cs} by leveraging the assumption on $\bm{\mu}$ imposed in Assumption~\ref{as:Sigma}. When $\bSigma \neq \bI$, the above proof also goes through by renaming $\balpha$ and $\bbeta$ as $\bSigma^{1 / 2} \balpha$ and $\bSigma^{1 / 2} \bbeta$.

Finally, we prove the last claim that, for $T_n = \bbU_{n,2} (Y_1 \bX_1^\top \bbM \bX_2 Y_2)$, $\var (T_n) = O(1/n+p/n^2)$ when $\bX = \bSigma^{1/2} \bZ + \bmu$. We note that by Assumption~\ref{as:Sigma} it is enough to consider the case $\bmu=\mathbf{0}$, which we assume henceforth. Thereafter, let $T_{1n}$ and $T_{2n}$ denote the first and second order term of the Hoeffding decomposition of $T_n$. Note that the flavor of this claim is different from the previous ones, as it is not about approximating the results that hold exactly under Gaussianity by Stein's lemma. We first show that $\var(T_{1n})=O(1/n)$. To show this, note that with $\upsilon_{\bbM} \coloneqq \bbM\bbE(\bX Y) = \bbM\bbE(\bX\phi(\bX^\top\bbeta))$ we have that
\begin{align*}
    \var(T_{1n})&=\frac{1}{n}\var\left(Y\bX^\top\upsilon_{\bbM}\right)\\
    &=\frac{1}{n}\var(\phi(\bX^\top\bbeta)\bX^\top \upsilon_{\bbM})+\frac{1}{n}\bbE(\sigma^2(\bX)(\upsilon_{\bbM}^\top\bX)^2)
\end{align*}
Now since $\|\upsilon_{\bbM}^\top\bX\|_{\psi_2}\leq M\|\upsilon_{\bbM}\|_2$, $\|\upsilon_{\bbM}\|_2^2\leq \lambda_{\max}^2(\bbM)\|\bbE(\bX^\top\phi(\bX^\top\bbeta))\|_2^2$. But $\bX^\top\bSigma^{-1}\bbE(\bX\phi(\bX^\top\bbeta))=\mathrm{Proj}_{L_2(\bbP)}\left(\phi(\bX^\top\bbeta)|V\right)$ where $V=\mathrm{Span}(\bX)$. Therefore by the length contraction property of projection, we have
\begin{align*}
\|\bX^\top\bSigma^{-1}\bbE(\bX\phi(\bX^\top\bbeta))\|_{L_2(\bbP)}^2\leq \|\phi(\bX^\top\bbeta)\|_{L_2(\bbP)}^2.
\end{align*}
However,
\begin{align*}
\|\bX^\top\bSigma^{-1}\bbE(\bX\phi(\bX^\top\bbeta))\|_{L_2(\bbP)}^2&=\bbE(\bX^\top\phi(\bX^\top\bbeta))\bSigma^{-1}\bbE(\bX\phi(\bX^\top\bbeta))\\
&\geq \lambda_{\max}^{-1}(\bSigma)\|\bbE(\bX\phi(\bX^\top\bbeta))\|_2^2.
\end{align*}
Therefore
\begin{align*}
   \|\upsilon_{\bbM}\|_2^2\leq \lambda_{\max}^2(\bbM)\|\bbE(\bX^\top\phi(\bX^\top\bbeta))\|_2^2 \leq  \lambda_{\max}(\bSigma)\lambda_{\max}(\bbM)\|\phi(\bX^\top\bbeta)\|_{L_2(\bbP)}^2<\infty.
\end{align*}
Thus we have
\begin{align*}\var(\phi(\bX^\top\bbeta)\bX^\top \upsilon_{\bbM})&\leq 
\bbE^{1/2}\left(\phi^4(\bX^\top\bbeta)\right)\bbE^{1/2}\left((\upsilon_{\bbM}^\top\bX)^4\right)<\infty, \\
\bbE(\sigma^2(\bX)(\upsilon_{\bbM}^\top\bX)^2)&\leq \bbE^{1/2}\left(\sigma^4(\bX)\right)\bbE^{1/2}\left((\upsilon_{\bbM}^\top\bX)^4\right)<\infty.
\end{align*}
This completes the proof of the fact that $\var(T_{1n})=O(1/n)$. Next, wring $\bZ_{l,M}=\bbM^{1/2}\bSigma^{1/2}\bZ_l$, we note that $\|\bX_1^\top \bbM\bX_2\|^2_{\psi_1}\leq 2 \|\|\bZ_{1,\bbM}+\bZ_{2,\bbM}\|_2^2\|_{\psi_1}^2+2\|\|\bZ_{1,\bbM}-\bZ_{2,\bbM}\|_{2}^2\|_{\psi_1}^2\leq 16\sum_{j=1}^p\|Z_{1,M}^2(j)\|_{\psi_1}^2\leq 16 M^2 p$ for some $C$ depending on $M$. Therefore for any $k\in \mathbb{N}$ we have $\bbE\left\{\bX_1^{\top}\bbM\bX_2\right\}^k\leq C'(M)p^{k/2}$ for some $C'$ depending on $M$. Consequently, we have by Jensen's inequality 
\begin{align*}
    \var(T_{2n})&\leq \frac{C}{n^2}\bbE\left(Y_i^2Y_j^2(\bX_i^{\top}\bbM\bX_j)^2\right)\\
    &=\frac{C}{n^2}\bbE\left[\left\{\sigma^{2}(\bX_1)\sigma^{2}(\bX_2)+2\sigma^{2}(\bX_1)\phi^2(\bX_2^{\top}\bbeta)+\phi^2(\bX_1^{\top}\bbeta)\phi^2(\bX_2^{\top}\bbeta)\right\}\left\{\bX_1^{\top}\bbM\bX_2\right\}^2\right]\\
    &\leq \frac{C}{n^2}\bbE^{1/2}\left[\left\{\sigma^{2}(\bX_1)\sigma^{2}(\bX_2)+2\sigma^{2}(\bX_1)\phi^2(\bX_2^{\top}\bbeta)+\phi^2(\bX_1^{\top}\bbeta)\phi^2(\bX_2^{\top}\bbeta)\right\}^2\right]\bbE^{1/2}\left[\left(\bX_1^{\top}\bbM\bX_2\right)^4\right]\\
    &\leq \frac{C''(M)p}{n^2},
\end{align*} 
for some $C''$ depending on $M$. This completes the proof of the lemma.
\end{proof}

With Lemma~\ref{lem:app universality}, we are ready to prove Lemma~\ref{lem:mu, universality}, and as a direct corollary of Lemma~\ref{lem:mu, universality}, Lemma~\ref{lem:glm mean zero, universality}.

\begin{proof}[Proof of Lemma~\ref{lem:mu, universality}]
It is easy to see that to show Lemma~\ref{lem:mu, universality}, it is sufficient to apply Lemma~\ref{lem:app universality} to each individual moment equation in \eqref{GLM chain} separately. In particular, the proof is complete once one observes that the LHS of each identity in \eqref{GLM chain} either has the form of $\bbE [\bbeta^{\top} \bX \phi (\balpha^{\top} \bX)]$ or $\bbE [\phi (\balpha^{\top} \bX)]$ with appropriately chosen $\balpha$, $\bbeta$, and $\phi$.
\end{proof}

\section{Unknown Population Covariance Matrix}
\label{app:unknown}

\subsection{Linear Models}
\label{app:unknown linear}

For notational convenience, we take $\bm{\mu} = \bm{0}$ so the population covariance matrix $\bSigma$ is identical to the population Gram matrix $\bbE [\bX \bX^{\top}] \equiv \bSigma + \bm{\mu} \bm{\mu}^{\top}$. But we consider the situation where the statistician does not know $\bm{\mu} = \bm{0}$. Under the linear model with homoscedastic variance, for estimating the linear and quadratic forms $\lambda_{\beta} = \bbeta^{\top} \bm{\mu}$ and $\gamma_{\beta}^{2} = \bbeta^{\top} \bSigma \bbeta$, we consider the following estimators without knowing $\bSigma$ but under the assumption that the empirical covariance matrix $\hat{\bSigma}$ is invertible:
\begin{align*}
\hat{\lambda}_{\beta} & \coloneqq \bbU_{n, 1} [Y] \equiv \frac{1}{n} \mathbf{1}^{\top} \bY, \\
\hat{\gamma}_{\beta}^{2} & \coloneqq \frac{1}{n - p} \left\{ \bY^{\top} \left( \bbH - \frac{p}{n} \bI \right) \bY \right\},
\end{align*}
where we denote $\bbX \coloneqq (\bX_{1}, \cdots, \bX_{n})^{\top}$ as the $n \times p$ design matrix, $\bbH \coloneqq \bbX (\bbX^{\top} \bbX)^{-1} \bbX$ as the ``hat'' projection matrix, $\bY \coloneqq (Y_{1}, \cdots, Y_{n})^{\top} \in \bbR^{n}$ as the vector collecting the responses over all the subjects, and $\mathbf{1}$ as the $n$-dimensional all-1 vector. The unbiasedness, $\sqrt{n}$-consistency and CAN of $\hat{\lambda}_{\beta}$ is trivial. The unbiasedness, $\sqrt{n}$-consistency and CAN of $\hat{\gamma}_{\beta}^{2}$ follow directly from Theorem 1 of \citet{guo2022moderate}. $\hat{\gamma}_{\beta}^{2}$ can also be viewed as approximating a second-order $U$-statistic with the removed diagonal part approximated by $\frac{1}{n} \bY^{\top} \cdot \frac{p}{n - p} \bI \cdot \bY$.

For a single coordinate $\beta_{j}$, we consider the following estimator:
\begin{align*}
\hat{\beta}_{j} = e_{j}^{\top} (\bbX^{\top} \bbX)^{-1} \bbX^{\top} \bY.
\end{align*}
First, $\hat{\beta}_{j}$ is obviously conditionally, and hence unconditionally, unbiased. We only need to control its variance:
\begin{align*}
\var (\hat{\beta}_{j}) & = \bbE [\var (\hat{\beta}_{j} | \bbX)] + \var (\bbE [\var (\hat{\beta}_{j}) | \bbX]) \\
& = \bbE [\var (e_{j}^{\top} (\bbX^{\top} \bbX)^{-1} \bbX^{\top} \bY | \bbX)].
\end{align*}
Then
\begin{align*}
\var (e_{j}^{\top} (\bbX^{\top} \bbX)^{-1} \bbX^{\top} \bY | \bbX) = \sigma^{2} e_{j}^{\top} (\bbX^{\top} \bbX)^{-1} e_{j}.
\end{align*}
So
\begin{align*}
\var (\hat{\beta}_{j}) = \sigma^{2} e_{j}^{\top} \bbE [(\bbX^{\top} \bbX)^{-1}] e_{j}.
\end{align*}

A common theme of the above estimators is the reliance on the correctness of the homoscedastic linear model. It remains to be seen if a similar strategy can be applied to estimating the moments involved in system of equations such as \eqref{GLM chain} for GLMs or more general nonlinear models, when $\bSigma$ is unknown and $p < n$.

\subsection{Generalized Linear Models}
\label{app:unknown GLMs}

\subsubsection{The case of \texorpdfstring{$p < n$}{well}}
\label{app:well-posed}

In this section, we first prove Proposition~\ref{prop:unknown}. 

\begin{proof}[Proof of Proposition~\ref{prop:unknown}]
In this proof, we take $n / 2$ as $n$ to avoid notation clutter. As a result, the notation in this proof will follow that in Section~\ref{sec:unknown} only. We only need to show that $\hat{\gamma}_{\beta}^{2}$, as defined in Section~\ref{sec:unknown}, is a $\sqrt{n}$-consistent estimator of $\gamma_{\beta}^{2}$. We also let $\bOmega \coloneqq \bSigma^{-1}$ and $\tilde{\bOmega} \coloneqq \tilde{\bSigma}^{-1}$ and each element of $\bOmega$ and $\tilde{\bOmega}$ as, respectively, $\omega_{i, j}$ and $\tilde{\omega}_{i, j}$ for $i, j \in [p]$.

We first make the following important observation, which is another important implication of the Gaussian design \citep{couillet2022random, derezinski2021sparse}.
\begin{lemma}
\label{lem:Wishart}
Under Assumption~\ref{as:normal mean zero}, if $n > p + 3$,
\begin{align*}
\tilde{\bSigma} \sim \Wishart_{p} (n^{-1} \bSigma, n), \tilde{\bSigma}^{-1} \equiv \tilde{\bOmega} \sim \Wishart_{p}^{-1} (n \bSigma^{-1} \equiv n \bOmega, n)
\end{align*}
where $\Wishart_{p} (\bV, m)$ and $\Wishart_{p}^{-1} (\bV, m)$, respectively, denote the $p$-dimensional Wishart and Inverse-Wishart distributions with $m$ degrees-of-freedom and scale matrix $\bV$. Thus
\begin{equation}
\label{mean wishart}
\bbE [\tilde{\bSigma}^{-1}] \equiv \bbE [\tilde{\bOmega}] \equiv \frac{n}{n - p - 1} \bSigma^{-1}
\end{equation}
and for $i, j \in [p]$,
\begin{equation}
\label{cov wishart}
\cov \left( \tilde{\omega}_{i, j}, \tilde{\omega}_{k, l} \right) = \frac{2 n^{2} \omega_{i, j} \omega_{k, l} + (n - p - 1) n^{2} (\omega_{i, k} \omega_{j, l} + \omega_{i, l} \omega_{k, j})}{(n - p) (n - p - 1)^{2} (n - p - 3)}
\end{equation}
\end{lemma}

It is easy to see that $\hat{m}_{\bX Y, 2}$ is an unbiased estimator of $m_{\bX Y, 2}$. We are now left to control its variance. Under Assumption~\ref{as:proportion}, the prefactor of \eqref{alternative} has limit $1 - \delta$, which can be treated as $O (1)$ without loss of generality. We first note that
\begin{align*}
\var (\hat{m}_{\bX Y, 2}) = \var \left( \bbE [\hat{m}_{\bX Y, 2} | \tilde{\bSigma}] \right) + \bbE \left[ \var (\hat{m}_{\bX Y, 2} | \tilde{\bSigma}) \right] \eqqcolon V_{1} + V_{2}.
\end{align*}

We first control $V_{1}$. Let $\bv \coloneqq \bbE [\bX Y]$. Following \eqref{cov wishart}, we have
\begin{align*}
V_{1} \lesssim \var \left( \bv^{\top} \tilde{\bSigma}^{-1} \bv \right) = \sum_{i = 1}^{p} \sum_{j = 1}^{p} \sum_{k = 1}^{p} \sum_{l = 1}^{p} v_{i} v_{j} v_{k} v_{l} \cov \left( \tilde{\omega}_{i, j}, \tilde{\omega}_{k, l} \right) \lesssim \frac{1}{n} (\bv^{\top} \bSigma^{-1} \bv)^{2}.
\end{align*}

Finally, we control $V_{2}$. Conditioning on $\tilde{\bSigma}$, we have, by Hoeffding decomposition and Corollary~\ref{cor:stein},
\begin{align*}
\var (\hat{m}_{\bX Y, 2} | \tilde{\bSigma}) \lesssim \frac{1}{n} \left( \bv^{\top} \bSigma \tilde{\bSigma}^{-1} \bSigma \bv \right)^{2} + \frac{1}{n^{2}} \trace \left( \bSigma \tilde{\bSigma}^{-1} \bSigma \tilde{\bSigma}^{-1} \right).
\end{align*}
Finally, the proof is complete once we marginalize over $\tilde{\bSigma}^{-1}$ using \eqref{cov wishart}:
\begin{align*}
\bbE [\var (\hat{m}_{\bX Y, 2} | \tilde{\bSigma})] \lesssim \frac{1}{n} + \frac{p}{n^{2}} \lesssim \frac{1}{n}.
\end{align*}
\end{proof}

\subsubsection{The case of \texorpdfstring{$p \geq n$}{ill}}
\label{app:ill-posed}

\begin{proof}[Proof of Proposition~\ref{prop:ill posed}]\leavevmode

The constructed estimator as in Section~\ref{sec:unknown ill-posed} crucially relies on the following result on the error of using Chebyshev polynomials to reciprocal $1 / x$.

\begin{lemma}[Rephrasing of Lemma 7.8 or Lemma 7.16 of \citet{orecchia2012approximating}]
\label{lem:chebyshev}
For any $\epsilon > 0$ and $b > a > 0$, let $\bM$ be a $p \times p$ p.s.d. matrix with eigenvalues bounded between $a$ and $b$. Then, there exists a polynomial $q_{J} (\cdot; a, b, \eps) \equiv q_{J} (\cdot)$ of maximum degree $J$ such that:
\begin{equation*}
\| \bM^{-1} - q_{J} (\bM) \|_{\op} \lesssim e^{- J}.
\end{equation*}
As a consequence, the following also holds: given any two vectors $\bv_{1}, \bv_{2} \in \bbR^{p}$,
\begin{equation*}
\left| \bv_{1}^{\top} \bM^{-1} \bv_{2} - \bv_{1}^{\top} q_{J} (\bM) \bv_{2} \right| \lesssim \left\Vert \bv_{1} \right\Vert_{2} \left\Vert \bv_{2} \right\Vert_{2} e^{- J}.
\end{equation*}
In particular, $q_{J} (\cdot)$ can be constructed as follows:
\begin{align*}
q_{J} (x) \coloneqq \frac{1}{x} \left( 1 - \frac{T_{J + 1} (\frac{b + a - 2 x}{b - a})}{T_{J + 1} (\frac{b + a}{b - a})} \right),
\end{align*}
where $T_{J + 1} (\cdot)$ is the Chebyshev polynomial of degree $J + 1$.
\end{lemma}

Armed with the above lemma, we can tell that $\tilde{m}_{\bX Y, 2, J (n)}$ with $J (n) \asymp \log \left( \frac{1}{\eps} \right)$ and coefficients $\{c_{l}, l = 0, \cdots, J (n)\}$ as specified in $q_{J (n)}$ satisfies:
\begin{align*}
|m_{\bX Y, 2} - \tilde{m}_{\bX Y, 2, J (n)}| \lesssim \eps.
\end{align*}

Since for each $l = 0, \cdots, J (n)$, $\hat{m}_{\bX Y, 2}^{(l)}$ is an unbiased, $l + 2$-th order $U$-statistic estimator of $m_{\bX Y, 2}^{(l)}$, we have
\begin{align*}
\var \left( \hat{m}_{\bX Y, 2, J (n)} \right) \lesssim \frac{1}{n} \left( \frac{J (n)^{6} \cdot p}{n} \right)^{J (n) + 1} \lesssim \frac{1}{n} \left( J (n)^{6} \cdot \delta \right)^{J (n) + 1},
\end{align*}
following Proposition 7 of \citet{kong2018estimating}. One could also obtain a similar variance bound by Lemma 12 of \citet{liu2017semiparametric}. By choosing, say, $J (n) \asymp (\log n)^{c}$ for some $c$ strictly less than 1, both the bias and variance still diminish to zero.
\end{proof}

\subsubsection{The case of knowing neither \texorpdfstring{$\bm{\mu}$}{mean} nor \texorpdfstring{$\bm{\Sigma}$}{covariance}}
\label{app:unknown complete}

Lastly, in this section, we explain how to handle the case where both $\bmu$ and $\bSigma$ are unknown and need to be estimated from data. Though almost a direct corollary of combining ideas in Sections~\ref{sec:mu} and~\ref{sec:unknown}, we decide to spell it out for the sake of completeness. In the \href{https://github.com/cxy0714/Method-of-Moments-Inference-for-GLMs}{accompanying GitHub repository}, we have also implemented the method for this general scenario.

We first focus on the case of $p < n / 2$, under Assumption~\ref{as:beta}. We first estimate $\bSigma$ by
\begin{align*}
\tilde{\bSigma} \coloneqq \frac{1}{\frac{n}{2} - p - 1} \sum_{j \in I_{2}} (\bX_{j} - \bar{\bX}_{I_{2}}) (X_{j} - \bar{\bX}_{I_{2}})^{\top}, \text{ where } \bar{\bX}_{I_{2}} \coloneqq \frac{1}{n / 2} \sum_{j \in I_{2}} \bX_{j}.
\end{align*}
Then we construct the same estimator $\hat{m}_{\bX Y, 2}$ as in \eqref{alternative}.

For the case of $p \geq n$, we use the approach of approximating reciprocal by Chebyshev polynomials. Since $\gamma_{\beta}^{2 (l)} = \bbE [Y \bX^{\top}] \bSigma^{l} \bbE [\bX Y]$ involves $\bSigma^{l}$, we construct the following unbiased higher-order $U$-statistic estimator:
\begin{align*}
\hat{\gamma}_{\beta}^{2 (l)} = \frac{(n - 2 l - 2)!}{n!} \sum_{1 \leq i_{1} \neq i_{2} \neq \cdots \neq i_{2 l + 1} \neq i_{2 l + 2} \leq n} Y_{i_{1}} \bX_{i_{1}}^{\top} \prod_{s = 1}^{l} \left( \bX_{i_{2 s + 1}} \bX_{i_{2 s + 1}}^{\top} - \bX_{i_{2 s + 1}} \bX_{i_{2 s + 2}}^{\top} \right) \bX_{i_{2}} Y_{i_{2}}.
\end{align*}
It is worth noting that when $\bmu$ is unknown, the involved $U$-statistics have asymmetric kernels, rendering the computation much more complicated as mentioned in Remark \ref{rem:hoif}.

\section{Additional Results on Numerical Experiments}
\label{app:sim}

In this section, we provide complementary simulation results to those in Section~\ref{sec:sims}. Appendix~\ref{app:sim fig} displays additional figures related to the settings in Section~\ref{sec:sims}. Appendix~\ref{app:sim univ} is devoted to comparing the performance of our proposed estimator and the debiased estimator of \citet{bellec2025observable} when universality conditions such as Assumption~\ref{as:beta} or \ref{as:beta mean zero} fail to hold. Appendix~\ref{app:sim mar} complements Section~\ref{sec:sims mar} by examining non-Gaussian designs

\subsection{Supplementary figures for Setting 2 of Section~\ref{sec:sims glms}}
\label{app:sim fig}

Here we collect Figures~\ref{fig:GLM, sparse-converge, gaussian，1.2} --~\ref{fig:GLM, sparse-converge, rademacher，1.2, hist} for numerical experiments of Setting 2 of Section~\ref{sec:sims glms}.

\begin{figure}[H]
\centering
\includegraphics[width = 0.65\textwidth]{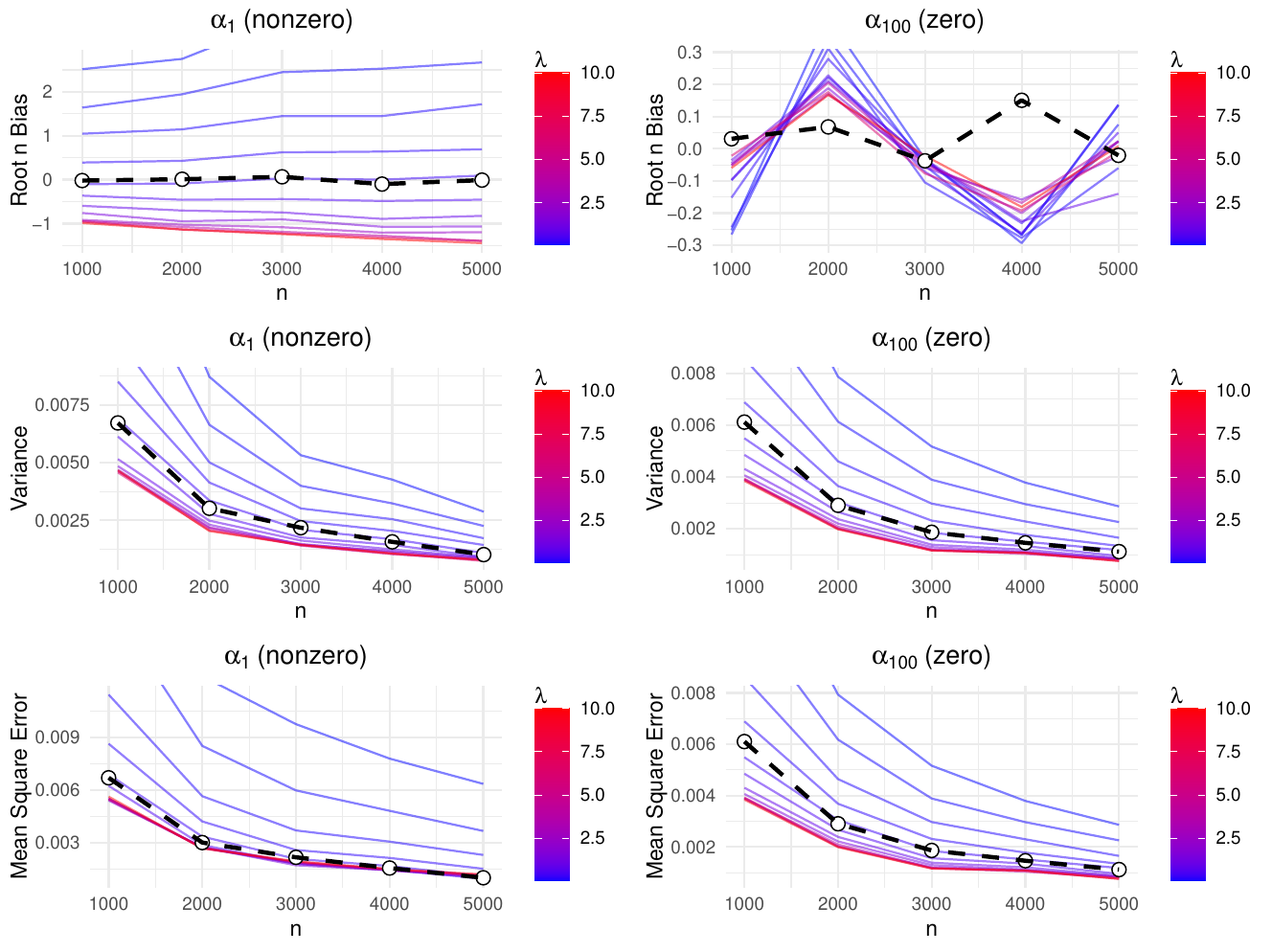}
\caption{Setting 2 of Section~\ref{sec:sims glms} (Gaussian design and sparse regression coefficients).}
\label{fig:GLM, sparse-converge, gaussian，1.2}
\end{figure}

\begin{figure}[H]
\centering
\includegraphics[width = 0.65\textwidth]{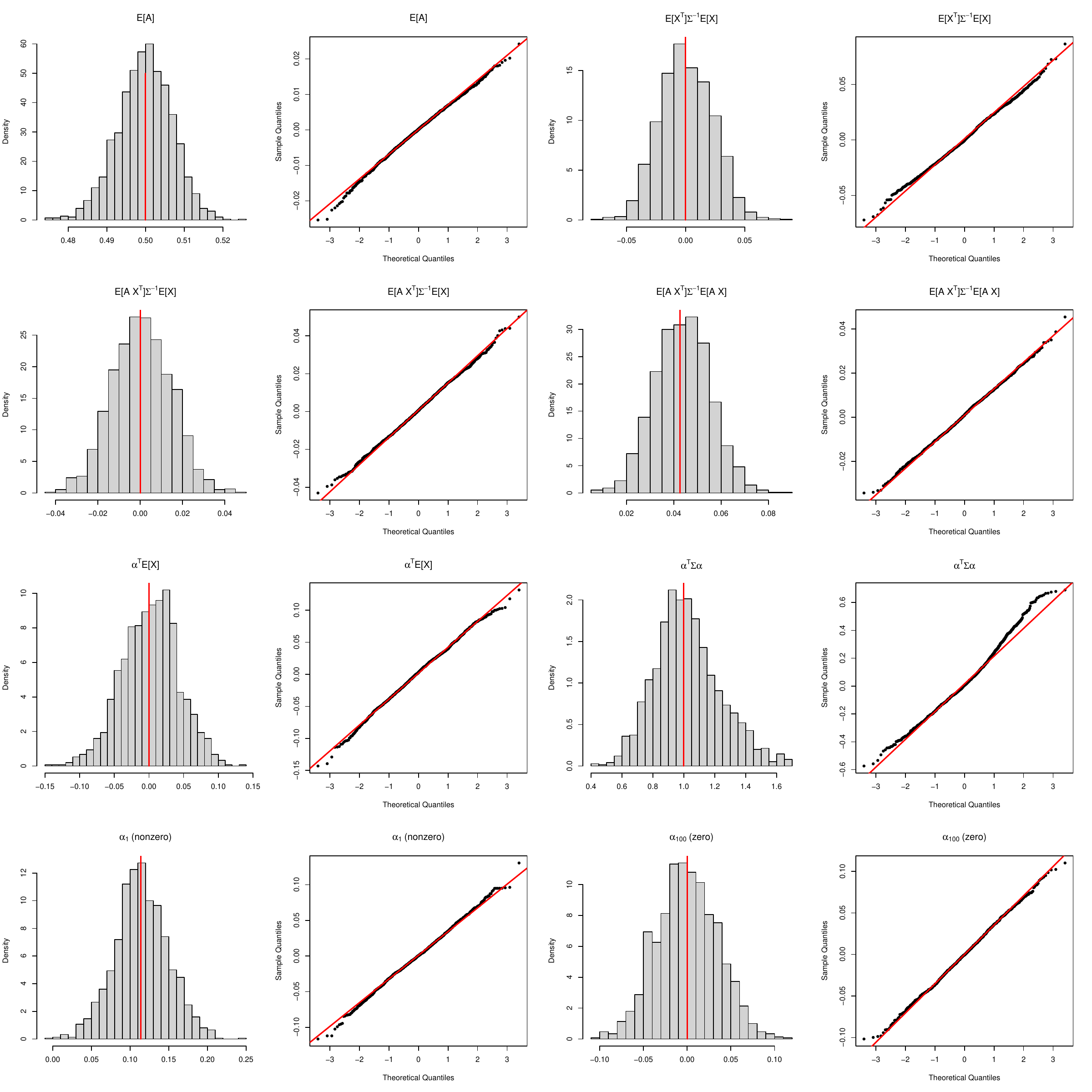}
\caption{Setting 2 of Section~\ref{sec:sims glms} (Gaussian design and sparse regression coefficients): Sampling distributions of the moment estimators and the parameter estimators, over 500 Monte Carlos are displayed for the case of $n = 5000$.}
\label{fig:GLM, sparse-converge, gaussian，1.2, hist}
\end{figure}

\begin{figure}[H]
\centering
\includegraphics[width = 0.65\textwidth]{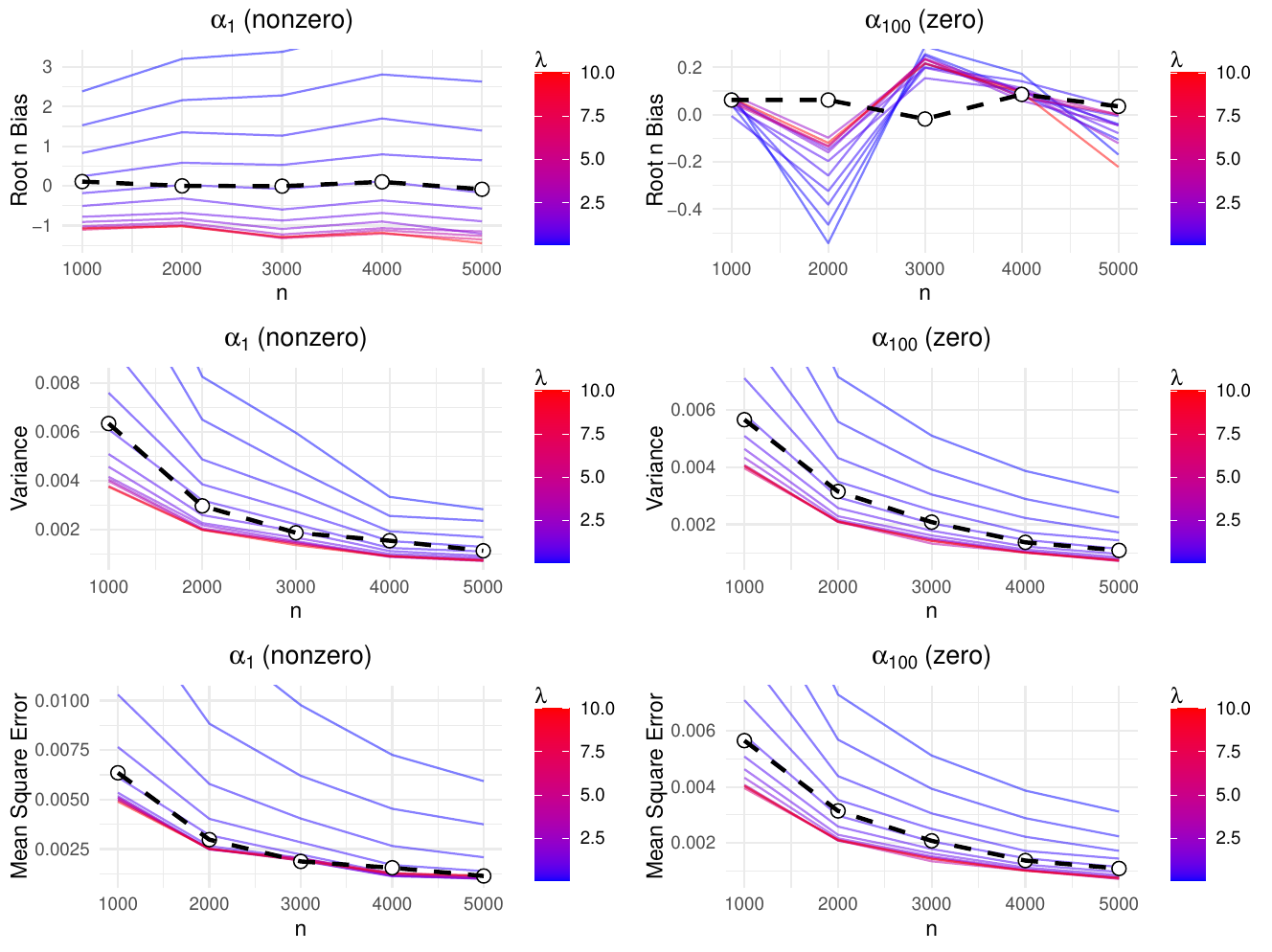}
\caption{Setting 4 of Section~\ref{sec:sims glms} (Rademacher design and sparse regression coefficients).}
\label{fig:GLM, sparse-converge, rademacher，1.2}
\end{figure}

\begin{figure}[H]
\centering
\includegraphics[width = 0.65\textwidth]{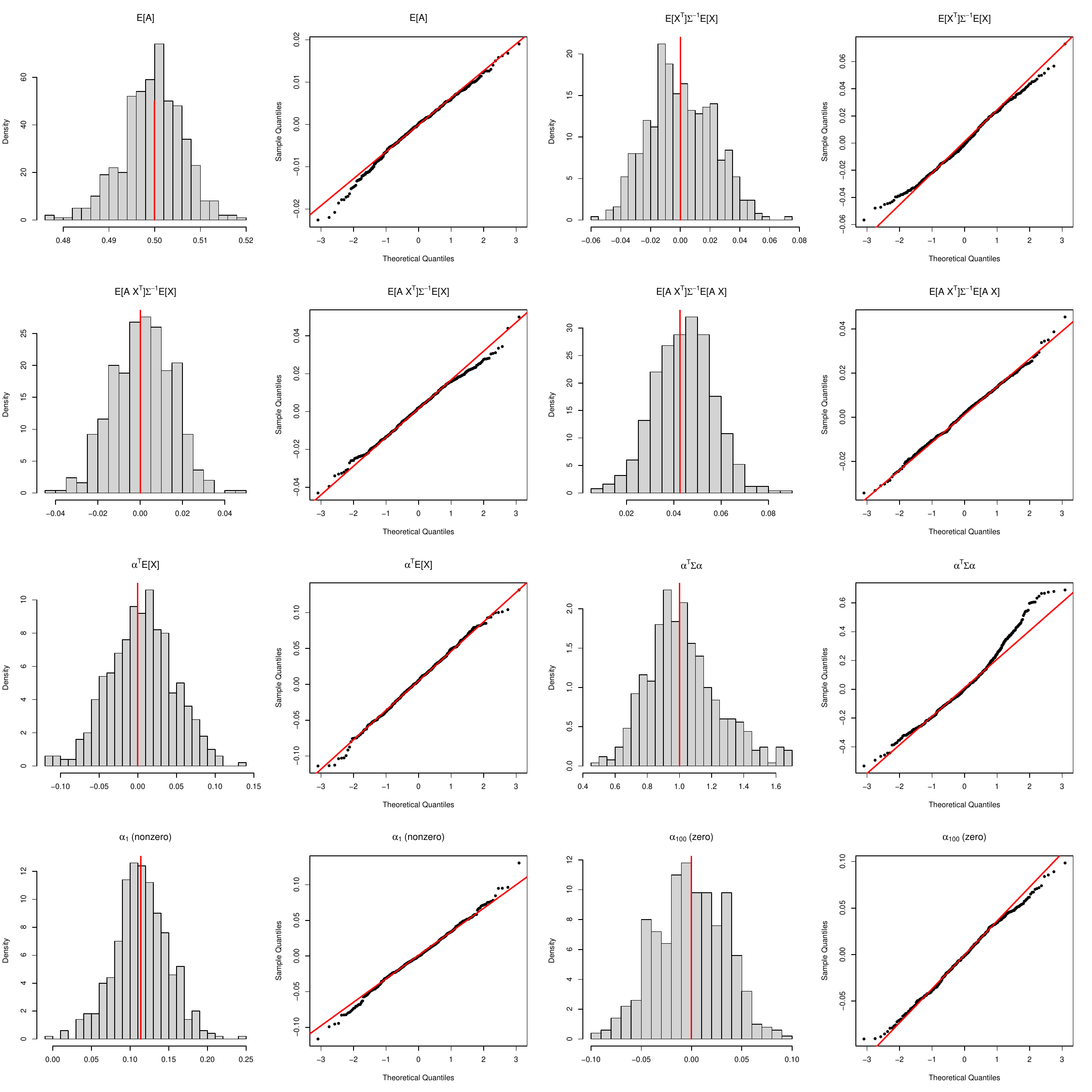}
\caption{Setting 4 of Section~\ref{sec:sims glms} (Rademacher design and sparse regression coefficients): Sampling distributions of the moment estimators and the parameter estimators, over 500 Monte Carlos are displayed for the case of $n = 5000$.}
\label{fig:GLM, sparse-converge, rademacher，1.2, hist}
\end{figure}

\subsection{A simulation setting where universality does not hold}
\label{app:sim univ}

As we mentioned in Remark~\ref{rem:non-gaussian}, the delocalization of regression coefficients or conditions alike shall be necessary to establish universality. Therefore, in this section, we conduct numerical experiments where Assumption~\ref{as:beta}\rm{(2)} or Assumption~\ref{as:beta mean zero}\rm{(2)} is violated and demonstrate that the identification strategy based on Gaussian design indeed fails to deliver the correct estimates of the quadratic form or a single coordinate. To this end, we consider the DGP in which $\bX_{i, j} \overset{\rm i.i.d.}{\sim} \mathrm{Rad} (1 / 2)$, and $\balpha = (1, 0, \cdots, 0)^{\top}$.

The results are summarized in Figures~\ref{fig:GLM, sparse-converge, rademacher, univ} and~\ref{fig:GLM, hist, sparse-converge, rademacher, univ} below. It is clear that our MoM estimators of the quadratic form $\gamma_{\alpha}^{2} \equiv \balpha^{\top} \bSigma \balpha$ and the single coordinate $\alpha_{1}$ are quite different from the true target parameters. For example, speculating Figure~\ref{fig:GLM, hist, sparse-converge, rademacher, univ}, the histogram of $\hat{\alpha}_{1}$ over 500 Monte Carlos does not even cover the true value of $\alpha_{1}$; similarly, the true value of $\gamma_{\alpha}^{2}$ is on the edge of the histogram of $\hat{\gamma}_{\alpha}^{2}$ over 500 Monte Carlos. Interestingly, when $\lambda$ is relatively large, the debiased estimator of \citet{bellec2025observable} has $\sqrt{n} \times \mathrm{bias}$ close to 0, in which case the estimated regression coefficients are mostly very small. It will be interesting to further investigate the effect of $\lambda$ on the bias, variance, and mean squared error of the debiased estimators.

\begin{figure}[htbp]
\centering
\includegraphics[width = 0.65\textwidth]{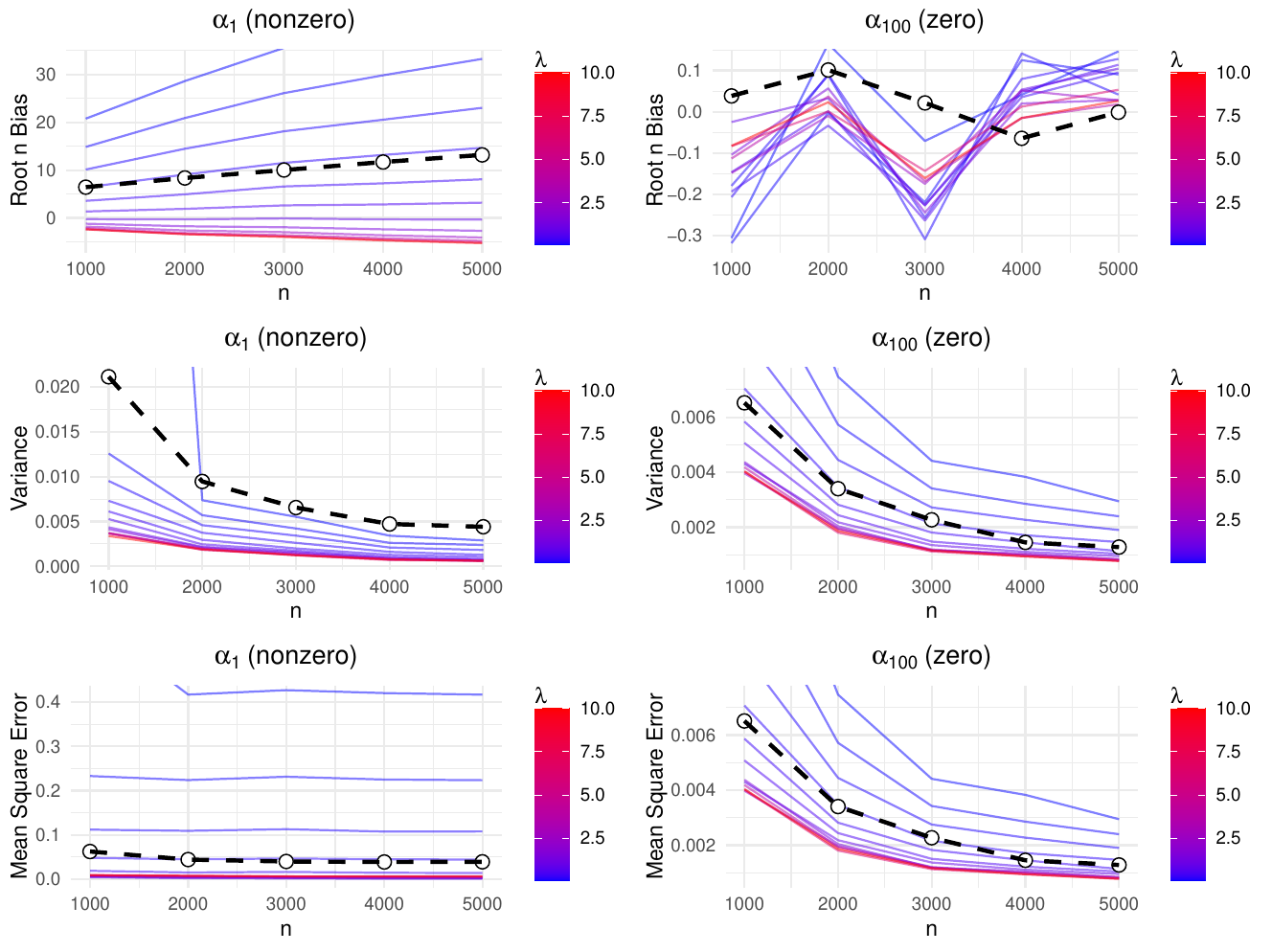}
\caption{Setting in Appendix~\ref{app:sim univ}.}
\label{fig:GLM, sparse-converge, rademacher, univ}
\end{figure}

\begin{figure}[htbp]
\centering
\includegraphics[width = 0.65\textwidth]{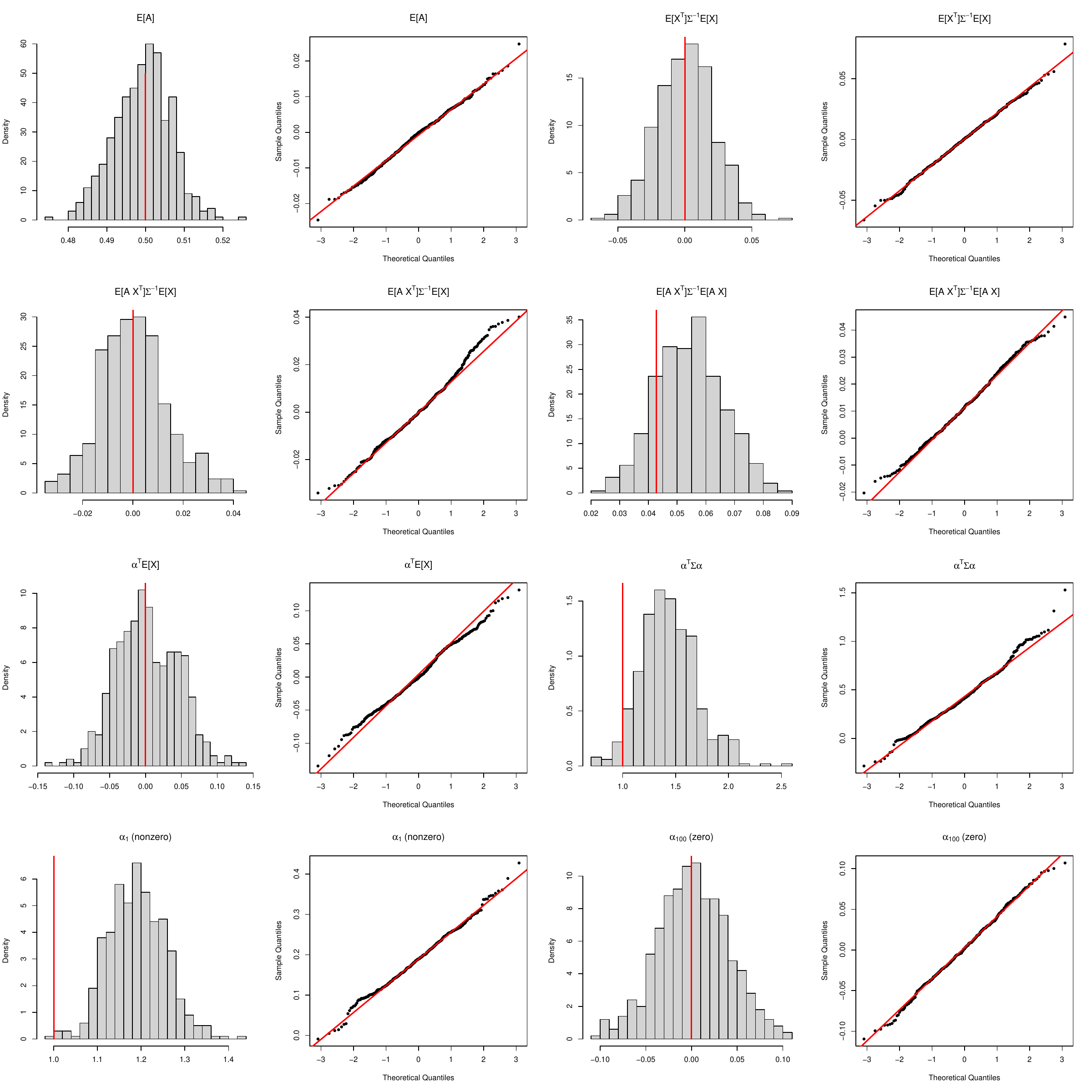}
\caption{Setting for Appendix~\ref{app:sim univ}: Sampling distributions of the moment estimators and the parameter estimators, over 500 Monte Carlos are displayed for the case of $n = 5000$.}
\label{fig:GLM, hist, sparse-converge, rademacher, univ}
\end{figure}

\subsection{Estimating the mean response under MAR: Supplementary simulation results}
\label{app:sim mar}

In this section, we revisit the numerical experiments in Section~\ref{sec:sims mar}, first with the Gaussian design replaced by the Rademacher design. Here we only consider the setting with dense regression coefficients. The results are very similar to the settings with the Gaussian design. In the end of this section, we display results when the regression coefficients are sparse under the Gaussian design.

\begin{figure}[htbp]
\centering
\includegraphics[width = 0.65\textwidth]{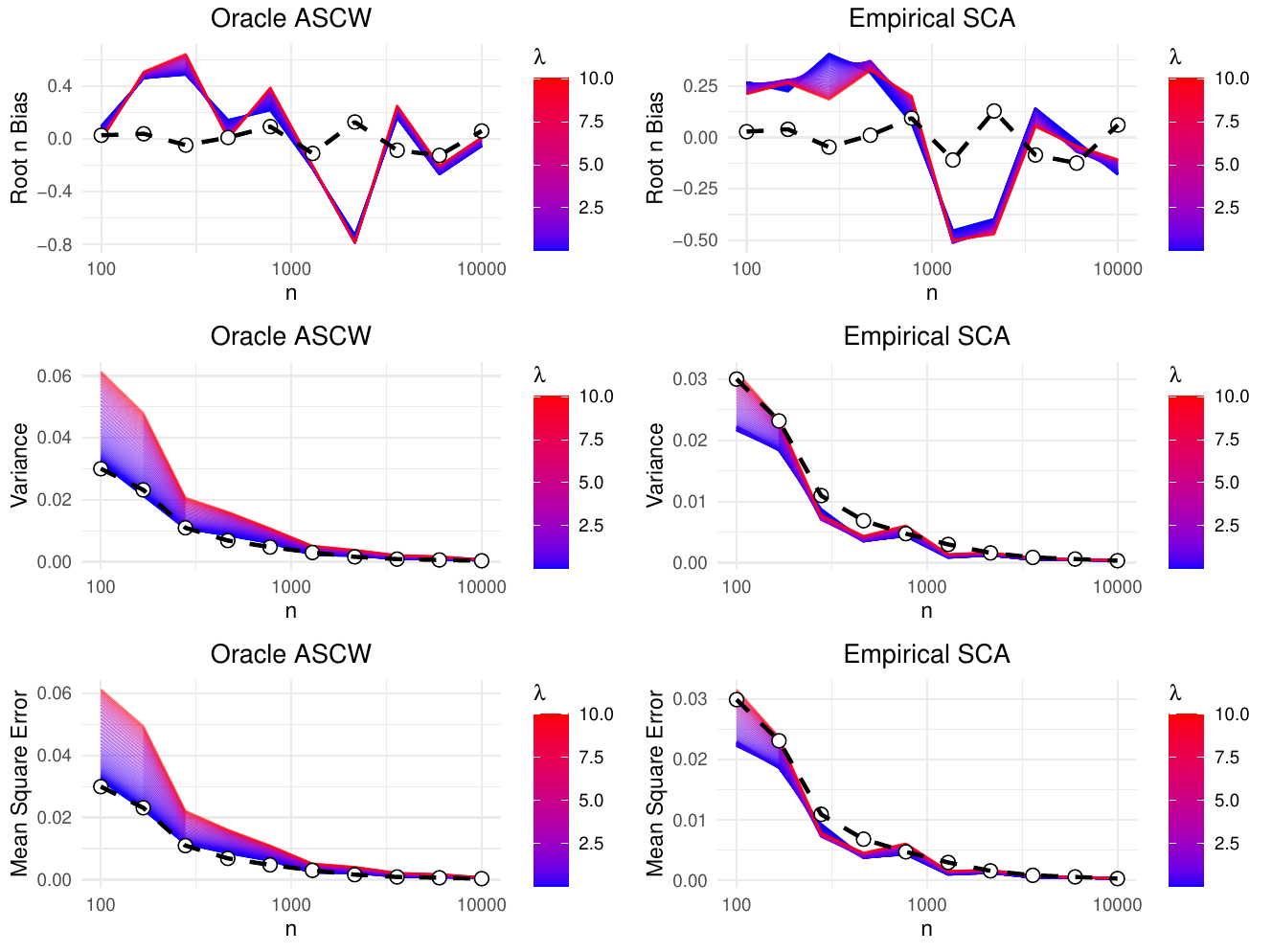}
\caption{Simulation results for Setting 2 (dense regression coefficients) in Section~\ref{sec:sims mar}, but with the Gaussian design replaced by the Rademacher design. The two methods proposed in \citet{celentano2023challenges} are plotted separately in two columns of the figure, with color gradients from blue to red representing the increasing value of the tuning parameter $\lambda$. The MoM-based estimators are plotted with white circles and dashed black lines.}
\label{fig:MAR, dense, Rademacher}
\end{figure}

\begin{figure}[htbp]
\centering
\includegraphics[width = 0.65\textwidth]{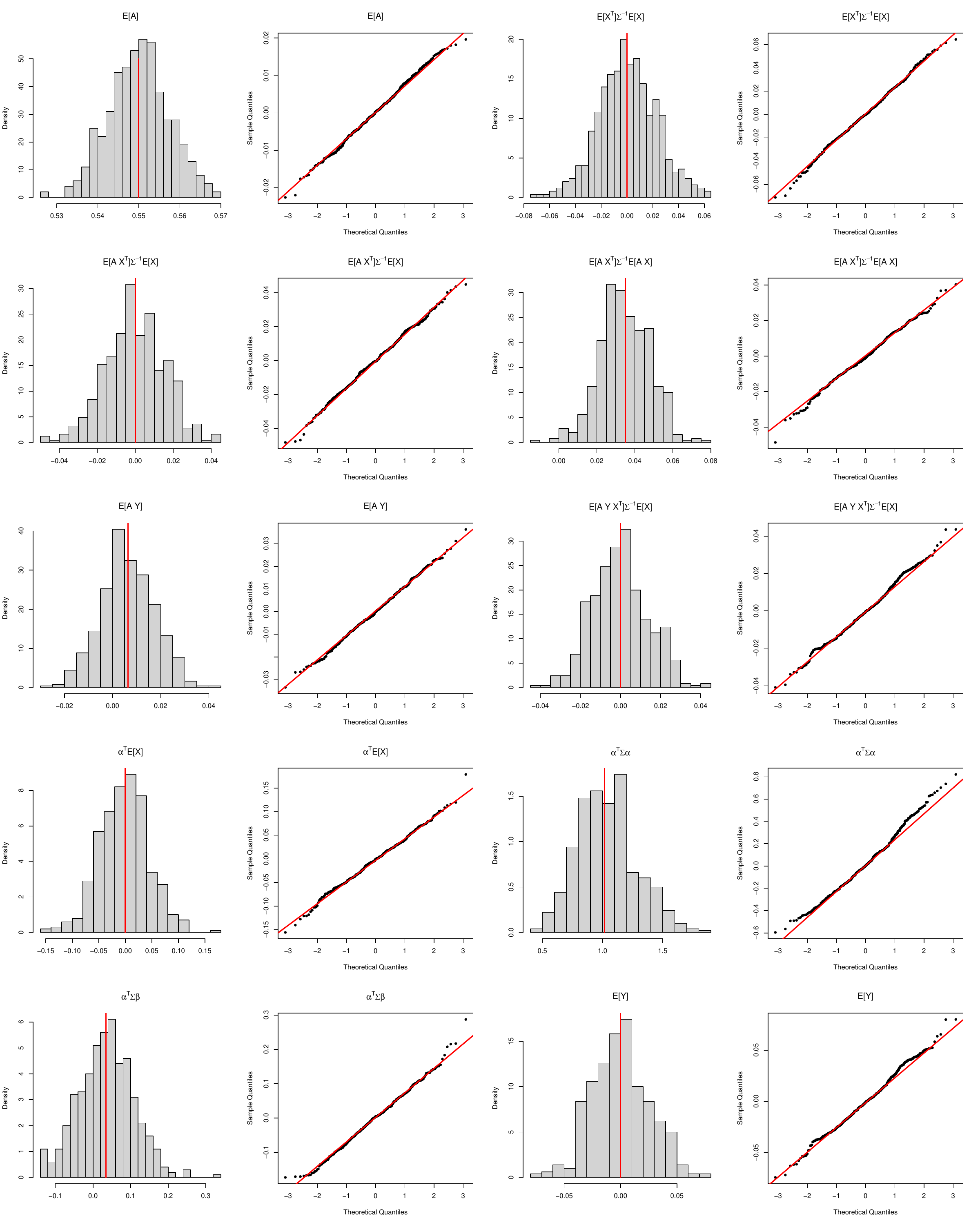}
\caption{Simulation results for Setting 2 (dense regression coefficients) in Section~\ref{sec:sims mar}, but with the Gaussian design replaced by the Rademacher design. The histograms and normal quantile-quantile plots of of the moment estimators and the estimators of the target parameters, including $\lambda_{\alpha}, \gamma_{\alpha}^{2}, 
\gamma_{\alpha, \beta}$ and $\psi$, over 500 Monte Carlos are displayed for the case of $n = 5000$.}
\label{fig:MAR, parameters, dense, rademacher}
\end{figure}

Finally, the simulation results when the regression coefficients are sparse are reported in Figures~\ref{fig:MAR, sparse, gaussian} --~\ref{fig:MAR, parameters, sparse, gaussian}.

\begin{figure}[htbp]
\centering
\includegraphics[width = 0.65\textwidth]{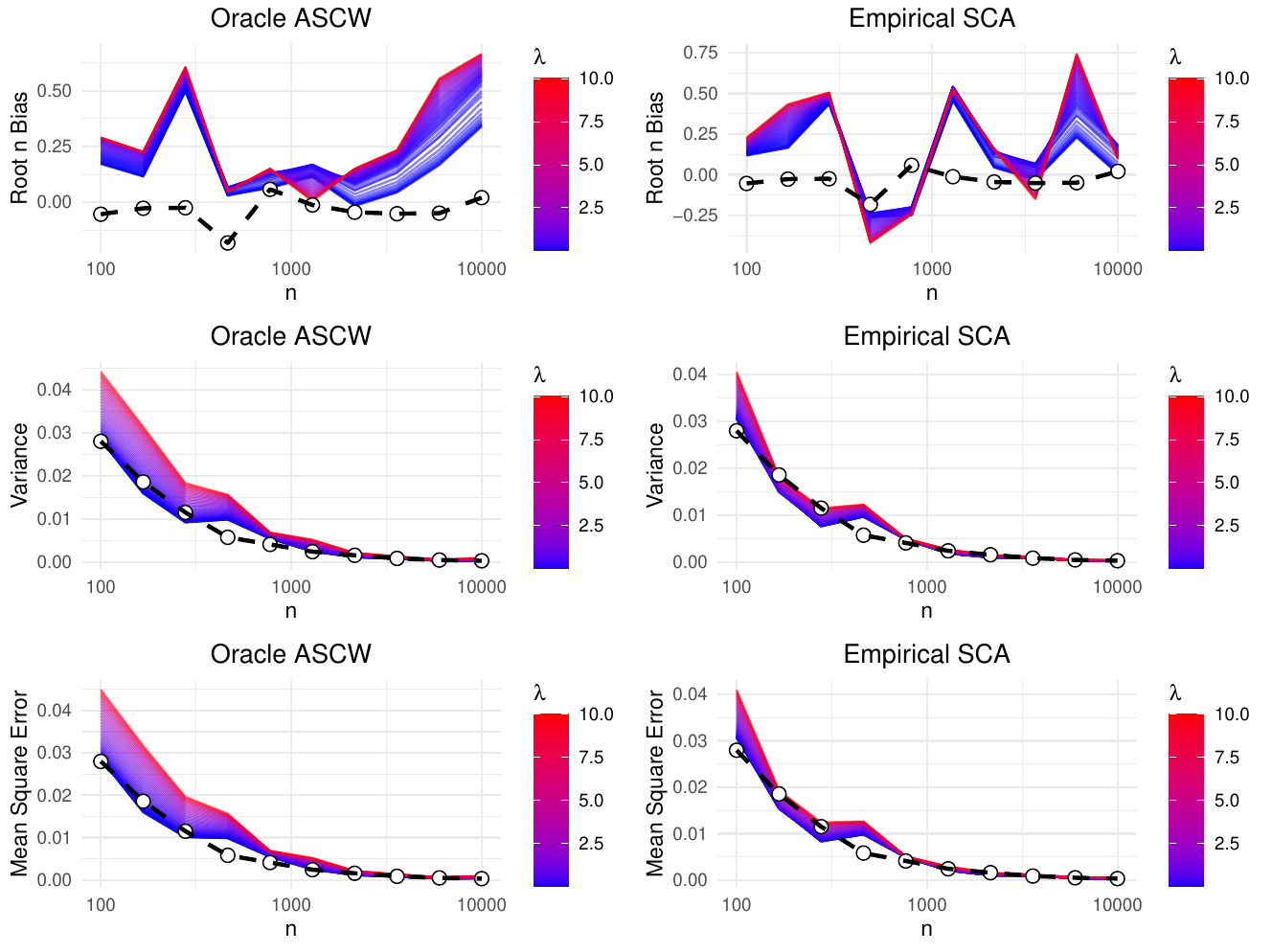}
\caption{Simulation results for Setting 2 (sparse regression coefficients) in Section~\ref{sec:sims mar}. The two methods proposed in \citet{celentano2023challenges} are plotted separately in two columns of the figure, with color gradients from blue to red representing the increasing value of the tuning parameter $\lambda$. The MoM-based estimators are plotted with white circles and dashed black lines.}
\label{fig:MAR, sparse, gaussian}
\end{figure}

\begin{figure}[htbp]
\centering
\includegraphics[width = 0.65\textwidth]{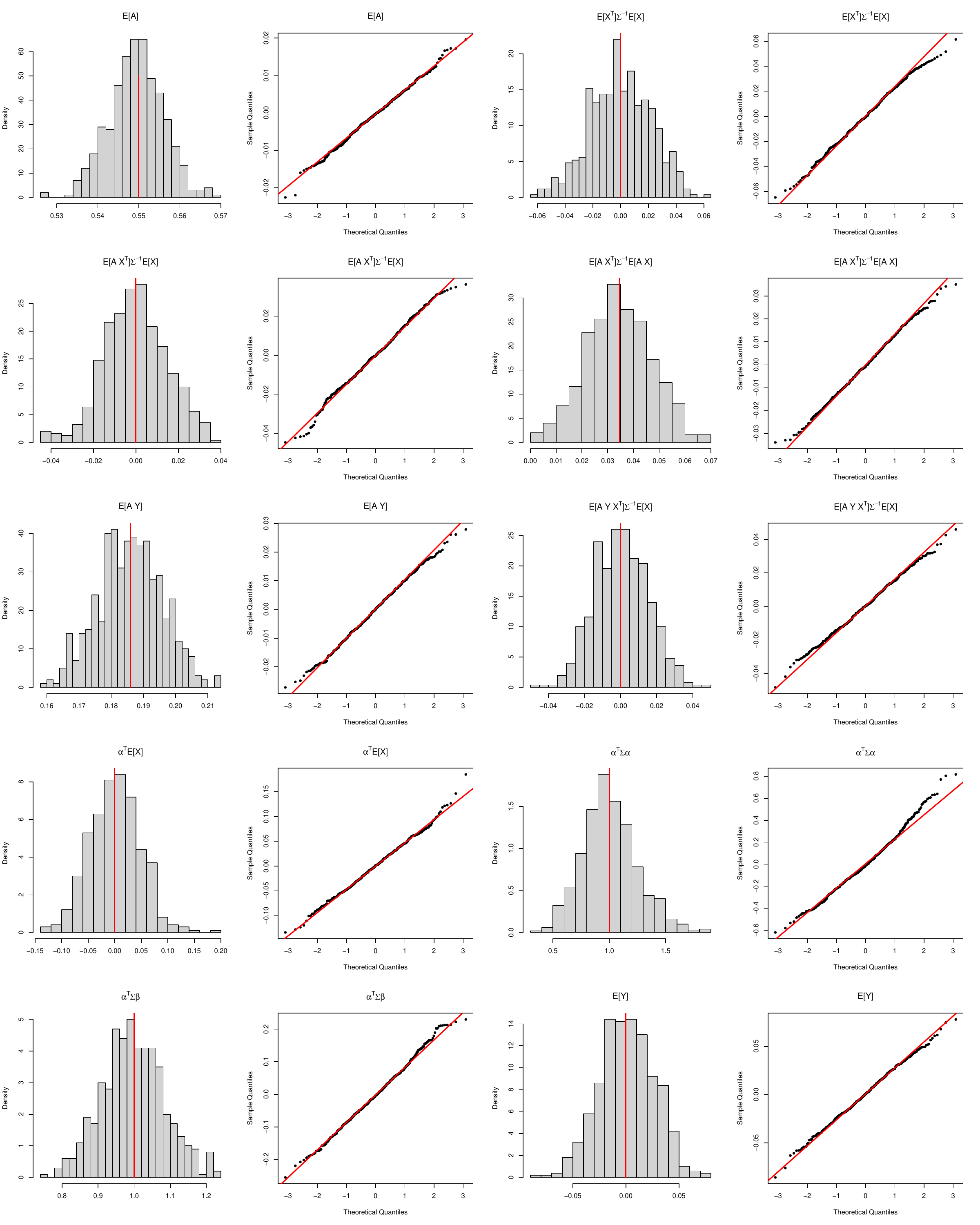}
\caption{Simulation results for Setting 2 (sparse regression coefficients) in Section~\ref{sec:sims mar}: Sampling distributions of the moment estimators and the parameter estimators, over 500 Monte Carlos are displayed for the case of $n = 5000$.}
\label{fig:MAR, parameters, sparse, gaussian}
\end{figure}

\subsection{Simulations knowing neither \texorpdfstring{$\bm{\mu}$}{mean} nor \texorpdfstring{$\bm{\Sigma}$}{covariance}}
\label{app:sims unknown}

In this section, we further examine the finite sample performance of the proposed moment-based estimators for linear and quadratic forms of GLM parameters under varying knowledge of $\bm{\mu}$ and $\bm{\Sigma}$, with $p < n / 2$. We focus exclusively on the case of dense regression coefficients and Gaussian designs. We split the sample into two non-overlapping subsamples $I_{1}$ and $I_{2}$ with equal size. We estimate $\bSigma$ from $I_{1}$ and then estimate the moments with $U$-statistics from $I_{2}$; afterwards we reverse the roles of $I_{1}$ and $I_{2}$ and implement cross-fitting. Finally we use the cross-fitted moment estimator to estimate the linear form $\balpha^{\top} \bmu$ and the quadratic form $\balpha^{\top} \bSigma \balpha$ of the regression coefficients.

In Figure~\ref{fig:GLM_known_nothing, linear,quad, comparision,dense, gaussian} and Figure~\ref{fig:GLM_known_nothing, alpha, comparision,dense, gaussian}, we compare four scenarios for estimators of linear form $\balpha^{\top} \bm{\mu}$ , quadratic form $\Vert \balpha \Vert_{\bSigma}^{2}$, $\alpha_1$ and $\alpha_{100}$: (i) neither $\bm{\mu}$ nor $\bm{\Sigma}$ is known, (ii) only $\bm{\mu}$ is known, (iii) only $\bm{\Sigma}$ is known, and (iv) both $\bm{\mu}$ and $\bm{\Sigma}$ are known. 

For the linear form $\balpha^{\top} \bm{\mu}$ and the quadratic form $\Vert \balpha \Vert_{\bSigma}^{2}$ in Figure~\ref{fig:GLM_known_nothing, linear,quad, comparision,dense, gaussian}, when $\bm{\Sigma}$ is unknown, the root-$n$ bias diminishes steadily with growing $n$, supporting $\sqrt{n}$-consistency of the proposed estimators. The variance and mean squared error also decrease as $n$ rises; however, compared to the scenario where $\bm{\Sigma}$ is known, the variance exhibits a more substantial inflation, highlighting the effect of estimating $\bSigma$.

In contrast, for the estimators of $\alpha_1$ and $\alpha_{100}$ in Figure~\ref{fig:GLM_known_nothing, alpha, comparision,dense, gaussian}, the root-$n$ bias, variance, and mean squared error become increasingly similar across all four scenarios as $n$ grows. Notably, the variances are only slightly elevated when $\bm{\Sigma}$ is unknown.

In Figure~\ref{fig:hist_known_nothing, GLM, dense, gaussian, p/n = 0.40} ($n = 5000, p = 2000$), Figure~\ref{fig:hist_known_nothing, GLM, dense, gaussian, p/n = 0.10} ($n = 10000, p = 1000$) and Figure~\ref{fig:hist_known_nothing, GLM, dense, gaussian, p/n = 0.20} ($n = 10000, p = 2000$), the histograms displayed entail that when both $\bSigma$ and $\bm{\mu}$ are unknown, the sampling distributions of our proposed estimators remain centered around the true values. In particular, it is evident that the sampling distributions our proposed estimators for both the linear form and quadratic form are getting closer to normal distribution as $n / p$ increases,

\begin{figure}[htbp]
\centering
\includegraphics[width = 0.65\textwidth]{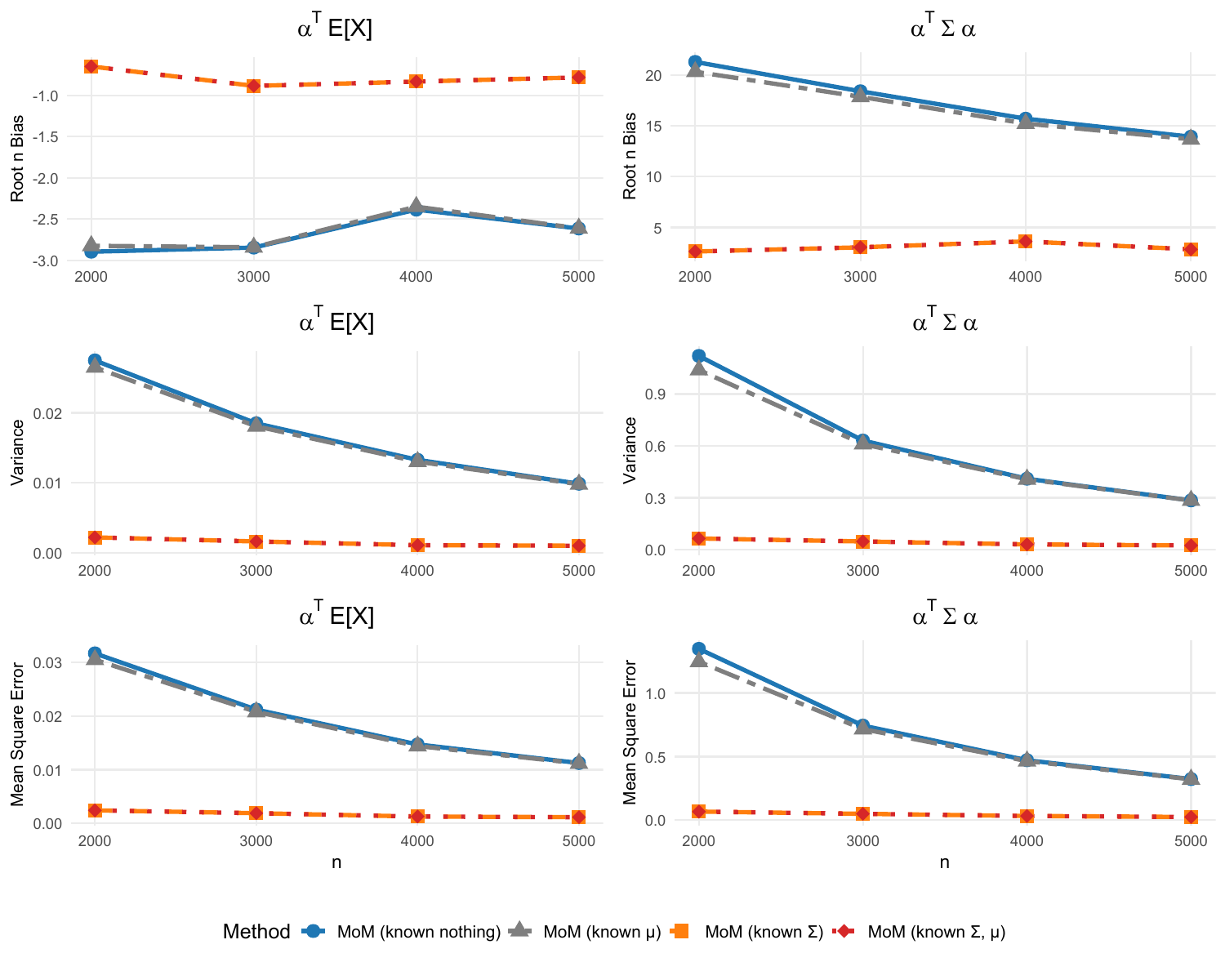}
\caption{Root-$n$ Bias, Variance, and Mean Squared Error of Estimators for linear form $\balpha^{\top} \bm{\mu}$ and quadratic form $\Vert \balpha \Vert_{\bSigma}^{2}$ of GLM parameters under Setting~\ref{sec:sims glms} (Gaussian
design and dense regression coefficients) in Section~\ref{sec:sims glms} with varying knowledge of $\bm{\mu}$ and $\bm{\Sigma}$, Based on 500 Monte Carlo Simulations with $p/n = 0.4$ fixed and $n$ range from $2000$ to $5000$.}
\label{fig:GLM_known_nothing, linear,quad, comparision,dense, gaussian}
\end{figure}

\begin{figure}[htbp]
\centering
\includegraphics[width = 0.65\textwidth]{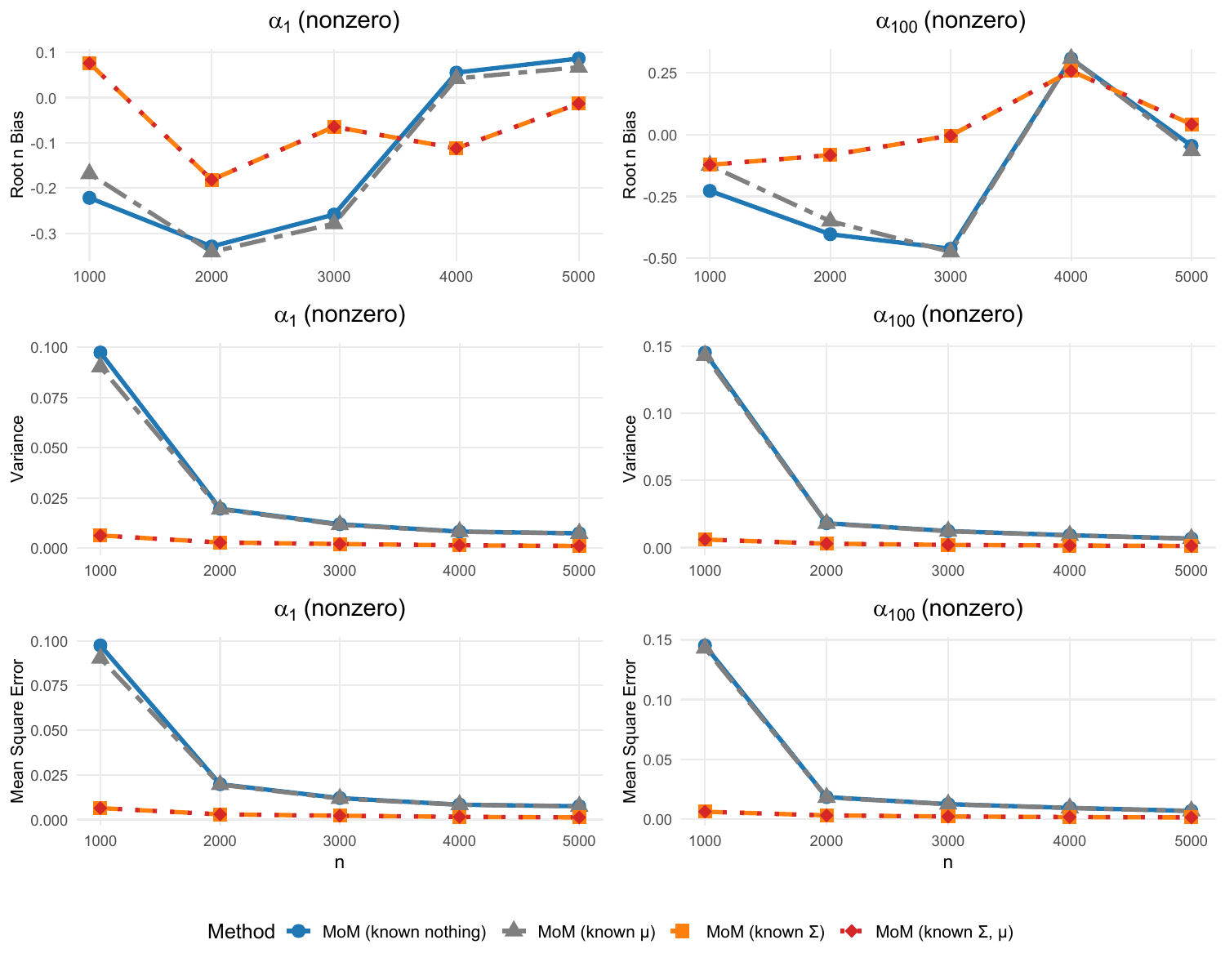}
\caption{Root-$n$ Bias, Variance, and Mean Squared Error of Estimators for $\alpha_1$ and $\alpha_{100}$ under Setting~\ref{sec:sims glms} (Gaussian
design and dense regression coefficients) in Section~\ref{sec:sims glms} with varying knowledge of $\bm{\mu}$ and $\bm{\Sigma}$, Based on 500 Monte Carlo Simulations with $p/n = 0.4$ fixed and $n$ range from $1000$ to $5000$.}
\label{fig:GLM_known_nothing, alpha, comparision,dense, gaussian}
\end{figure}

\begin{figure}[htbp]
\centering
\includegraphics[width = 0.65\textwidth]{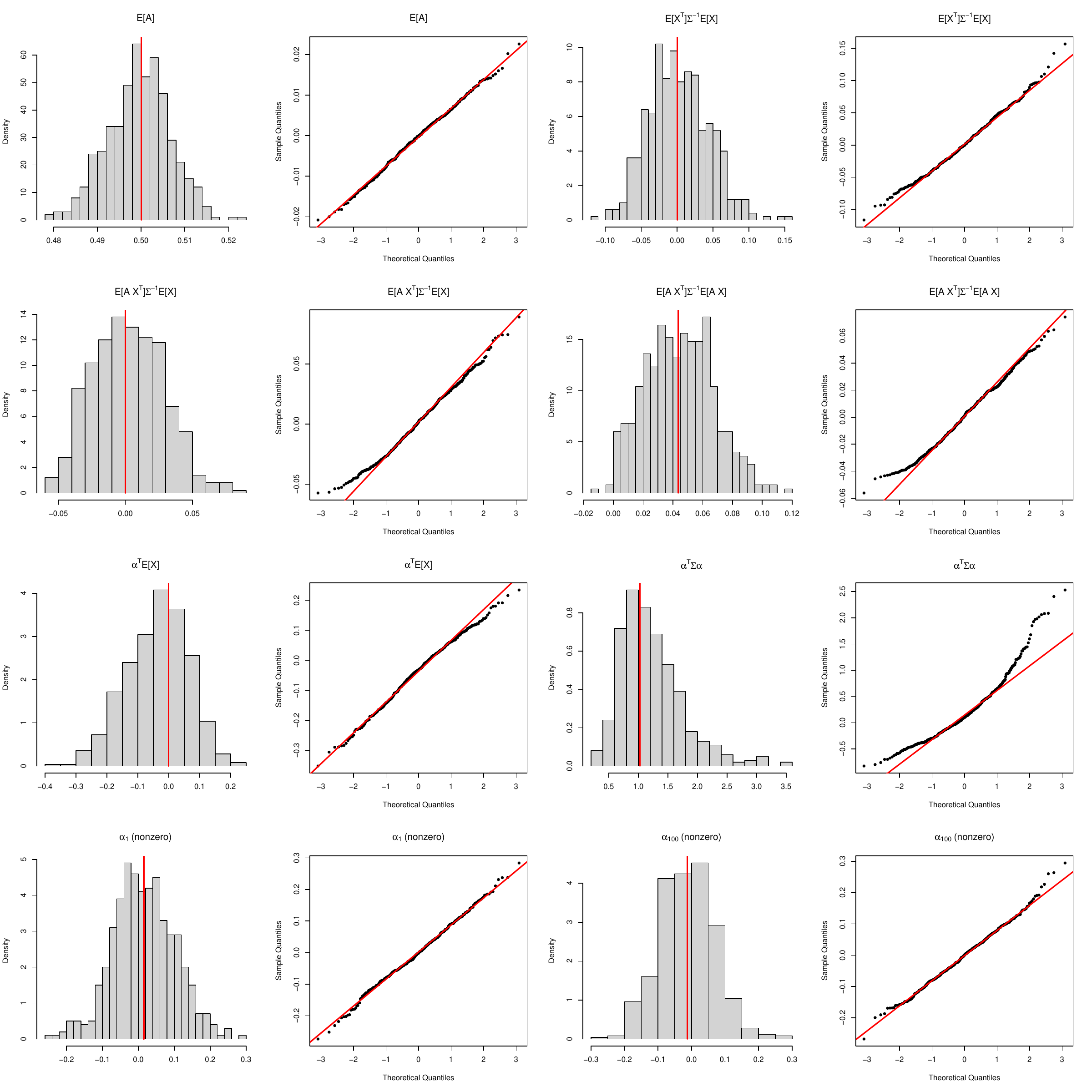}
\caption{Simulation results for Setting~\ref{sec:sims glms} (Gaussian
design and dense regression coefficients) in Section~\ref{sec:sims glms}: Sampling distributions of the moment estimators and the parameter estimators when both $\bSigma$ and $\bm{\mu}$ are unknown, over 500 Monte Carlos are displayed for the case of $n = 5000$ and $p/n = 0.4$.}
\label{fig:hist_known_nothing, GLM, dense, gaussian, p/n = 0.40}
\end{figure}

\begin{figure}[htbp]
\centering
\includegraphics[width = 0.65\textwidth]{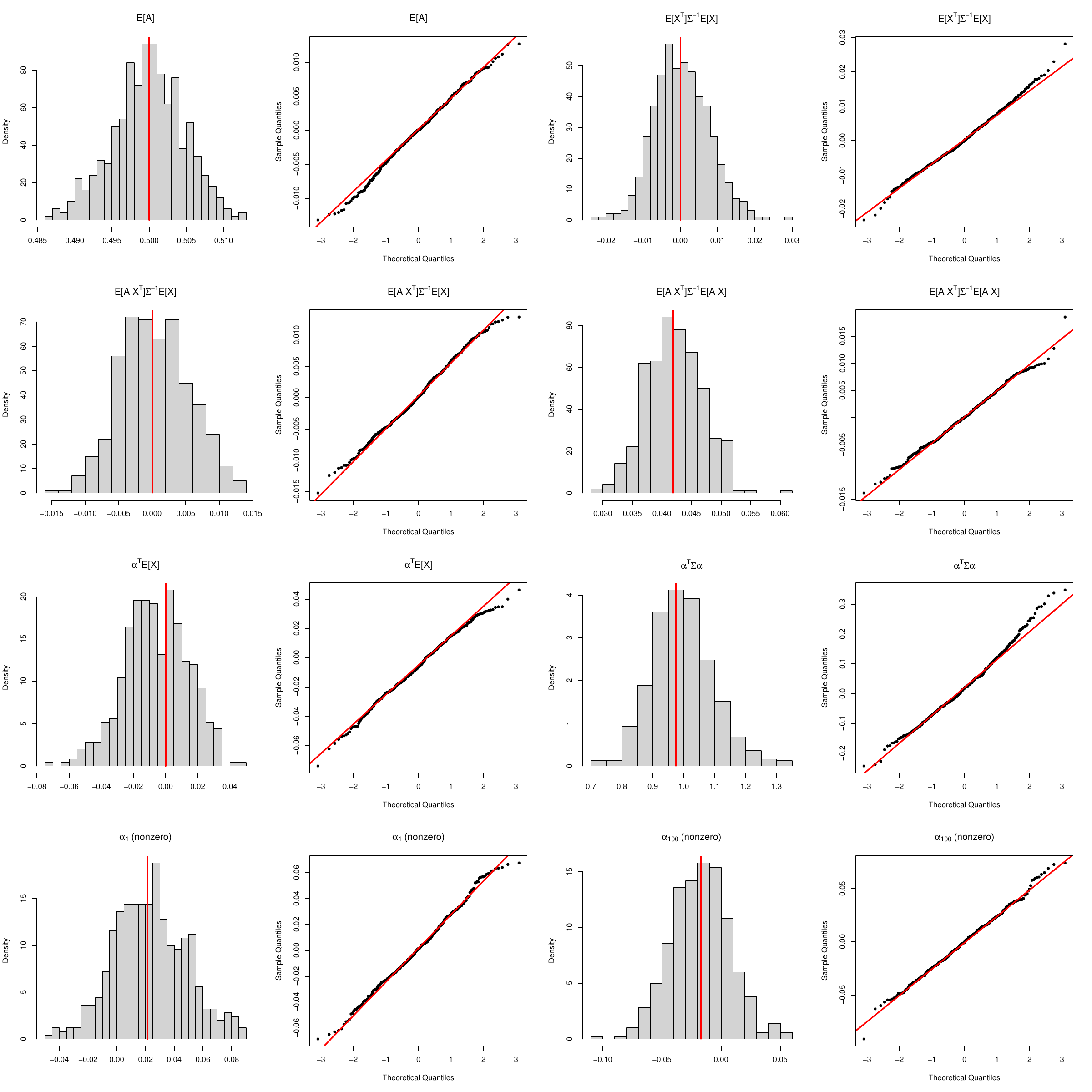}
\caption{Simulation results for Setting~\ref{sec:sims glms} (Gaussian
design and dense regression coefficients) in Section~\ref{sec:sims glms}: Sampling distributions of the moment estimators and the parameter estimators when both $\bSigma$ and $\bm{\mu}$ are unknown, over 500 Monte Carlos are displayed for the case of $n = 10000$ and $p/n = 0.1$.}
\label{fig:hist_known_nothing, GLM, dense, gaussian, p/n = 0.10}
\end{figure}

\begin{figure}[htbp]
\centering
\includegraphics[width = 0.65\textwidth]{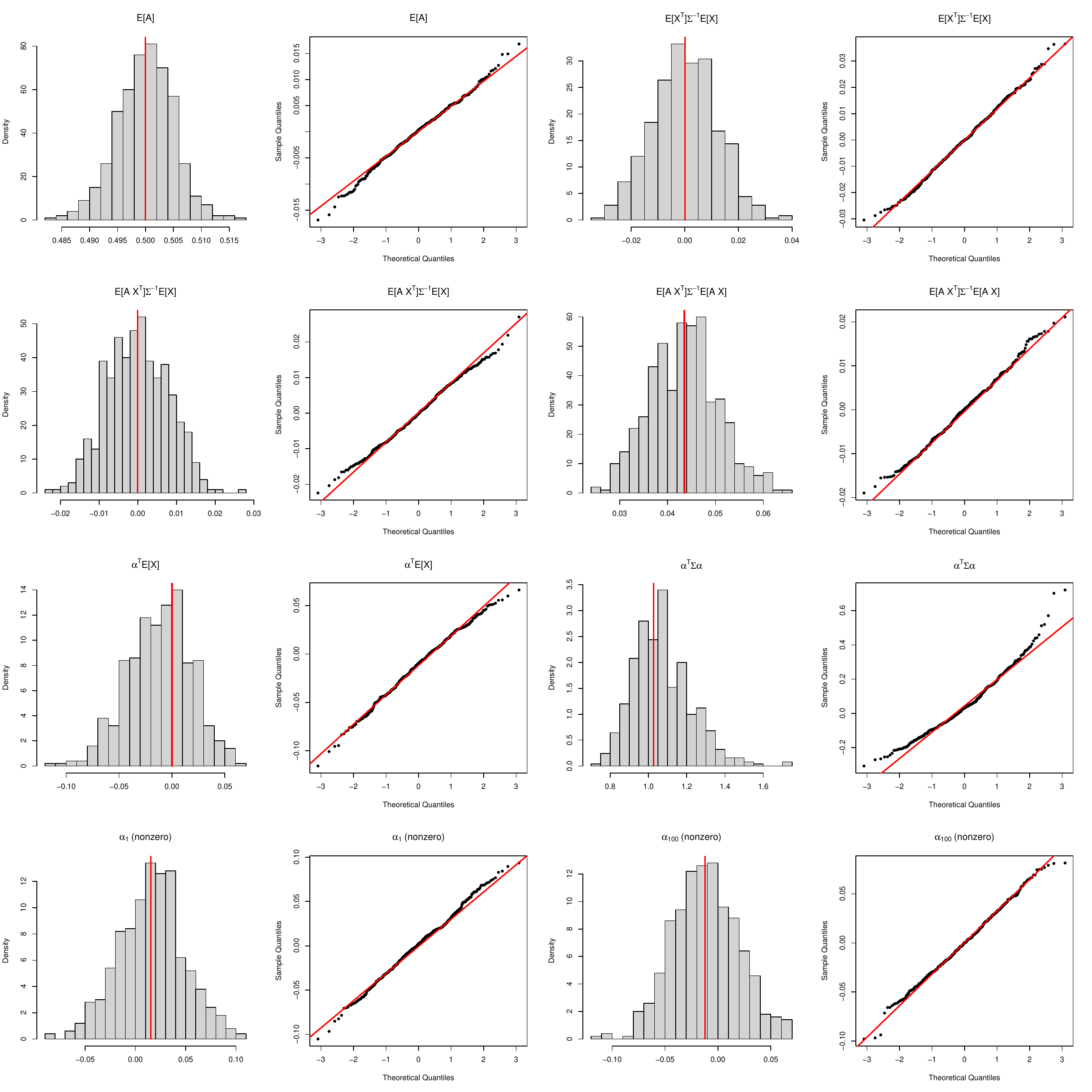}
\caption{Simulation results for Setting~\ref{sec:sims glms} (Gaussian
design and dense regression coefficients) in Section~\ref{sec:sims glms}: Sampling distributions of the moment estimators and the parameter estimators when both $\bSigma$ and $\bm{\mu}$ are unknown, over 500 Monte Carlos are displayed for the case of $n = 10000$ and $p/n = 0.2$.}
\label{fig:hist_known_nothing, GLM, dense, gaussian, p/n = 0.20}
\end{figure}

\subsection{Bootstrap variance estimators}
\label{app:bootstrap}

As mentioned in Section~\ref{sec:bootstrap}, to conduct statistical inference, the variances of the proposed moment-based estimators also need to be estimated. In \citet{liu2024assumption}, we have developed consistent variance estimators based on bootstrap. 

For settings where $p > n$ ($n = 5000, p = 6000$), we consider the four scenarios in Section~\ref{sec:sims glms} under known $\bSigma$ and unknown $\bmu$, with results summarized in Tables~\ref{tab:Bootstrap_variance_estimator_p/n_1.20, gaussian, dense} -- \ref{tab:Bootstrap_variance_estimator_p/n_1.20, rademacher, sparse}. For settings where $p < n$ ($n = 5000, p = 2000$ and $n = 10000$ with $p = 1000$ or $2000$), we focus on Gaussian designs with dense regression coefficients under both unknown $\bSigma$ and unknown $\bmu$, as shown in Tables~\ref{tab:Bootstrap_variance_estimator_p/n_0.40} -- \ref{tab:Bootstrap_variance_estimator_p/n_0.20}.

Across all settings, the Monte Carlo variances -- serving as proxies for the true variances -- closely align with the mean of the bootstrap variance estimates over 500 Monte Carlo draws, as evidenced by their ratios floating around 1. The variances of the bootstrap variance estimators are also generally smaller in magnitude.

\begin{table}[htbp]
\centering
\small
\setlength{\tabcolsep}{4pt}
\renewcommand{\arraystretch}{1.2}
\caption{Bootstrap Variance Estimators vs. Monte Carlo Variances under Setting~\ref{sec:sims glms} (Gaussian design and dense regression coefficients) in Section~\ref{sec:sims glms}, Based on 500 Monte Carlo Simulations with $n = 5000$, $p/n = 1.2$. Here $\bm{\mu}$ is unknown but $\bm{\Sigma}$ is known.}
\begin{tabular}{lccccc}
\toprule
& \textbf{MC Var} & \textbf{Mean Est. Var} & $\frac{\textbf{Mean Est. Var}}{\textbf{MC Var}}$ & \textbf{Std Est. Var} & \textbf{MSE} \\
\midrule
$\mathbb{E} A$ & 4.81e-05 & 5.01e-05 & 1.041 & 2.24e-06 & 3.00e-06 \\
$\mathbb{E} [A \bX^{\top}] \bSigma^{-1} \bmu$ & 4.38e-04 & 4.77e-04 & 1.091 & 7.04e-05 & 8.08e-05 \\
$\mathbb{E} [A \bX^{\top}] \bSigma^{-1} \mathbb{E} [A \bX]$ & 1.85e-04 & 1.87e-04 & 1.009 & 2.83e-05 & 2.84e-05 \\
$\mathbb{E} [A \bX^{\top}] \bSigma^{-1} \mathbb{E} [A \bX]$ & 1.41e-04 & 1.36e-04 & 0.965 & 2.03e-05 & 2.09e-05 \\
$\balpha^{\top} \bmu$ & 2.96e-03 & 3.01e-03 & 1.017 & 1.62e-03 & 1.62e-03 \\
$\balpha^{\top} \bSigma \balpha$ & 6.42e-02 & 6.76e-02 & 1.053 & 3.69e-02 & 3.71e-02 \\
$\alpha_{1}$ & 1.15e-03 & 1.20e-03 & 1.043 & 1.02e-04 & 1.13e-04 \\
$\alpha_{100}$ & 1.13e-03 & 1.20e-03 & 1.060 & 1.08e-04 & 1.28e-04 \\
\bottomrule
\end{tabular}
\label{tab:Bootstrap_variance_estimator_p/n_1.20, gaussian, dense}
\end{table}

\begin{table}[htbp]
\centering
\small
\setlength{\tabcolsep}{4pt}
\renewcommand{\arraystretch}{1.2}
\caption{Bootstrap Variance Estimators vs. Monte Carlo Variances under Setting~\ref{sec:sims glms} (Gaussian design and sparse regression coefficients) in Section~\ref{sec:sims glms}, Based on 500 Monte Carlo Simulations with $n = 5000$, $p/n = 1.2$. Here $\bm{\mu}$ is unknown but $\bm{\Sigma}$ is known.}
\begin{tabular}{lccccc}
\toprule
& \textbf{MC Var} & \textbf{Mean Est. Var} & $\frac{\textbf{Mean Est. Var}}{\textbf{MC Var}}$ & \textbf{Std Est. Var} & \textbf{MSE} \\
\midrule
$\mathbb{E} A$ & 4.59e-05 & 4.99e-05 & 1.087 & 2.19e-06 & 4.57e-06 \\
$\mathbb{E} [A \bX^{\top}] \bSigma^{-1} \bmu$ & 4.48e-04 & 4.78e-04 & 1.065 & 6.95e-05 & 7.54e-05 \\
$\mathbb{E} [A \bX^{\top}] \bSigma^{-1} \mathbb{E} [A \bX]$ & 1.80e-04 & 1.88e-04 & 1.047 & 2.75e-05 & 2.88e-05 \\
$\mathbb{E} [A \bX^{\top}] \bSigma^{-1} \mathbb{E} [A \bX]$ & 1.29e-04 & 1.36e-04 & 1.058 & 1.98e-05 & 2.12e-05 \\
$\balpha^{\top} \bmu$ & 2.67e-03 & 3.05e-03 & 1.140 & 1.62e-03 & 1.67e-03 \\
$\balpha^{\top} \bSigma \balpha$ & 5.93e-02 & 6.67e-02 & 1.126 & 3.46e-02 & 3.54e-02 \\
$\alpha_{1}$ & 1.23e-03 & 1.23e-03 & 0.996 & 1.13e-04 & 1.13e-04 \\
$\alpha_{100}$ & 1.21e-03 & 1.21e-03 & 0.997 & 1.07e-04 & 1.07e-04 \\
\bottomrule
\end{tabular}
\label{tab:Bootstrap_variance_estimator_p/n_1.20, gaussian, sparse}
\end{table}

\begin{table}[htbp]
\centering
\small
\setlength{\tabcolsep}{4pt}
\renewcommand{\arraystretch}{1.2}
\caption{Bootstrap Variance Estimators vs. Monte Carlo Variances under Setting~\ref{sec:sims glms} (Rademacher design and dense regression coefficients) in Section~\ref{sec:sims glms}, Based on 500 Monte Carlo Simulations with $n = 5000$, $p/n = 1.2$. Here $\bm{\mu}$ is unknown but $\bm{\Sigma}$ is known.}
\begin{tabular}{lccccc}
\toprule
& \textbf{MC Var} & \textbf{Mean Est. Var} & $\frac{\textbf{Mean Est. Var}}{\textbf{MC Var}}$ & \textbf{Std Est. Var} & \textbf{MSE} \\
\midrule
$\mathbb{E} A$ & 4.77e-05 & 5.00e-05 & 1.046 & 2.21e-06 & 3.13e-06 \\
$\mathbb{E} [A \bX^{\top}] \bSigma^{-1} \bmu$ & 1.20e-04 & 1.19e-04 & 0.995 & 1.81e-05 & 1.81e-05 \\
$\mathbb{E} [A \bX^{\top}] \bSigma^{-1} \mathbb{E} [A \bX]$ & 4.44e-05 & 4.60e-05 & 1.037 & 6.96e-06 & 7.15e-06 \\
$\mathbb{E} [A \bX^{\top}] \bSigma^{-1} \mathbb{E} [A \bX]$ & 3.19e-05 & 3.24e-05 & 1.017 & 4.73e-06 & 4.76e-06 \\
$\balpha^{\top} \bmu$ & 4.70e-04 & 5.52e-04 & 1.175 & 2.79e-04 & 2.91e-04 \\
$\balpha^{\top} \bSigma \balpha$ & 4.66e-03 & 5.17e-03 & 1.109 & 2.45e-03 & 2.51e-03 \\
$\alpha_{1}$ & 4.67e-04 & 4.47e-04 & 0.958 & 2.34e-05 & 3.05e-05 \\
$\alpha_{100}$ & 4.63e-04 & 4.49e-04 & 0.971 & 2.35e-05 & 2.72e-05 \\
\bottomrule
\end{tabular}
\label{tab:Bootstrap_variance_estimator_p/n_1.20, rademacher, dense}
\end{table}

\begin{table}[htbp]
\centering
\small
\setlength{\tabcolsep}{4pt}
\renewcommand{\arraystretch}{1.2}
\caption{Bootstrap Variance Estimators vs. Monte Carlo Variances under Setting~\ref{sec:sims glms} (Rademacher design and sparse regression coefficients) in Section~\ref{sec:sims glms}, Based on 500 Monte Carlo Simulations with $n = 5000$, $p/n = 1.2$. Here $\bm{\mu}$ is unknown but $\bm{\Sigma}$ is known.}
\begin{tabular}{lccccc}
\toprule
& \textbf{MC Var} & \textbf{Mean Est. Var} & $\frac{\textbf{Mean Est. Var}}{\textbf{MC Var}}$ & \textbf{Std Est. Var} & \textbf{MSE} \\
\midrule
$\mathbb{E} A$ & 4.60e-05 & 4.99e-05 & 1.086 & 2.18e-06 & 4.51e-06 \\
$\mathbb{E} [A \bX^{\top}] \bSigma^{-1} \bmu$ & 1.20e-04 & 1.19e-04 & 0.995 & 1.81e-05 & 1.81e-05 \\
$\mathbb{E} [A \bX^{\top}] \bSigma^{-1} \mathbb{E} [A \bX]$ & 4.66e-05 & 4.56e-05 & 0.979 & 6.69e-06 & 6.76e-06 \\
$\mathbb{E} [A \bX^{\top}] \bSigma^{-1} \mathbb{E} [A \bX]$ & 3.06e-05 & 3.20e-05 & 1.046 & 4.72e-06 & 4.93e-06 \\
$\balpha^{\top} \bmu$ & 4.73e-04 & 5.48e-04 & 1.157 & 2.84e-04 & 2.94e-04 \\
$\balpha^{\top} \bSigma \balpha$ & 4.40e-03 & 5.26e-03 & 1.195 & 2.70e-03 & 2.84e-03 \\
$\alpha_{1}$ & 4.81e-04 & 4.49e-04 & 0.934 & 2.34e-05 & 3.93e-05 \\
$\alpha_{100}$ & 4.32e-04 & 4.48e-04 & 1.037 & 2.34e-05 & 2.83e-05 \\
\bottomrule
\end{tabular}
\label{tab:Bootstrap_variance_estimator_p/n_1.20, rademacher, sparse}
\end{table}

\begin{table}[htbp]
\centering
\small
\setlength{\tabcolsep}{4pt}
\renewcommand{\arraystretch}{1.2}
\caption{Bootstrap Variance Estimators vs. Monte Carlo Variances under Setting~\ref{sec:sims glms} (Gaussian
design and dense regression coefficients) in Section~\ref{sec:sims glms}, Based on 500 Monte Carlo Simulations with $n = 5000$, $p/n = 0.4$. Here neither $\bm{\mu}$ nor $\bm{\Sigma}$ is known.}
\begin{tabular}{lccccc}
\toprule
& \textbf{MC Var} & \textbf{Mean Est. Var} & $\frac{\textbf{Mean Est. Var}}{\textbf{MC Var}}$ & \textbf{Std Est. Var} & \textbf{MSE} \\
\midrule
$\mathbb{E} A$ & 5.08e-05 & 4.99e-05 & 0.982 & 2.41e-06 & 2.57e-06 \\
$\mathbb{E} [A \bX^{\top}] \bSigma^{-1} \bmu$ & 1.72e-03 & 1.62e-03 & 0.940 & 4.26e-04 & 4.38e-04 \\
$\mathbb{E} [A \bX^{\top}] \bSigma^{-1} \mathbb{E} [A \bX]$ & 6.92e-04 & 6.50e-04 & 0.940 & 1.60e-04 & 1.65e-04 \\
$\mathbb{E} [A \bX^{\top}] \bSigma^{-1} \mathbb{E} [A \bX]$ & 5.26e-04 & 4.95e-04 & 0.941 & 1.16e-04 & 1.20e-04 \\
$\balpha^{\top} \bmu$ & 9.88e-03 & 1.13e-02 & 1.142 & 5.67e-03 & 5.84e-03 \\
$\balpha^{\top} \bSigma \balpha$ & 2.85e-01 & 3.41e-01 & 1.197 & 3.29e-01 & 3.34e-01 \\
$\alpha_{1}$ & 7.41e-03 & 6.39e-03 & 0.863 & 1.21e-03 & 1.58e-03 \\
$\alpha_{100}$ & 6.68e-03 & 6.39e-03 & 0.956 & 1.20e-03 & 1.23e-03 \\
\bottomrule
\end{tabular}
\label{tab:Bootstrap_variance_estimator_p/n_0.40}
\end{table}

\begin{table}[htbp]
\centering
\small
\setlength{\tabcolsep}{4pt}
\renewcommand{\arraystretch}{1.2}
\caption{Bootstrap Variance Estimators vs. Monte Carlo Variances under Setting~\ref{sec:sims glms} (Gaussian
design and dense regression coefficients) in Section~\ref{sec:sims glms}, Based on 500 Monte Carlo Simulations with $n = 10000$, $p/n = 0.1$. Here neither $\bm{\mu}$ nor $\bm{\Sigma}$ is known.}
\begin{tabular}{lccccc}
\toprule
& \textbf{MC Var} & \textbf{Mean Est. Var} & $\frac{\textbf{Mean Est. Var}}{\textbf{MC Var}}$ & \textbf{Std Est. Var} & \textbf{MSE} \\
\midrule
$\mathbb{E} A$ & 2.26e-05 & 2.49e-05 & 1.101 & 1.15e-06 & 2.57e-06 \\
$\mathbb{E} [A \bX^{\top}] \bSigma^{-1} \bmu$ & 5.45e-05 & 5.04e-05 & 0.924 & 8.27e-06 & 9.25e-06 \\
$\mathbb{E} [A \bX^{\top}] \bSigma^{-1} \mathbb{E} [A \bX]$ & 2.64e-05 & 2.42e-05 & 0.915 & 3.73e-06 & 4.35e-06 \\
$\mathbb{E} [A \bX^{\top}] \bSigma^{-1} \mathbb{E} [A \bX]$ & 2.18e-05 & 2.24e-05 & 1.029 & 2.90e-06 & 2.97e-06 \\
$\balpha^{\top} \bmu$ & 3.88e-04 & 4.07e-04 & 1.051 & 1.53e-04 & 1.54e-04 \\
$\balpha^{\top} \bSigma \balpha$ & 9.73e-03 & 1.03e-02 & 1.064 & 2.81e-03 & 2.88e-03 \\
$\alpha_{1}$ & 6.48e-04 & 7.35e-04 & 1.134 & 4.33e-05 & 9.72e-05 \\
$\alpha_{100}$ & 6.17e-04 & 7.29e-04 & 1.182 & 4.88e-05 & 1.22e-04 \\
\bottomrule
\end{tabular}
\label{tab:Bootstrap_variance_estimator_p/n_0.10}
\end{table}

\begin{table}[htbp]
\centering
\small
\setlength{\tabcolsep}{4pt}
\renewcommand{\arraystretch}{1.2}
\caption{Bootstrap Variance Estimators vs. Monte Carlo Variances under Setting~\ref{sec:sims glms} (Gaussian
design and dense regression coefficients) in Section~\ref{sec:sims glms}, Based on 500 Monte Carlo Simulations with $n = 10000$, $p/n = 0.2$. Here neither $\bm{\mu}$ nor $\bm{\Sigma}$ is known.}
\begin{tabular}{lccccc}
\toprule
& \textbf{MC Var} & \textbf{Mean Est. Var} & $\frac{\textbf{Mean Est. Var}}{\textbf{MC Var}}$ & \textbf{Std Est. Var} & \textbf{MSE} \\
\midrule
$\mathbb{E} A$ & 2.57e-05 & 2.50e-05 & 0.970 & 1.14e-06 & 1.38e-06 \\
$\mathbb{E} [A \bX^{\top}] \bSigma^{-1} \bmu$ & 1.43e-04 & 1.34e-04 & 0.936 & 2.16e-05 & 2.35e-05 \\
$\mathbb{E} [A \bX^{\top}] \bSigma^{-1} \mathbb{E} [A \bX]$ & 6.18e-05 & 5.74e-05 & 0.929 & 8.82e-06 & 9.85e-06 \\
$\mathbb{E} [A \bX^{\top}] \bSigma^{-1} \mathbb{E} [A \bX]$ & 5.09e-05 & 4.73e-05 & 0.929 & 6.91e-06 & 7.81e-06 \\
$\balpha^{\top} \bmu$ & 9.02e-04 & 9.67e-04 & 1.071 & 4.12e-04 & 4.17e-04 \\
$\balpha^{\top} \bSigma \balpha$ & 2.34e-02 & 2.37e-02 & 1.013 & 8.84e-03 & 8.84e-03 \\
$\alpha_{1}$ & 1.07e-03 & 1.00e-03 & 0.935 & 7.46e-05 & 1.02e-04 \\
$\alpha_{100}$ & 9.68e-04 & 9.99e-04 & 1.033 & 7.27e-05 & 7.92e-05 \\
\bottomrule
\end{tabular}
\label{tab:Bootstrap_variance_estimator_p/n_0.20}
\end{table}


\newpage
\putbib[\myreferencesapp]

\end{bibunit}
\end{appendices}

\end{document}

%% file: preamble.tex
\usepackage{color, graphicx, appendix}
\usepackage[dvipsnames]{xcolor}
\definecolor{darkblue}{rgb}{0.0, 0.0, 0.55}
\definecolor{codegreen}{rgb}{0.0, 0.5, 0.0}
\usepackage[bookmarks, colorlinks=true, plainpages = false, citecolor=darkblue, linkcolor=OliveGreen, anchorcolor=red, urlcolor=darkblue]{hyperref}
\usepackage{url}
\usepackage{amsmath, amsfonts, amsthm, amssymb, bm, bbm, verbatim, dsfont, mathtools, mathrsfs, empheq}
\usepackage{etoolbox}
\usepackage{almendra}
\usepackage{authblk}
\usepackage{array}
\usepackage{multirow}
\usepackage{ulem}
\usepackage{cancel}
\usepackage{esvect}

\usepackage{float}
\usepackage{pgf,tikz}
\usetikzlibrary{arrows,shapes.arrows,shapes.geometric,shapes.multipart,fit,automata,decorations.pathmorphing,positioning}
\tikzset{
	>=stealth',
	true/.style={
		rectangle,
		draw=black, very thick,
		text width=6.5em,
		minimum height=2em,
		text centered,
		fill=gray, opacity = 0.5},
	punkt/.style={
		rectangle,
		rounded corners,
		draw=black, very thick,
		text width=6.5em,
		minimum height=2em,
		text centered},
	est/.style={
		circle,
		draw=black, very thick,
		text centered},
	shade/.style={
		circle,
		draw=black, very thick, fill=gray!50,
		text centered},
	weight/.style={
		circle,
		draw=black, very thick,
		text width=6.5em,
		minimum height=2em,
		text centered},
	pil/.style={
		->,
		thick,
		shorten <=2pt,
		shorten >=2pt,},
	double/.style={
		<->,
		thick,
		shorten <=2pt,
		shorten >=2pt,},
	dash/.style={
		dashed,
		thick,
		shorten <=2pt,
		shorten >=2pt,},
	dashdouble/.style={
		<->,
		dashed,
		thick,
		shorten <=2pt,
		shorten >=2pt,}
}

\usepackage{tikz-cd}
\makeatletter 
\newcolumntype{C}[1]{>{\centering\arraybackslash}p{#1}}

\usepackage{epstopdf}
\usepackage{fullpage}
\usepackage{algorithm}
\usepackage{algorithmicx}
\usepackage{subcaption}

\usepackage[nameinlink]{cleveref}
\usepackage[round]{natbib}

\usepackage[utf8]{inputenc} 
\usepackage[T1]{fontenc}    
\usepackage{url}            
\usepackage{booktabs}     
\usepackage{nicefrac}       
\usepackage{microtype}      
\usepackage{pdfpages}
\usepackage{scalerel}
\usepackage{dashrule}
\usepackage{lmodern}
\usepackage{fancyhdr}
\usepackage{tabu}
\usepackage{enumitem}
\definecolor{ForestGreen}{RGB}{34,139,34}

\def\var{\mathsf{var}}
\def\cov{\mathsf{cov}}

\def\op{\mathsf{op}}

\renewcommand{\hat}{\widehat}
\renewcommand{\tilde}{\widetilde}

\newcommand{\myreferences}{Master.bib}
\newcommand{\myreferencesapp}{Master.bib}

\usepackage{xr,xspace}
\usepackage{todonotes}

\usepackage{caption,subcaption,soul}
\usepackage{algpseudocode}
\makeatletter
\usepackage{romanbar}

\theoremstyle{plain}
\newtheorem{theorem}{Theorem}
\newtheorem{lemma}{Lemma}
\newtheorem{proposition}{Proposition}
\newtheorem{corollary}{Corollary}

\theoremstyle{definition}
\newtheorem{definition}{Definition}
\newtheorem{example}{Example}

\newtheorem{assumption}{Assumption}

\newtheorem{problem}{Problem}

\newtheorem{remark}{Remark}
\AtBeginEnvironment{remark}{%
  \pushQED{\qed}%
}
\AtEndEnvironment{remark}{\popQED\endexample}

\newtheorem*{remark*}{Remark}

\usepackage{algorithm}
\usepackage{algpseudocode}

\usepackage{xspace, prettyref}

\newcommand{\diff}{{\mathrm d}}

\newcommand\indep{\protect\mathpalette{\protect\independenT}{\perp}}
\def\independenT#1#2{\mathrel{\rlap{$#1#2$}\mkern2mu{#1#2}}}

\newrefformat{eq}{(\ref{#1})}
\newrefformat{chap}{Chapter~\ref{#1}}
\newrefformat{sec}{Section~\ref{#1}}
\newrefformat{alg}{Algorithm~\ref{#1}}
\newrefformat{fig}{Fig.~\ref{#1}}
\newrefformat{tab}{Table~\ref{#1}}
\newrefformat{rmk}{Remark~\ref{#1}}
\newrefformat{clm}{Claim~\ref{#1}}
\newrefformat{def}{Definition~\ref{#1}}
\newrefformat{cor}{Corollary~\ref{#1}}
\newrefformat{lmm}{Lemma~\ref{#1}}
\newrefformat{prop}{Proposition~\ref{#1}}
\newrefformat{prob}{Problem~\ref{#1}}
\newrefformat{app}{Appendix~\ref{#1}}
\newrefformat{hyp}{Hypothesis~\ref{#1}}
\newrefformat{thm}{Theorem~\ref{#1}}

\newcommand{\sfb}{{\mathsf{b}}}

\newcommand{\calA}{{\mathcal{A}}}

\newcommand{\calD}{{\mathcal{D}}}

\newcommand{\calI}{{\mathcal{I}}}

\newcommand{\calL}{{\mathcal{L}}}
\newcommand{\calM}{{\mathcal{M}}}

\newcommand{\calO}{{\mathcal{O}}}

\newcommand{\calR}{{\mathcal{R}}}

\newcommand{\calW}{{\mathcal{W}}}

\renewcommand{\tilde}{\widetilde}
\renewcommand{\hat}{\widehat}

\newcommand{\bmu}{{\bm{u}}}

\newcommand{\bbA}{{\mathbb{A}}}

\newcommand{\bbE}{{\mathbb{E}}}

\newcommand{\bbH}{{\mathbb{H}}}

\newcommand{\bbM}{{\mathbb{M}}}

\newcommand{\bbP}{{\mathbb{P}}}

\newcommand{\bbR}{{\mathbb{R}}}
\newcommand{\bbS}{{\mathbb{S}}}

\newcommand{\bbU}{{\mathbb{U}}}

\newcommand{\bbX}{{\mathbb{X}}}

\usepackage{tcolorbox}

\makeatletter
\def\ubar#1{\underline{\sbox\tw@{$#1$}\dp\tw@\z@\box\tw@}}
\makeatother

\def\leftarrowCirc{\hbox{$\leftarrow$}\kern-1.5pt\hbox{$\circ$}}
\def\Circrightarrow{\hbox{$\circ$}\kern-1.5pt\hbox{$\rightarrow$}}
\def\Circleftarrow{\hbox{$\circ$}\kern-1.5pt\hbox{$\leftarrow$}}
\def\rightarrowCirc{\hbox{$\rightarrow$}\kern-1.5pt\hbox{$\circ$}}

\def\expit{\mathrm{expit}}

\def\N{\mathrm{N}}

\def\f{\mathrm{f}}
\def\g{\mathrm{g}}

\def\bI{\mathbf{I}}

\def\bM{\mathbf{M}}

\def\bV{\mathbf{V}}

\def\bX{\mathbf{X}}
\def\bY{\mathbf{Y}}
\def\bZ{\mathbf{Z}}

\def\be{\mathbf{e}}

\def\bu{\mathbf{u}}
\def\bv{\mathbf{v}}

\def\bx{\mathbf{x}}

\def\eps{\epsilon}

\def\bbeta{\bm{\beta}}

\def\bmu{\bm{\mu}}

\def\bXi{\bm{\Xi}}
\def\balpha{\bm{\alpha}}

\def\btheta{\bm{\theta}}

\def\bOmega{\bm{\Omega}}

\def\bSigma{\bm{\Sigma}}

\def\prox{\mathsf{prox}}

\newcommand\ind{\mathrm{ind}}